\documentclass[10pt]{amsart}
\usepackage[usenames,dvipsnames]{xcolor}
\usepackage{latexsym, amssymb, amscd, amsfonts,pstricks,enumitem,amsmath}
\usepackage{tikz-cd}
\usetikzlibrary{matrix,arrows}
\usepackage{graphicx}
\usepackage[all]{xypic}
\usepackage[normalem]{ulem}
\usepackage{todonotes,marginnote}

\renewcommand{\subsection}[1]{\vspace{3mm}\refstepcounter{subsection}\noindent{\bf \thesubsection. #1.} }

\newcommand{\np}[1]{\vspace{3mm}\refstepcounter{subsection}\noindent{\bf \thesubsection.} }

\renewcommand{\subsubsection}[1]{\vspace{3mm}\refstepcounter{subsubsection}\noindent{\bf \thesubsubsection. #1.} }

\newcommand{\trdeg}{\operatorname{trdeg}}
\newcommand{\DD}{\mathbb{D}}

\newcommand{\rad}{\operatorname{rad}}

\newcommand{\la}{\langle}
\newcommand{\ra}{\rangle}

\newcommand{\Hom}{\operatorname{Hom}}
\renewcommand{\geq}{\geqslant}
\renewcommand{\leq}{\leqslant}
\newcommand{\Osh}{{\mathcal O}}                        

\renewcommand{\H}{\mathrm{H}}                          
                            
\newcommand{\lcm}{\operatorname{lcm}}

\newcommand{\WDiv}{\operatorname{WDiv}}
\newcommand{\DDiv}{\mathbb{D}\mathrm{iv}}
\newcommand{\Div}{\operatorname{CDiv}}
\newcommand{\Disc}{\operatorname{Disc}}

\newcommand{\idlim}{\varprojlim} 
                        
\newcommand{\K}{\mathrm{K}}                            
\newcommand{\rank}{\operatorname{rank}}                
\newcommand{\cchar}{\operatorname{char}}
\newcommand{\Spec}{\operatorname{Spec}}

\newcommand{\GKdim}{\operatorname{GKdim}}

\newcommand{\Gal}{\operatorname{Gal}}

\newcommand{\G}{\mathrm{G}}
\newcommand{\Br}{\operatorname{Br}}

\newcommand{\kk}{\mathbf{k}}

\newcommand{\Trd}{\operatorname{Trd}}

\newcommand{\N}{\operatorname{N}}

\newcommand{\spec}{\operatorname{Spec}}

\newcommand{\KK}{\mathbf{K}}
\newcommand{\FF}{\mathbf{F}}
\newcommand{\PP}{\mathbb{P}} 
\newcommand{\QQ}{\mathbb{Q}} 
\newcommand{\RR}{\mathbb{R}} 
\newcommand{\ZZ}{\mathbb{Z}} 

\newtheorem{theorem}{Theorem}[section]
\newtheorem{lemma}[theorem]{Lemma}
\newtheorem{corollary}[theorem]{Corollary}
\newtheorem{proposition}[theorem]{Proposition}

\theoremstyle{definition}
\newtheorem{defn}[theorem]{Definition}
\newtheorem{defns}[theorem]{Definitions}
\newtheorem{remark}[theorem]{Remark}
\newtheorem{example}[theorem]{Example}

\newtheorem{question}[theorem]{Question}
\newtheorem{conjecture}[theorem]{Conjecture}

\begin{document}
\title{On the Kodaira dimension of maximal orders}
\author{Nathan Grieve}
{ 

\address{
School of Mathematics and Statistics, 4302 Herzberg Laboratories, Carleton University, 1125 Colonel By Drive, Ottawa, ON, K1S 5B6, Canada
}
\address{D\'{e}partement de math\'{e}matiques, Universit\'{e} du Qu\'{e}bec \'a Montr\'{e}al, Local PK-5151, 201 Avenue du Pr\'{e}sident-Kennedy, Montr\'{e}al, QC, H2X 3Y7, Canada}
\email{nathan.m.grieve@gmail.com}%
}

\author{Colin Ingalls}
\address{
School of Mathematics and Statistics, 4302 Herzberg Laboratories, Carleton University, 1125 Colonel By Drive, Ottawa, ON, K1S 5B6, Canada
}
\email{cingalls@math.carleton.ca}%

\begin{abstract}
  Let $\kk$ be an algebraically closed field of characteristic zero and $\KK$ a finitely generated field over $\kk$.  Let $\Sigma$ be a central simple $\KK$-algebra, $X$ a normal projective model of $\KK$ and $\Lambda$ a sheaf of maximal $\Osh_X$-orders in $\Sigma$.  There is a ramification $\QQ$-divisor $\Delta$ on $X$, which is related to the canonical bimodule $\omega_\Lambda$ by an adjunction formula. It only depends on the class of $\Sigma$ in the Brauer group of $\KK$.  When the numerical abundance conjecture holds true, or when $\Sigma$ is a central simple algebra, we show that the Gelfand-Kirillov dimension (or GK dimension) of the canonical ring of $\Lambda$ is one more than the Iitaka dimension (or D-dimension) of the log pair $(X,\Delta)$.   In the case that $\Sigma$ is a division algebra, we further show that this GK dimension is also one more than the transcendence degree of the division algebra of degree zero fractions of the canonical ring of $\Lambda$.  We prove that these dimensions are birationally invariant when the b-log pair determined by the ramification divisor has b-canonical singularities.   In that case we refer to the  Iitaka (or  D-dimension) of $(X,\Delta)$ as the Kodaira dimension of the order $\Lambda$.  For this, we establish birational invariance of the Kodaira dimension of b-log pairs with b-canonical singularities. We also show that the Kodaira dimension can not decrease for an embedding of  central simple algebras, finite dimensional over their centres, which induces a Galois extension of their centres, and satisfies a divisibility condition on the ramification which we call an effective embedding.  For example, this condition holds if the target  central simple algebra has the property that its period equals its index.
  In proving our main result, we establish existence of equivariant b-terminal resolutions of $\G$-b-log pairs and we also find two variants of the Riemann-Hurwitz formula.  The first variant applies to  effective embeddings of central simple algebras with fixed centres while the second applies to the pullback of a central simple algebra by a Galois extension of its centre. We also give two different local characterizations of effective embeddings.  The first is in terms of complete local invariants, while the second uses Galois cohomology.

\end{abstract}

\maketitle
\tableofcontents
\section{Introduction}
Let $\kk$ be an algebraically closed field of characteristic zero and let $\KK$ be a finitely generated field over $\kk$.  By birational invariance of Kodaira dimension for proper models of $\KK$, define the Kodaira dimension of $\KK$ to be $\kappa(\KK):=\kappa(X)$, the Kodaira dimension of some  smooth proper model $X$ of $\KK$.  The Kodaira dimension of such an $X$ can be defined in terms of its canonical ring 
$$R(X,\K_X) := \bigoplus_{\ell \geq 0} \H^0(X,\Osh_X(\ell \K_X)) =  \bigoplus_{\ell \geq 0} \Hom_{\Osh_X}(\Osh_X,\omega_X^{\otimes \ell})$$
in various ways.
For instance, let $\GKdim R(X,\K_X)$ denote the GK dimension of $R(X,\K_X)$.  Then: 
\begin{equation}\label{GK:Kod:defn}
 \kappa(X):= 
\GKdim R(X,\K_X) -1. 
\end{equation}

Our convention is that $\kappa(X)=-1$ when negative rather than $-\infty$.
In this article,  we wish to extend the concept of Kodaira dimension for $\KK$ to a central simple $\KK$-algebra $\Sigma$.  In place of the model $X$, we choose a sheaf of maximal $\Osh_X$-orders $\Lambda$ in $\Sigma$.  In this situation, there is a canonical bimodule $\omega_\Lambda$ and, by using tensor products over $\Lambda$ and reflexive hulls $(-)^{\vee\vee}$, we define the canonical ring of $\Lambda$ to be: 
$$
R(\Lambda,\omega_\Lambda) := \bigoplus_{\ell \geq 0} \Hom_{\Lambda}(\Lambda, (\omega_\Lambda^{\otimes \ell })^{\vee \vee}),
$$
see Definition~\ref{defn:cansheafOfOrder} and Definition~\ref{can:ring:defn:1} for more details.  In the spirit of \eqref{GK:Kod:defn}, we define
\begin{equation*}\label{GK:Kod:defn'}
 \kappa(\Lambda):= 
\GKdim R(\Lambda,\omega_\Lambda) -1. 
\end{equation*}
At this point, we refer to $\kappa(\Lambda)$ as the \emph{Iitaka dimension} of $\Lambda$ and we reserve the designation of \emph{Kodaira dimension} for a birational invariant of $\Lambda$.  
To understand the Iitaka dimension of $\Lambda$, and ultimately the Kodaira dimension of $\Sigma$, we relate the canonical ring $R(\Lambda,\omega_\Lambda)$ to geometric information on its centre.  To that end, we need an observation of M.~Artin.  Then, as in Corollary~\ref{canonical:ring:order:cor4}, there is the ramification divisor $\Delta_\Lambda$, which is a $\QQ$-divisor on $X$ so that
\begin{equation}\label{Adjunction:intro:defn1}
(\omega_\Lambda^{\otimes n})^{\vee\vee} \simeq (\Osh_X(n(\K_X+\Delta_{\Lambda})) \otimes \Lambda )^{\vee\vee}.
\end{equation}  
Here, $n^2 = [\Sigma : \KK]$.  Motivated by \eqref{Adjunction:intro:defn1}, we set $\K_\Lambda := \K_X+\Delta_\Lambda$ and consider the log canonical ring
$$R(X,\K_\Lambda) := \bigoplus_{\ell\geq 0} \H^0(X, \Osh_X( \lfloor \ell(\K_X+  \Delta_{\Lambda}) \rfloor) ).$$
Further, if we assume that $(X,\Delta_{\Lambda})$ is $\QQ$-Gorenstein, then the $n$th Veronese subrings of the canonical ring $R(\Lambda,\omega_\Lambda)$
and the log canonical ring $R(X,\K_\Lambda)$ are, respectively:
$$R(\Lambda,\omega_\Lambda)^{(n)} := \bigoplus_{\ell \geq 0} R(\Lambda,\omega_\Lambda)_{\ell n} = \bigoplus_{\ell \geq 0}\H^0(X,\Lambda(\ell n (\K_X+\Delta_{\Lambda})))
$$
and
$$R(X,\K_\Lambda)^{(n)} := \bigoplus_{\ell\geq 0} \H^0(X, \Osh_X( \ell n(\K_X+  \Delta_{\Lambda}) ) ). 
$$
So $R(X,\K_\Lambda)^{(n)}$ is also a central subalgebra of $R(\Lambda,\omega_\Lambda)^{(n)}$.

We can obtain an equivalence of GK dimensions if we assume the numerical abundance conjecture from birational geometry.  First of all, recall the concept of numerical dimension \cite{Fujita81}.  If $D$ is a $\QQ$-Cartier divisor on $X$, then its \emph{numerical dimension} is defined to be
$$
\kappa_\sigma(D) := \max \{\sigma(D; \Osh_X(A)):  \text{$A$ is an ample divisor on $X$} \}.
$$
Here, for a reflexive sheaf $F$ on $X$, we put: 
$$
\sigma(D;F) := \max \left\{ k \in \ZZ_{\geq 0} : \limsup_{m \to \infty} \frac{h^0(X,F \otimes \Osh_X(\lfloor m D \rfloor)) }{m^k} > 0 \right\}
$$
(By convention, $\sigma(D;F) = - 1$ in case that $h^0(X,F\otimes\Osh_X(\lfloor m D \rfloor)) = 0$ for all sufficiently divisible $m \gg 0$.)

\begin{conjecture}[{Numerical Abundance Conjecture \cite[Conjecture 1.1]{Gongyo:2013}}]\label{perturbed:growth:abundance:conj}
Let $(X,\Delta)$ be a projective Kawamata log terminal (klt) pair, where $\K_X + \Delta$ is $\QQ$-Cartier.  Then the numerical and Iitaka dimensions of the log canonical divisor coincide:
\begin{equation}\label{num:abundance:growth:eqn}
\kappa_{\sigma}(X,\K_X + \Delta) = \kappa(X,\K_X + \Delta).
\end{equation}
\end{conjecture}

We use Conjecture~\ref{perturbed:growth:abundance:conj} to prove the following theorem in \S \ref{perturbed:growth}. 

\begin{theorem}\label{log:abundance:growth}
Let $\KK$ be a finitely generated field over $\kk$ and let $\Sigma$ be a central simple $\KK$-algebra.  Let $X$ be a normal projective model of $\KK$ and let $\Lambda$ be a maximal $\Osh_X$-order in $\Sigma$.  Let $\Delta_\Lambda$ be the ramification divisor of $\Lambda$ and assume that the log pair $(X,\Delta_\Lambda)$ is klt, $\QQ$-Gorenstein, and satisfies the numerical abundance conjecture.
Then
$$\GKdim R(\Lambda,\omega_\Lambda) = \GKdim R(X,\K_\Lambda).$$
\end{theorem}

At this point we still need birational invariance to define the Kodaira dimension $\kappa(\Sigma)$.  Note that the ramification divisor of $\Sigma$ is defined for all proper normal models of $\KK$ and so $\Sigma$ determines a b-divisor on $\KK$.  This naturally leads us to the notion of b-singularities from~\cite{Chan:plus:10} and the existence of b-terminal resolutions~\cite[Theorem 2.30]{Chan:plus:10}.  We obtain the following theorem which is an alternative formulation of  Theorem~\ref{b-divisor-theorem}.

\begin{theorem}\label{brauer:iitakacanonical:singulatiries:main:result}
Let $\KK$ be a field finitely generated over $\kk$.  Let $\DD$ be a b-divisor on $\KK$ with coefficients in $[0,1)\cap\QQ$.  Let $X$ and $Y$ be proper normal models of $\KK$ such that the b-log pairs
$(X,\DD)$ and $(Y,\DD)$ have b-canonical singularities.
If $\ell(\K_X+\DD_X)$ and $\ell(\K_Y+\DD_Y)$ are Cartier, for some positive integer $\ell$, then
$$ R(X,\K_X+\DD_X)^{(\ell)} = R(Y,\K_Y+\DD_Y)^{(\ell)}$$
and
$$\GKdim R(X,\K_X+\DD_X) = \GKdim R(Y,\K_Y+\DD_Y).$$
\end{theorem}

Theorem \ref{brauer:iitakacanonical:singulatiries:main:result}, combined with the existence of b-terminal resolutions for b-log pairs~\cite[Theorem 2.31]{Chan:plus:10}, allows us to define the Kodaira dimension of a b-divisor $\DD$ as $\kappa(\DD):=\kappa(X,\K_X+\DD_X)$ where $X$ is any normal proper model of $\KK$ with $(X,\DD)$ having b-canonical singularities.  This also allows us to formulate one of our main results, which includes birational and Morita invariance of Kodaira dimension for maximal orders.

\begin{theorem}\label{kodOfcsa}
Let $\KK$ be a field finitely generated over $\kk$.  Let $\alpha$ be a Brauer class in $\Br(\KK)$ with ramification b-divisor $\DD(\alpha)$.  Suppose that the numerical abundance conjecture holds for all $(X,\DD(\alpha))$ such that $X$ is a normal projective model of $\KK$ and $(X,\DD(\alpha))$ has b-canonical singularities.  Then if we choose  a central simple algebra $\Sigma$ with Brauer class $\alpha$, a maximal $\Osh_X$-order $\Lambda$ in $\Sigma$ with the b-log pair $(X,\DD(\alpha))$ having b-canonical singularities then $\kappa(\Lambda)$
depends only on $\alpha$.
\end{theorem}

\begin{proof}[Proof of Theorem~\ref{kodOfcsa}] If $(X,\DD(\alpha))$ is b-canonical then it is b-lt and so $(X,\DD(\alpha)_X)$ is klt by~\cite[Corollary 2.23]{Chan:plus:10}.  So we can 
combine Theorem~\ref{log:abundance:growth} and Theorem~\ref{brauer:iitakacanonical:singulatiries:main:result} to obtain the result.
\end{proof}

In what follows, in the setting of Theorem~\ref{kodOfcsa} and under the hypothesis that $(X,\DD(\alpha))$ has b-canonical singularities, we set: $ \kappa(\alpha):=\kappa(\Sigma) :=\kappa(\Lambda)$ and we then have birational invariance of Kodaira dimension for maximal orders.   

There are easy examples (compare with Example~\ref{notPrimeExample}) where $R(\Lambda,\omega_\Lambda)$ is not prime, so we can not form a degree zero ring of fractions in general.  However, when $\Sigma$ is a division ring, we can form the degree zero graded fractions of $R(\Lambda,\omega_\Lambda)$ and we then give an alternative description for the growth of $R(\Lambda,\omega_\Lambda)$.  In addition we can relate the Kodaira dimension of $\Lambda$ to the transcendence degree of this division algebra without using the abundance conjecture.  This is the content of Theorem~\ref{Ld:maximal:order:section:ring:thm:intro} which is a special case of Theorem~\ref{Ld:maximal:order:section:ring:thm}; we prove these theorems in \S \ref{division:ring:function:field:centre}.  

\begin{theorem}\label{Ld:maximal:order:section:ring:thm:intro} 
Let $\KK$ be a field finitely generated over $\kk$.  Fix a $\KK$-central division algebra $\mathcal{D}$, with finite dimension over $\KK$, with Brauer class $\alpha \in \Br (\KK)$ and ramification b-divisor $\DD(\alpha)$.  Let $X$ be a proper normal model of $\KK$,  and write $\K_\alpha:=\K_X+\DD(\alpha)_X$.  Choose a maximal $\Osh_X$-order $\Lambda$ in $\mathcal{D}$ and suppose that $(X,\DD(\alpha)_X)$ is $\QQ$-Gorenstein.  Let $\kk(\Lambda, \K_\alpha)$ be the degree zero division ring of fractions of $R(\Lambda,\omega_\Lambda)$ and let $\kk(X,\K_\alpha)$ be the degree zero field of fractions of the log canonical ring $R(X,\K_\alpha)$.  Then 
$$\GKdim R(\Lambda,\omega_\Lambda) -1 =
\trdeg \kk(\Lambda,\K_\alpha)
= 
\trdeg \kk(X,\K_\alpha) = \GKdim R(X,\K_\alpha)-1.$$
If, in addition, $(X,\DD(\alpha))$ has b-canonical singularities, then these numbers depend only on $\alpha$. 
 \end{theorem}

Note that Theorem~\ref{Ld:maximal:order:section:ring:thm:intro} provides a proof of an unsubstantiated assertion of the second author in~\cite{BIR}, before~\cite[Question 1, p.~427]{BIR}. At the same time, it partially answers~\cite[Question 1, p.~427]{BIR}.  
  
Finally, we turn to the main result of this paper (Theorem \ref{Division:Alg:Galois:Embedding:Thm}). As a special case, see  Theorem~\ref{Division:Alg:Galois:Embedding:Thm:Intro} below.  To place this result into perspective, we first mention that its statement relies on the birationally invariant concept of Kodaira dimension for Brauer classes which we obtain here.  It also requires some more subtle conditions on the ramification.  We make these conditions precise in Definitions \ref{effective:embedding:intro}  and \ref{weak:ramification:intro}.

 \begin{defn}\label{effective:embedding:intro} Let $\Sigma_1 \subseteq \Sigma_2$ be central simple algebras with centres $\KK$ and Brauer classes $\alpha_1,\alpha_2 \in \Br(\KK).$   Let $\mathbb{D}(\alpha_1)$ and $\mathbb{D}(\alpha_2)$ be their respective birational ramification divisors. Let $X$ be a normal model of $\KK$.  We write $\Sigma_1 \leq_X \Sigma_2$  and say that the given embedding is   
   {\it $X$-effective} if $\mathbb{D}(\alpha_2)_X - \mathbb{D}(\alpha_1)_X \geq 0$.
   If we have the stronger condition that the ramification index of $\alpha_1$ divides that of $\alpha_2$ at every divisorial place of $X$, we write
   $\Sigma_1 \mid_X \Sigma_2$.  If these conditions hold  for all normal models $Y$ of $\KK$, or equivalently for all divisorial places of $\KK$, we write
   $\Sigma_1 \leq_b \Sigma_2$  (respectively $\Sigma_1 \mid_b \Sigma_2$).  In case that $\Sigma_1 |_X \Sigma_2$, respectively $\Sigma_1 |_b \Sigma_2$, we refer to such phenomenon as the property that the given embedding has $X$-divisible ramification, respectively $b$-divisible ramification.    
 \end{defn}

   While the concept of effective embedding is not a partial order on central simple algebras, there are the following two evident implications:
   \begin{enumerate}
   \item{  $\Sigma_1 \mid_X \Sigma_2 \implies \Sigma_1 \leq_X \Sigma_2$;}
   \item{ $\Sigma_1 \leq_X \Sigma_2$ and $\Sigma_2 \leq_X \Sigma_1 \implies
     \mathbb{D}(\alpha_1)_X = \mathbb{D}(\alpha_2)_X$\text{.}
     }
   \end{enumerate}
Similar conclusions hold true for the relations $\leq_b$ and  $\mid_b$.

If $\Sigma_1 \leq_b \Sigma_2$ is a b-effective embedding, then the difference of their ramification b-divisors is effective.
In order to formulate Theorem \ref{Division:Alg:Galois:Embedding:Thm:Intro} below, we need one other concept, see Definition \ref{weak:ramification:intro}, namely that of weak ramification with respect to a Galois extension together with its birational counter part.  
 
This condition of $b$-effective embedding (Definition \ref{effective:embedding:intro}) is ensured when   $\operatorname{period}(\Sigma_2) = \operatorname{index}(\Sigma_2)$, so, in particular, it holds 
if $\Sigma_2$ is a cyclic division algebra \cite[p.~37, before Theorem 2.5.7]{Gille:Szamuley:2006}, or when the transcendence degree of $\KK$ is two over an algebraically closed base field~\cite[p.~72]{deJong:2004}.
Proposition \ref{perinduni} below is a global version of  Proposition~\ref{period:index}.
  
\begin{proposition}\label{perinduni}  Let $\Sigma_2$ be a central simple algebra with centre $\KK$.  If $\operatorname{period}(\Sigma_2) = \operatorname{index}(\Sigma_2)$, then every central simple  subalgebra $\Sigma_1$ with centre $\KK$ is $b$-effectively embedded in $\Sigma_2$.
 \end{proposition}

Our main results in the direction of Kodaira dimensions, for embeddings of central simple algebras, are the subject of Theorems \ref{Division:Alg:Galois:Embedding:Thm} and \ref{Division:Alg:Galois:Embedding:Thm:Intro}.
    
\begin{theorem}[See also Theorem \ref{Division:Alg:Galois:Embedding:Thm}]\label{Division:Alg:Galois:Embedding:Thm:Intro}
Suppose that $\kk$ is an algebraically closed field of characteristic zero.  Let  $\Sigma_1 \subseteq \Sigma_2$ be central simple algebras, finite dimensional over their centres $\KK, \FF$ which are finitely generated over $\kk$.  Suppose further that:
\begin{enumerate}
\item{ $\KK \subseteq \FF$;}
\item{ the extension $\FF / \KK$ is Galois with degree relatively prime to the index of $\Sigma_1$; and 
}
\item{ $\FF \otimes_{\KK} \Sigma_1 \leq_b \Sigma_2$.
}
\end{enumerate}
Then
$$ \kappa(\Sigma_1) \leq \kappa(\Sigma_2).$$
\end{theorem}

Together, 
Proposition \ref{perinduni} and Theorem \ref{Division:Alg:Galois:Embedding:Thm:Intro} have the following consequence.

\begin{corollary}[See also Corollary \ref{rel-prime-period-index}]\label{rel-prime-period-index-intro}
Suppose that $\kk$ is an algebraically closed field of characteristic zero.  Let  $\Sigma_1 \subseteq \Sigma_2$ be central simple algebras, finite dimensional over their centres $\KK, \FF$ which are finitely generated over $\kk$.  Suppose further that:
\begin{enumerate}
\item{ $\KK \subseteq \FF$;}
\item{ the extension $\FF / \KK$ is Galois with degree relatively prime to the index of  $\Sigma_1$; and}
\item{ $\operatorname{period}(\Sigma_2) = \operatorname{index}(\Sigma_2)$.
}
\end{enumerate}
Then
$$ \kappa(\Sigma_1) \leq \kappa(\Sigma_2).$$
\end{corollary}

Although Theorem~\ref{Division:Alg:Galois:Embedding:Thm:Intro} provides a partial answer to~\cite[Question 2, p.~427]{BIR}, the following questions remain:
\begin{question}
Does the conclusion of Theorem~\ref{Division:Alg:Galois:Embedding:Thm:Intro} hold  without the hypothesis that the extension $\FF / \KK$ is Galois?
  \end{question}
\begin{question}
Does the conclusion of Theorem~\ref{Division:Alg:Galois:Embedding:Thm:Intro} hold  without the hypothesis that the embedding is effective?
  \end{question}
We now discuss some aspects to the proof of Theorem~\ref{Division:Alg:Galois:Embedding:Thm:Intro}.  In doing so, we  describe some additional results of this paper in some detail.  Precisely, to prove Theorem~\ref{Division:Alg:Galois:Embedding:Thm:Intro} we break the problem into two pieces.  First, we study the nature of ramification under embeddings of  central simple algebras over a fixed centre.  Second, we extend this result and compare the ramification of  a given central simple algebra $\Sigma$ to that of a  central simple algebra which is a  tensor product $\FF \otimes_{Z(\mathcal{D})} \Sigma$, for $\FF$ a Galois extension of the centre of  $\Sigma$.
In both cases we obtain the desired bound on Kodaira dimension as a consequence of a formula akin to the Riemann-Hurwitz formula.   For completeness we state these Riemann-Hurwitz type theorems as { Proposition}~\ref{embedding:main:result1} and Theorem~\ref{R-H-orders:Kod:cor1:main:result} below. 

\begin{proposition}[See also Theorem \ref{embedding:iitaka:theorem}]\label{embedding:main:result1}  
Let $\KK$ be a finitely generated field over an algebraically closed  characteristic zero field $\kk$.
Fix two finite dimensional $\KK$-central simple algebras  $\Sigma_1$ and $\Sigma_2$, with an  
$X$-effective embedding  $\Sigma_1 \hookrightarrow \Sigma_2$ of $\KK$-algebras, where $X$ is a normal proper model of $\KK$.  Let $\alpha_1$ and $\alpha_2$ denote the respective Brauer classes of $\Sigma_1$ and $\Sigma_2$.   Let $\Delta_{\alpha_1}$ and $\Delta_{\alpha_2}$ denote the Weil divisors on $X$ determined by the ramification of $\alpha_1$ and $\alpha_2$.   Finally, fix a canonical divisor $\K_X$ on $X$.
Then,
in the above setting, we have the following inequality of Weil divisors:
$$\Delta_{\alpha_1} \leq \Delta_{\alpha_2}. $$
In addition, let  
$$\K_{X,\alpha_i} = \K_{\alpha_i} := \K_X + \Delta_{\alpha_i},$$ 
for $i=1,2$, denote the log canonical divisors determined by the classes $\alpha_i \in \Br(\KK)$, let $e_{\mathfrak{p}}(\alpha_1)$ denote the ramification of $\alpha_1$ at $\mathfrak{p}$ and let $e_{\alpha_2/\alpha_1}(\mathfrak{p})$ denote the ramification of the embedding $\Sigma_1 \hookrightarrow \Sigma_2$ at $\mathfrak{p}$.
Then, with these notations, if $\Sigma_1 \mid_X \Sigma_2$ then it holds true that:
$$
\K_{\alpha_2} = \K_{\alpha_1} + \sum\limits_{\mathfrak{p}}\frac{1}{e_{\mathfrak{p}}(\alpha_1)}\left(1 - \frac{1}{e_{\alpha_2/\alpha_1}(\mathfrak{p})} \right)\mathfrak{p}.
$$

If the log canonical divisors $\K_{\alpha_i}$, for $i = 1,2$, are $\QQ$-Cartier,
and $\Sigma_1 \leq_X \Sigma_2$
then the inequality 
$$\kappa(X,\alpha_1) \leq \kappa(X,\alpha_2) $$
holds true. 

\end{proposition}

Theorem~\ref{R-H-orders:Kod:cor1:main:result} below  establishes a Riemann-Hurwitz type theorem for a given $\QQ$-Gorenstein pair $(X,\Delta_\alpha)$, with $\alpha \in \Br(\KK)$.  In order to state its hypothesis, we introduce the following concept   which implicitly arose  in the statement of Theorem \ref{Division:Alg:Galois:Embedding:Thm:Intro}.

\begin{defn}[See also Definitions \ref{weak:ramification} and \ref{G:weak:ramification:2}]\label{weak:ramification:intro}
Let $X$ be a normal projective variety with function field $\KK = \kk(X)$.  Let  $\FF / \KK$ be a finite dimensional extension field and let $f \colon X' \rightarrow X$ be the normalization of $X$ in $\FF$.  We say that $\alpha \in \Br(\KK)$ is \emph{$f$-weakly ramified} at all primes $\mathfrak{p} \in \operatorname{Supp}(\Delta_{\alpha})$ if for all primes $\mathfrak{p}' \in X'$ with $\mathfrak{p}' | \mathfrak{p}$
either 
$e_{\mathfrak{p}' / \mathfrak{p}}(\FF / \KK) >  1$ 
or
$e_{\mathfrak{p}'}(\alpha') = e_{\mathfrak{p}}(\alpha) \text{.}$  
Here, $e_{\mathfrak{p}'/\mathfrak{p}}(\FF / \KK)$ is the ramification index of $\FF / \KK$ at $\mathfrak{p}' | \mathfrak{p}$ whereas $e_{\mathfrak{p}'}(\alpha')$ denotes the ramification of $\alpha' := \FF \otimes_{\KK} \alpha$ at $\mathfrak{p}'$.  
 \end{defn}

In particular, keeping with the context of Definition \ref{weak:ramification:intro}, when $\FF / \KK$ is assumed to be a Galois extension, then this concept of weak ramification indeed holds true provided that the degree of $\FF / \KK$ is relatively prime to the index of the given Brauer class $\alpha \in \Br(\KK)$.  We refer to Proposition \ref{weak:ramification:prop} for more details in that regard.

\begin{theorem}[See also Corollary~\ref{R-H-orders:Kod:cor1} and Proposition \ref{brauer-r-h}]\label{R-H-orders:Kod:cor1:main:result}  
Let $X$ be a normal projective variety, over an algebraically closed characteristic zero field $\kk$, with function field $\KK = \kk(X)$.  Let $\FF / \KK$ be a finite field extension and let $f : X' \rightarrow X$ be the normalization of $X$ in $\FF$.  Finally, fix a Brauer class $\alpha \in \Br(\KK)$
and put $\alpha' = f^*\alpha \in \Br(\FF)$.  
Then, in this situation, the divisor $\K_{f^* \alpha} - f^* \K_\alpha$ is effective
if and only if $\alpha$ is $f$-weakly ramified, in the sense of Definition \ref{weak:ramification}, at all primes $\mathfrak{p} \in \operatorname{Supp}(\Delta_{\alpha})$.
 Assume
further that the classes $\alpha$ and $\alpha'$ determine $\QQ$-Gorenstein pairs $(X,\Delta_{\alpha})$ and $(X',\Delta_{\alpha'})$.  Then, in this situation, it holds true that:
$$ \kappa(X',\K_{f^*\alpha}) \geq \kappa(X,\K_{\alpha}).$$
\end{theorem}

An important aspect in proving Theorem~\ref{R-H-orders:Kod:cor1:main:result} is to study the condition that the difference of the log-canonical divisor on $X'$ determined by the class $\alpha' \in \Br(\FF)$ and the pull-back of the log canonical divisor on $X$ determined by the class $\alpha \in \Br(\KK)$ is effective.  This is the content of Proposition \ref{brauer-r-h}.  A key point is the concept of weak ramification (see Definition \ref{weak:ramification}).

In \S \ref{ramification:embeddings}, we establish  auxiliary results which are related to   Proposition  \ref{embedding:main:result1}.  They pertain to the behaviour of ramification under embeddings of central simple algebras.     The aim is to address the question of divisibility of ramification indices.  For example, Proposition \ref{ramification:division} characterizes this property in terms of complete local numerical invariants of central simple algebras and maximal orders.  We believe that this result is of an independent interest.  We also show, by example, that this property does not hold in general.  Finally, we prove that it does hold true when the target central simplealgebra has the property that its period equals its index.  We are indebted to an anonymous referee for helpful comments on this topic.  Finally, we give a cohomological interpretation of Proposition \ref{ramification:division} in \S \ref{Galois:divisible:ramifcation}. 

Recall that we use the term \emph{Kodaira dimension} to emphasize birational invariance and we use the term \emph{Iitaka dimension} to emphasize dependence on a particular choice of model.

Note that many of our results on canonical rings and log canonical rings hold in situations where we do not have birational invariance and so apply to the Iitaka dimension.
We note that the growth of canonical and log canonical rings are important invariants of canonical divisors and log-pairs on projective varieties.  We also recall that, in the literature, these invariants are often referred to as Iitaka, Kodaira, or $\mathrm{D}$-dimensions, see~\cite{Iitaka:1971},~\cite{Laz}, and~\cite{Mori:1985} for instance.
To place our results obtained here into their proper context, we mention that one source of motivation for what we do here is the joint work~\cite{Chan:Ingalls:2005} of the second author and Daniel Chan.  

In that work, the authors consider Brauer classes arising from function fields of algebraic surfaces.  Using the Artin-Mumford sequence,~\cite{Artin:Mumford:1971}, the authors define ramification data and log pairs determined by such Brauer classes.  Further, they establish a minimal model program for such pairs in dimension two.  This has been extended to higher dimensions in~\cite{Chan:plus:10}, where b-divisors are applied to the birational geometry of orders.  

We should also mention that the idea of using log pairs, arising from ramification data, to study classification questions for maximal orders is not new (\cite{Chan:Kulkarni2003}, \cite{Chan:Kulkarni:2005}, \cite{Chan:Ingalls:2005}). 

Finally we mention that, by \cite[Corollary 1.1.2, p.~408]{Birkar:Cascini:Hacon:McKernan}, if $(X,\Delta)$ is a projective klt $\QQ$-Gorenstein pair $(X,\Delta)$, as in Conjecture \ref{perturbed:growth:abundance:conj}, then the log canonical ring $R(X,\K_X + \Delta)$ is a finitely generated $\kk$-algebra. So we ask the following question.

\begin{question}
Let $X$ be a normal projective variety and let $\Lambda$ be a maximal $\Osh_X$-order in a central simple $\kk(X)$-algebra with ramification b-divisor $\DD$.  Are there natural conditions on the pair $(X,\DD)$ so that $R(\Lambda,\omega_\Lambda)$ is a finitely generated $\kk$-algebra?
\end{question}

\subsection{Acknowledgements} This work began while the first author was supported by an AARMS postdoctoral fellowship at the University of New Brunswick, NB.   Portions of this work completed while he was a postdoctoral fellow at Michigan State University. Both authors hold NSERC grants (RGPIN-2017-04623, DGECR-2021-00218 and RGPIN-2021- 03821).  They thank the Banff International Research Station for its hospitality during November 2019. The second author would like to thank Daniel Chan for useful conversations.    Both authors thank Rajesh Kulkarni for providing comments related to an early version of this work.  They also thank Adam Logan for help with Lemma~\ref{p:e1:not:divide:e3}.   We also thank colleagues for their interest and comments.  
Finally, we are indebted to an anonymous referee for his or her 
careful reading, for catching a gap in an earlier draft of this paper, and for suggestions that led to Propositions~\ref{perinduni} and \ref{weak:ramification:prop}. 

\subsection{Notations and conventions}
Unless otherwise stated, all varieties are assumed to be integral, normal, proper and defined over an algebraically closed characteristic zero field $\kk$.  Further, all division algebras and central simple algebras are assumed to be finitely generated over some given base field, finite dimensional over their centres and have degree prime to the characteristic of their given base field should this characteristic be positive.

\section{Discrete valuation rings, maximal orders and ramifcation}

We summarize, following \cite{Reiner:2003}, \cite{Serre:local:fields}, \cite{Artin:Chan:deJong:Lieblich} and \cite{Chan:Orders},   
 the theory of maximal orders over discrete valuation rings.  
  
\subsection{Discrete  valuation rings}  Let $R = (R,\mathfrak{m},t)$ be a discrete valuation ring.  Here, $\mathfrak{m} = \mathfrak{m}_R$ is the maximal ideal and $t = t_R \in \mathfrak{m}$ a uniformizer.  Write $\KK = \KK(v)$ for the function field, $\kappa = \kappa(v)$ for the residue field, $v \colon \KK^\times \rightarrow \RR_{>0}$ for the discrete valuation and $\Gamma_v = \Gamma_{\mathfrak{m}}$ for the valuation group.  If $R$ is not complete, then $\hat{R}$ and $\hat{\KK}$ denote the respective $v$-adic completions of $R$ and $\KK$.

\subsection{Maximal orders in division algebras over complete discrete valuation rings}\label{maximal:orders:division:algebras:complete:dvrs}  Let $R$ be a complete discrete valuation ring and $\mathfrak{D}$ a degree $m$ $\KK$-central division algebra.  Then:
$$
m^2 = [\mathfrak{D} : \KK] = \dim_{\KK} \mathfrak{D}.
$$
Let
$
\N_{\mathfrak{D} / \KK}(\cdot) \colon \mathfrak{D} \rightarrow \KK
$
be the norm from $\mathfrak{D}$ to $\KK$ and put:
$$
w_{\mathfrak{D}}(\cdot) = \frac{1}{m^2} v(\N_{\mathfrak{D} / \KK}(\cdot)) \text{.}
$$
Recall, that the integral closure of $\mathfrak{D}$ in $\KK$ is the unique maximal order
$
\mathfrak{R} := \{ a \in \mathfrak{D} : \N_{\mathfrak{D} / \KK}(a) \in R \} 
$
and that the \emph{ramification index}
$
e_{\mathfrak{D}} := e(\mathfrak{D}/\KK) = e
$
has the property that
$
e^{-1} := \min \{ w_{\mathfrak{D}}(a) : a\in \mathfrak{D} \text{ and } w_{\mathfrak{D}}(a) > 0 \} \text{.}
$
Let 
$J := \operatorname{rad} \mathfrak{R}$ 
be the \emph{Jacobson radical} of $\mathfrak{R}$ and fix a \emph{prime element} $\pi_{\mathfrak{D}} \in \mathfrak{R}$.  Then, 
$e \cdot w_{\mathfrak{D}}(\pi_{\mathfrak{D}}) = 1$ and  
the \emph{powers} of $J$ are the two sided ideals
$
J^k := \pi_{\mathfrak{D}}^k \cdot \mathfrak{R} = \mathfrak{R} \cdot \pi_{\mathfrak{D}}^k \text{.}
$
In particular 
$
J^e = \pi^e_{\mathfrak{D}} \cdot \mathfrak{R} = t \cdot \mathfrak{R} \text{.}
$
Recall, that the \emph{residue class ring} 
$
\overline{\mathfrak{R}} := \mathfrak{R} / \operatorname{rad}{\mathfrak{R}} 
$
is a division ring over the residue field $\kappa$.  It has centre $\kappa_{\mathfrak{R}}$, which is a finite dimensional extension field of $\kappa$.

\subsection{Normal orders in central simple algebras over complete discrete valuation rings}  Let $R$ be a complete discrete valuation ring and $\Sigma \simeq \mathfrak{D}^{n \times n}$ a $\KK$-central simple algebra.  Here, 
$ \mathfrak{D}^{n \times n} \text{,}
$ 
denotes an $n \times n$ matrix algebra with entries in a $\KK$-central division algebra $\mathfrak{D}$.  Let $\Lambda$ be an $R$-order in $\Sigma$.  Recall, that $\Lambda$ is  
 \emph{hereditary} if all of its left ideals are projective.  It is \emph{normal} if its \emph{Jacobson radical} has the form:
$$ J := \operatorname{rad} \Lambda = \Lambda \pi = \pi \Lambda,$$ 
for some $\pi \in \Lambda$.  The Jacobson radical of a normal order is projective and hence normal orders are hereditary \cite[Theorem 39.1]{Reiner:2003}.  While the structure of hereditary orders is well understood,~\cite[Theorem 39.14]{Reiner:2003}, it is possible to give a more refined description of the structure of normal orders (see Theorem \ref{normal:complete:dvr:theorem2}).

Let  $\mathfrak{R}$ be the unique maximal $R$-order in $\mathfrak{D}$ with residue ring $\overline{\mathfrak{R}}$.  The structure of hereditary $R$-orders in $\Sigma$ is summarized in {\cite[Theorem 39.14]{Reiner:2003}}.  In particular, if $\Lambda$ is a hereditary order in: 
$$\Sigma \simeq \mathfrak{D}^{n \times n},$$ then there exists positive integers $\{n_1,\dots,n_r \}$, with $n = \sum_{i=1}^r n_i$, so that:
$$\Lambda / \operatorname{rad} \Lambda \simeq \prod_{i=1}^r \overline{\mathfrak{R}}^{n_i \times n_i}. $$
In this setting, $r$ is the \emph{type} of $\Lambda$ and $\{n_1,\dots, n_r \}$ its \emph{invariants}.

The following well known theorem distinguishes normal $R$-orders amongst hereditary orders.   It follows from the above discussion together with \cite[Corollary 39.18, p.~ 360]{Reiner:2003}.

\begin{theorem}[{\cite[\S 39]{Reiner:2003}}, {\cite[1.1 (6)]{Hijikata:Nishida:1998}}]
\label{normal:complete:dvr:theorem2}
Let $\Lambda$ be a normal $R$-order in $\Sigma \simeq \mathfrak{D}^{n\times n}$.  If $r$ denotes its type and $\{n_1,\dots, n_r\}$ its invariants, then $n_i = n/r$, for all $i$.
\end{theorem}

Now let $\mathfrak{A}_{(1,r)}(\mathfrak{D})$ denote the normal order of type $r$ with invariants 
$$\{\underbrace{1,\dots,1}_{\text{$r$-copies}} \}$$ 
in the matrix algebra $\mathfrak{D}^{r \times r}$.  
Fix a prime element $\pi_{\mathfrak{D}} \in \mathfrak{R}${\text .} 
Combining we obtain that
\begin{equation}\label{standard:generator:eqn} \operatorname{rad} \mathfrak{A}_{(1,r)}(\mathfrak{D}) = p \cdot \mathfrak{A}_{(1,r)}(\mathfrak{D}).
\end{equation}
Further
\begin{equation}\label{standard:gen:ramification:calc} 
p^r = \operatorname{diag}(\pi_{\mathfrak{D}},\dots, \pi_{\mathfrak{D}});
\end{equation}
we say that $p$ is the \emph{standard generator} of $\operatorname{rad} \mathfrak{A}_{(1,r)}(\mathfrak{D})$.

 Next we discuss type $r$ normal orders with invariants 
$$\{\underbrace{s,\dots,s}_{\text{$r$-copies}} \}\text{.}$$ 
Set $n = rs$ and observe that
\begin{equation}\label{standard:type:s:r:described}
\mathfrak{A}_{(s,r)}(\mathfrak{D}) \simeq \mathfrak{A}_{(1,r)}(\mathfrak{D}) \otimes_R R^{s \times s} \subseteq \mathfrak{D}^{n \times n}. 
\end{equation}
Moreover
\begin{equation}\label{standard:type:s:r:rad} \operatorname{rad} \mathfrak{A}_{(s,r)}(\mathfrak{D}) = \operatorname{rad} \mathfrak{A}_{(1,r)}(\mathfrak{D}) \otimes_R R^{s \times s}.
\end{equation}
Finally, if $p$ is the standard generator for $\operatorname{rad} \mathfrak{A}_{(1,r)}(\mathfrak{D})$,  then using \eqref{standard:type:s:r:described} and \eqref{standard:type:s:r:rad}, it is deduced that $p\otimes 1$ generates $\operatorname{rad} \mathfrak{A}_{(s,r)}(\mathfrak{D})$.  Because of this, we also refer to $p\otimes 1$ as the \emph{standard generator} for $\operatorname{rad} \mathfrak{A}_{(s,r)}(\mathfrak{D})$.

Theorem~\ref{normal:complete:dvr:theorem2} has the following consequence which will be of use when we prove Theorem~\ref{ramification:data:theorem1}.

\begin{corollary}\label{normal:complete:dvr:corollary3}
If $\Lambda$ is a normal order in $\KK^{n\times n}$ of type $r$, then 
$$ \Lambda \simeq \mathfrak{A}_{(s,r)}(R),$$
for $s = n/r$.
\end{corollary}
\begin{proof}
Corollary~\ref{normal:complete:dvr:corollary3}  is a special case of Theorem ~\ref{normal:complete:dvr:theorem2} with $\mathfrak{D} = \KK$.  Indeed,  in this case, we have $\mathfrak{R} = R$.
\end{proof}

\subsection{Ramification of maximal orders over complete discrete valuation rings}\label{ramification:data}  We continue to denote by $R$ a complete discrete valuation ring and we discuss ramification of $\Lambda$, a maximal $R$-order in a $\KK$-central simple algebra $\Sigma$.
Let  
$\kappa_\Lambda$ be
the centre of the residue ring  
$\bar{\Lambda} := \Lambda / \operatorname{rad} \Lambda$.  
The numerical invariants which are of interest are positive integers, $e,e',f$ and $n$, which are defined by the conditions that
\begin{equation}\label{ram:numbers:defined}
\left(\operatorname{rad} \Lambda \right)^e = \Lambda \mathfrak{m}, 
\ 
e' = [ \kappa_{\Lambda} : \kappa], 
\ 
f^2 = [\bar{\Lambda} : \kappa_{\Lambda}], 
\text{ and }
 n^2 = [\Sigma:\KK];
\end{equation}
they have the properties that
\begin{equation}\label{ramification:data:eqn2}
ee'f^2=n^2.
\end{equation}
Indeed, since $\Lambda$ is normal we can choose $\pi \in \Lambda$ which generates its radical $J$.  Next, consider the $J$-adic filtration on $\Lambda$.  It follows that 
$$n^2 = \dim_{\kappa} \Lambda / \mathfrak{m} \Lambda = e \dim_{\kappa} \bar{\Lambda} = e e' f^2$$
\cite[Theorem 13.3]{Reiner:2003}, \cite[Lemma 1.5.9 (i)]{Artin:Chan:deJong:Lieblich}.

\begin{defn}
We say that a maximal $R$-order $\Lambda$, as above, is \emph{ramified} if $ee' > 1$; the number $e$ is the \emph{ramification index} of $\Lambda$.
\end{defn}

Consider now a finite \'etale extension 
$ R \rightarrow S$ of complete discrete valuation rings. We employ standard base change notation.  For example
${}_S\Lambda := S \otimes_R \Lambda$ whereas 
 $\kappa({}_S\Lambda) = \kappa_{ {}_S\Lambda}$ denotes the centre of the residue ring  $_S\bar{\Lambda} :={} _S\Lambda / \operatorname{rad}{}_S\Lambda$.  Let $\FF$ be the field of fractions of $S$ and put
$${}_S\Sigma := \FF \otimes_\KK  \Sigma \text{.}$$  

\begin{proposition}[Compare with {\cite[Proposition 1.5.8]{Artin:Chan:deJong:Lieblich}}]\label{ramification:data:proposition:1'}
With the notations and assumptions as above,
the
following assertions hold true:
\begin{enumerate}
\item{$\operatorname{rad}({}_S\Lambda) = S \otimes_R  \operatorname{rad}(\Lambda) $;}
\item{the $S$-order ${}_S\Lambda$ in ${}_S\Sigma$ is normal; and}
\item{the numbers $e,e',f$ and $n$, as defined above, satisfy the conditions that
$$\left(\operatorname{rad} {}_S\Lambda \right)^e = \mathfrak{m} {}_S\Lambda,  \ 
e' = [ \kappa({}_S\Lambda) : \kappa(S)], 
\  
f^2 = [{}_S\bar{\Lambda} : \kappa({}_S\Lambda)], 
 \text{ and }
n^2 = [{}_S\Sigma:\FF].$$
}
\end{enumerate}
\end{proposition}
\begin{proof}
The ring $\kappa(S) \otimes_\kappa \bar{\Lambda}$ is a semisimple algebra, by~\cite[Chapter XVII, Theorem 6.2]{Lang:Algebra}, and it follows that 
$$\operatorname{rad} {}_S \Lambda = S \otimes_R  \operatorname{rad}(\Lambda)  $$
which establishes (a) and (b) since $\Lambda$ is normal.  
For (c), we first note that, in light of (a) and (b), the equation
$$(\operatorname{rad} {}_S\Lambda)^e = \ {}_S\Lambda  \ \mathfrak{m}{}_S $$
holds true.  Finally, since centres are compatible with flat base change, the 
equations given by (c) 
are valid.
\end{proof}

We now discuss ramification within the context of splitting fields.

\begin{defn}
We say that a $\KK$-central simple algebra $\Sigma$ has an \emph{unramified splitting field} if there exists an \'etale extension $R\rightarrow S$ of complete discrete valuation rings so that if $\FF$ denotes the function field of $S$, then the $\FF$-central simple algebra $\FF \otimes_\KK  \Sigma $ splits.
\end{defn}

The $\KK$-central simple algebra $\Sigma$ is Morita equivalent to some $\KK$-central division ring $\mathfrak{D}$ and the \emph{index} of $\Sigma$ is defined to be 
the positive integer with the property that its square equals $[\mathfrak{D} : \KK]$.  We now discuss existence of unramified splitting fields.

\begin{proposition}
[Compare with {\cite[Proposition 1.5.11]{Artin:Chan:deJong:Lieblich}} or {\cite[Chapter XII, \S 1]{Serre:local:fields}}]
\label{ramification:data:prop2}  Let $R$ be a complete discrete valuation ring and $\Sigma$ a $\KK$-central simple algebra.  Assume either that $\kappa$ is perfect or that the index of $\Sigma$ is prime to $\operatorname{char}(\kappa)$.  Then $\Sigma$ admits an unramified splitting field.
\end{proposition}
\begin{proof}  
The argument in \cite[Chapter XII, \S 2]{Serre:local:fields} treats the case that $\kappa$ is perfect.  It can be adapted to treat the case that the index of $\Sigma$ is prime to $\operatorname{char}(\kappa)$.
\end{proof}

Theorem~\ref{ramification:data:theorem1} below gives a sufficient condition for equality of $e$ and $e'$.

\begin{theorem}
[Compare with {\cite[Proposition 1.7.14]{Artin:Chan:deJong:Lieblich}}]
\label{ramification:data:theorem1}  Let $R$ be a complete discrete valuation ring and $\Lambda$ a maximal order in a $\KK$-central simple algebra $\Sigma$.
Assume either that $\kappa$ is perfect or that the index of $\Sigma$ is prime to the characteristic of $ \kappa$.  Then the ramification invariants of $\Lambda$ have the properties that
$$e = e' \text{ and } ef=n.$$
\end{theorem}

\begin{proof}  
Since, by Equation \eqref{ramification:data:eqn2}, $ee'f^2=n^2$, it suffices to show that $ef = n$.  To this end, Proposition~\ref{ramification:data:prop2} implies that $\Sigma$ has an unramified splitting field.  Let $R \rightarrow S$ be a finite \'etale extension, of complete discrete valuation rings, which splits $\Sigma$ and write
$${}_S\Sigma := \FF \otimes_\KK \Sigma  \simeq \FF^{n \times n}, $$
for $\FF$ the fraction field of $S$.

By Proposition~\ref{ramification:data:proposition:1'}, we have that ${}_S \Lambda$ is normal in ${}_S\Sigma$  and also that 
$$(\operatorname{rad} {}_S\Lambda)^e = {}_S\Lambda (\mathfrak{m}{}_S) {}_S\Lambda, \ 
e' = [\kappa({}_S\Lambda) : \kappa(S)], \
f^2 = [\bar{\Lambda}{}_S : \kappa({}_S\Lambda)] \
\text{ and }
n^2 = [ {}_S\Sigma : \FF]. $$
On the other hand, by Corollary~\ref{normal:complete:dvr:corollary3}, 
$${}_S\Lambda \simeq \mathfrak{A}_{(s,r)}(S) $$
for some $(s,r)$ with $sr=n$; by considering the standard generator for $\operatorname{rad} \mathfrak{A}_{(s,r)}(S)$, it is clear that $e = r$.  Further,
$$\mathfrak{A}_{(s,r)}(S) / \operatorname{rad} \mathfrak{A}_{(s,r)}(S) \simeq \prod_{i=1}^r \kappa(S)^{s \times s} $$
and so 
$$ \kappa({}_S\Lambda) \simeq \kappa(S)^{\oplus r}$$
whence $f = s$.  Then 
$ e' = r = e $
since $n = rs = e f$ and  
$e e' f^2 = n^2$.
\end{proof}
Theorem~\ref{ramification:data:theorem1} has the following consequence.
\begin{corollary}\label{ramification:data:corollary1'}
In the setting of Theorem~\ref{ramification:data:theorem1}, the ramification invariants of the maximal order $\Lambda$ are equal to those of the unique maximal order $\mathfrak{R}$ in $\mathfrak{D}$ the unique division algebra which is Morita equivalent to $\Sigma$.
\end{corollary}
\begin{proof}  Since 
$\Sigma \simeq \mathfrak{D}^{k \times k}$, 
for 
$k [\mathfrak{D}:\KK] = n^2$, 
the maximal order $\Lambda$ in 
$\Sigma$ will be of the form 
$\Lambda \simeq \mathfrak{R}^{k \times k}$.  It follows that the centres of the algebras $\overline{\Lambda}$, $(\mathfrak{R} / \operatorname{rad}(\mathfrak{R}))^{k \times k}$ and  $\overline{\mathfrak{R}}$ coincide.  
In particular, the $e'$ of $\Lambda$, as defined in \eqref{ram:numbers:defined}, equals the $e_{\mathfrak{D}}'$ of $\mathfrak{D}$, also as defined in \eqref{ram:numbers:defined}.   The other relations follow from Theorem~\ref{ramification:data:theorem1}.
\end{proof}

\subsection{Maximal orders in central simple algebras over discrete valuation rings}  Now let $R$ be a discrete valuation ring, not necessarily complete.  Let $\Lambda$ be a maximal $R$-order, in a central simple $\KK$-algebra $\Sigma$. Put
$\hat{\Lambda} := \hat{R}\otimes_R \Lambda $
and
$\hat{\Sigma} := \hat{\KK}\otimes_\KK \Sigma. $
By~\cite[Proposition 2.5]{Auslander:Goldman:1960}, $\hat{\Lambda}$ is a maximal $\hat{R}$-order in $\hat{\Sigma}$ and its radical is described as
$$\operatorname{rad} \Lambda =\Lambda \cap \operatorname{rad}\hat{\Lambda}.$$
By \cite[Theorems 18.3, 19.3]{Reiner:2003},
the powers of $\operatorname{rad} \Lambda$ are expressed as
$$(\operatorname{rad} \Lambda)^k := \Lambda \cap (\operatorname{rad} \hat{\Lambda})^k = \Lambda \cap \hat{\Lambda} \pi_{\mathfrak{D}}^k \hat{\Lambda},  $$
for $k \in \ZZ_{\geq 0}$.
In this notation
$$\Lambda t = \Lambda \cap (\operatorname{rad} \hat{\Lambda})^e = (\operatorname{rad} \Lambda)^e. $$

We now assume that the index of $\Sigma$ is prime to $\cchar(\kappa)$ or that $\kappa$ is perfect.  Proposition~\ref{dual:trace:rad:lemma} below shows how the  $\Lambda$-bimodule 
$$\Lambda^\vee := \Hom_R(\Lambda,R)$$
relates to $\operatorname{rad}(\Lambda)$ and the ramification invariant $e$.  It can be used to establish Proposition \ref{azumaya:discrim:locus:prop}.

\begin{proposition}\label{dual:trace:rad:lemma}  In the above setting, the reduced trace pairing
\begin{equation}\label{reduced:trace:pairing}
\Sigma \times \Sigma \rightarrow \KK,\end{equation}
which is defined by 
$$(x,y) \mapsto \Trd_{\Sigma/\KK}(xy),$$ 
induces an isomorphism
$$\Lambda^\vee \simeq (\operatorname{rad} \Lambda)^{1-e} $$
of $\Lambda$-bimodules. 
\end{proposition}
\begin{proof}
Via the reduced trace pairing \eqref{reduced:trace:pairing}, the module $\Lambda^\vee$ is identified with the inverse different
$$ \widetilde{\Lambda} = \{x \in \Sigma : \Trd_{\Sigma / \KK}(x\Lambda) \subseteq R \},$$
\cite[p.~184]{Reiner:2003}.  Further, as explained in~\cite[\S 20]{Reiner:2003}, the inverse different commutes with taking completion, and with taking matrix algebras.  Thus, without loss of generality, we may assume that $(R,\mathfrak{m})$ is complete, that $\Sigma = \mathfrak{D}$, for $\mathfrak{D}$ a $\KK$-central division ring, and $\Lambda =\mathfrak{R}$ the unique maximal $R$-order in $\mathfrak{D}$.  In this setting it is known, see for instance~\cite[p.~148 and p.~149]{Schilling:1950}, compare also with~\cite[Proposition 1.8.2]{Artin:Chan:deJong:Lieblich},~\cite[Proposition 10.12]{Chan:Orders} and~\cite[Discussion on p.~ 152]{Chan:Kulkarni2003}, that when $\kappa$ is perfect or the degree of $\mathfrak{D}$ is prime to $\cchar(\kappa)$, that
$$ \widetilde{\Lambda} \simeq (\operatorname{rad} \mathfrak{R})^{-e+1}$$
as desired.
\end{proof}

\subsection{Cohomological interpretation of ramification for maximal orders}\label{cohom:max:orders:ramification}
Finally, we discuss ramification of maximal orders from the point of view of Galois cohomology.  
Our approach is explicit, requires existence of primitive roots of unity and handles the case of characteristic zero discrete valuation rings.  Alternative treatments of ramification maps can be found in \cite{Saltman:1999}, \cite{Artin:Chan:deJong:Lieblich}, \cite{Chan:Orders} or \cite[Chapter XII, \S 3]{Serre:local:fields}.  Our conventions about Galois cohomology are  based on those of \cite{Gille:Szamuley:2006}.   

Let $R$ be a characteristic zero discrete valuation ring, with residue field $\kappa$, and assume that $R$ contains a primitive $m$th root of unity $\zeta_m \in R$.  Assume that $(m,\operatorname{char}(\kappa)) = 1$ if $\operatorname{char}(\kappa) > 0$.
The fixed choice of $\zeta_m$ induces a group isomorphism
$
\mu_m \simeq \ZZ / m \ZZ \text{.}
$
By Kummer theory, together with the Merkurjev-Suslin theorem \cite[Theorem 4.6.6]{Gille:Szamuley:2006}, there exists a commutative diagram
\begin{equation}\label{ramification:commutative:diagram}
\begin{tikzcd}
\K_2^M(\KK) / m \arrow[r, "\sim"] 
\arrow[d] & \H^2(\KK,\mu_m^{\otimes 2}) 
\arrow[d] & \arrow[d] \arrow[l, swap, "\sim"] \Br(\KK)[m] \\
\K_2^M(\hat{\KK}) / m \arrow[r, "\sim"] \arrow[d] & \H^2(\hat{\KK},\mu_m^{\otimes 2}) \arrow[d] & \arrow[l, swap, "\sim"] \Br(\hat{\KK})[m] \\
\K_1^M(\kappa) / m \arrow[r, "\sim"] & \H^1(\kappa,\mu_m) 
\end{tikzcd}
\end{equation}
which induces the \emph{ramification group homomorphism}
\begin{equation}\label{ramification:group:homomorphism:defn'}
a \colon \Br(\KK)[m] \rightarrow \H^2(\hat{\KK},\mu_m^{\otimes 2}) \rightarrow \H^1(\kappa,\mu_m) \text{.}
\end{equation}
Using this homomorphism \eqref{ramification:group:homomorphism:defn'}, we are able to give a direct explicit proof of Theorem \ref{explicit:ramification:map:theorem} below.  Our viewpoint here complements other accounts which are currently present in the literature. 

\begin{theorem}\label{explicit:ramification:map:theorem}  Let $R$ be a characteristic zero discrete valuation ring, with residue field $\kappa$, and assume that $R$ contains a primitive $m$th root of unity $\zeta_m \in R$.  Assume that $(m, \operatorname{char}(\kappa)) = 1$ if $\operatorname{char}(\kappa) > 0$.
Let $\Lambda$ be a maximal order in a $\KK$-central simple algebra $\Sigma$.  Assume that $\Sigma$ has period $m$.  With these assumptions, the ramification index of $\Lambda$ equals the degree of the cyclic field extension that corresponds to the image of the class $\alpha$ of $\Sigma$ under the ramification map $a$.  This number equals the order of $a(\alpha) \in \H^1(\kappa,\mu_m)$.
\end{theorem}

The proof of Theorem \ref{explicit:ramification:map:theorem} is a consequence of a collection of auxiliary results.
Moreover, applying the theorem of Merkurjev and Suslin, \cite[Theorem 2.5.7]{Gille:Szamuley:2006}, a primitive generator for this extension can be described explicitly under the image of the tame symbol map.  First, we record the following proposition which collects a handful of facts which were observed in the course of proving Theorem \ref{normal:complete:dvr:theorem2} and Corollary \ref{ramification:data:corollary1'}.

\begin{proposition}\label{maximal:order:prop}
Let $R$ be a characteristic zero discrete valuation ring, with residue field $\kappa$, and assume that $R$ contains a primitive $m$th root of unity $\zeta_m \in R$.  Assume that $(m, \operatorname{char}(\kappa)) = 1$ if $\operatorname{char}(\kappa) > 0$.  Let $\Lambda$ be a maximal $R$-order in a $\KK$-central simple algebra $\Sigma$.  Let $\kappa_{\Lambda}$ be the centre of  $\Lambda / \operatorname{rad} \Lambda$.  Write $\operatorname{rad} \Lambda = \pi \Lambda$ for some prime element  $\pi \in \Lambda$.  The following assertions hold true.
\begin{enumerate}
\item{
The field $\kappa_{\Lambda}$ is a cyclic field extension of $\kappa$.  It is independent of the choice of maximal order $\Lambda$.  More precisely, it depends only on the class of $\hat{\Sigma}$ in the Brauer group of $\hat{\KK}$.
}
\item{
Conjugation by $\pi$ induces a generator $\sigma$ for the Galois group $\Gal(\kappa_{\Lambda} / \kappa)$.  This automorphism of $\kappa_{\Lambda}$ is well-defined.  It depends only on the class of $\hat{\Sigma}$ in the Brauer group of $\hat{\KK}$.  This is the distinguished generator which maps to $1/[\kappa_{\Lambda} : \kappa]$ in $\QQ/\ZZ$ under our chosen isomorphism 
$$
\H^1(\kappa,\mu_m) \simeq \H^1(\kappa, \ZZ/m\ZZ) \simeq \Hom_{\mathrm{cont}}(\Gal(\overline{\kappa}/\kappa), \ZZ/m\ZZ).
$$}
\end{enumerate}
\end{proposition}
\begin{proof}
Write $\operatorname{rad} \Lambda = \pi \Lambda$.  In the course of proving Theorem \ref{normal:complete:dvr:theorem2}, we noted that conjugation by $\pi$ induces an automorphism of $\kappa_{\Lambda}$.  This automorphism has order equal to $e = [\kappa_{\Lambda}:\kappa]$.  Both assertions (a) and (b) are then easily checked.  
\end{proof}

We now prove Theorem \ref{explicit:ramification:map:theorem}.

\begin{proof}[Proof of Theorem \ref{explicit:ramification:map:theorem}]
By the theorem of Merkurjev-Suslin, \cite[Theorem 2.5.7]{Gille:Szamuley:2006}, we may write the class of $\Sigma$ in the form
$$
\Sigma \sim (a_1,b_1)_{\zeta_m} \otimes \dots \otimes (a_r,b_r)_{\zeta_m}\text{,}
$$
for some positive integer $r$ and for certain cyclic algebras $(a_i,b_i)_{\zeta_{m}}$, $a_i,b_i \in \KK^\times$, for $i = 1,\dots, r$.  This allows for the calculation of the image of the class of $\Sigma$ under the ramification map $a$ to be carried out at the level of symbols.  As noted in Proposition \ref{maximal:order:prop}, the  radical of $\Lambda$ determines a generator for the Galois group.  More precisely, this generator is determined by $\pi$ a generator for $\operatorname{rad} \Lambda$.  By Corollary \ref{ramification:data:corollary1'}, the ramification index of $\Lambda$ depends only on the Morita equivalence class of $\Lambda$.  For each $i = 1,\dots,r$, let
$$
\alpha_i := (-1)^{v(a_i) v(b_i) } \overline{a_i^{v(b_i)}b_i^{-v(a_i)}}  \in \kappa(v) = \kappa
$$
and put
$$
\alpha := \prod_{i=1}^r \alpha_i \text{.}
$$
Then 
using 
Proposition \ref{maximal:order:prop} together with the above commutative diagram \eqref{ramification:commutative:diagram}, we deduce   that
$$\kappa_{\Lambda} = \kappa(\alpha^{1/m}) \text{.}$$
The desired result then follows since $ e = [\kappa_{\Lambda}:\kappa]$ by Theorem \ref{ramification:data:theorem1}.
\end{proof}

\section{Iitaka-Kodaira-D-dimensions, b-divisors and ramification of Brauer classes}  Let $X$ be a normal proper variety, defined over an algebraically closed field $\kk$, with function field $\KK = \kk(X)$.  In this section, we define and study concepts related to log pairs $(X,\Delta_\alpha)$ determined by Brauer classes $\alpha \in \Br(\KK)$ which have degree prime to $\cchar(\kk)$.  Initially, we allow $\kk$ to have positive characteristic, but in order to consider Kodaira dimensions of Brauer classes, we later insist that $\kk$ have characteristic zero.

\subsection{Divisors, section rings and Iitaka-dimensions}  If $\mathfrak{p} \in X$ is a codimension $1$ point, then it is identified with the irreducible codimension $1$ subvariety of $X$ that it determines.  Let $(\Osh_{X,\mathfrak{p}},\mathfrak{m}_{\mathfrak{p}})$ be its local ring, $\kappa(\mathfrak{p}) := \Osh_{X,\mathfrak{p}}/\mathfrak{m}_{\mathfrak{p}}
$
the residue field and $\hat{\Osh}_{X,\mathfrak{p}}$ and $\KK_{\mathfrak{p}} = \hat{\KK}$, respectively, the $\mathfrak{m}_{\mathfrak{p}}$-adic completions of  $\Osh_{X,\mathfrak{p}}$ and $\KK$.
Let $\WDiv(X)$ be the group of Weil divisors, set $\WDiv_\QQ(X) := \QQ \otimes_\ZZ \WDiv(X)$ and $\lfloor D \rfloor := \sum \lfloor d_i \rfloor \mathfrak{p}_i$.  Here $\lfloor d \rfloor$, for $d \in \QQ$, is the largest integer which is at most $d$.   Let $\Div(X) \subseteq \WDiv(X)$, be the subgroup of Cartier divisors and set $\Div_\QQ(X) := \QQ \otimes_\ZZ \Div(X)$. A divisor $D \in \WDiv_\QQ(X)$ is called $\QQ$-Cartier if some positive integral multiple of $D$ is an element of $\Div(X)$. 
Recall that the Weil divisorial sheaf of a divisor $D \in \WDiv(X)$ is the rank $1$ reflexive sheaf on $X$ defined by the condition that:
$$
\Osh_X(D)(U) := \{s \in \KK : \operatorname{div}(s)|_U + D  |_U \geq 0 \}
$$
for Zariski open subsets $U \subseteq X$.  Here $\operatorname{div}(s)$ denotes the divisor determined by $s \in \KK$.

If $R = \bigoplus_{\ell \geq 0} R_\ell$ is a graded commutative $\kk$-domain with $R_0 = \kk$, then $\kk(R)$ denotes the field obtained by taking the degree zero piece of $R$ localized at the set of nonzero homogeneous elements.
Further, $\N(R)$ is the semigroup of $R$.  Explicitly,
$
\N(R) := \{\ell \geq 0 : R_\ell \not = 0 \}.
$

Given a nonnegative integer $d$, let $R^{(d)} := \bigoplus_{\ell \geq 0} R_{\ell d}$ be the $d$th Veronese subalgebra of $R$. Recall that:
\begin{equation}\label{prelim:normal:var:eqn3}
\kk(R^{(d)}) = \kk(R).
\end{equation}
For our purposes, we are interested in 
$$
R(X,D) := \bigoplus_{\ell \geq 0} \H^0(X,\Osh_X(\lfloor \ell D \rfloor ))\text{,}
$$
the section ring of a given Weil divisor $D \in \WDiv_\QQ(X)$.  To indicate dependence on $D$, $\N(X,D)$ and $\kk(X,D)$ denote, respectively, its semigroup and degree zero fraction field.  
Our conventions for the \emph{Iitaka dimension} of $D \in \WDiv_\QQ(X)$ are such that:
\begin{equation}\label{prelim:normal:var:eqn5}
\kappa(X,D) := \begin{cases}
-1 & \text{ if $\N(X,D) = (0)$} \\
\trdeg \kk(X,D) & \text{ if $\N(X,D) \not = (0)$.}
\end{cases}
\end{equation}
Recall, that one consequence of \eqref{prelim:normal:var:eqn3} and \eqref{prelim:normal:var:eqn5} is that
\begin{equation}\label{prelim:normal:var:eqn6}
\kappa(X,D) = \kappa(X,\ell D)
\end{equation}
for all positive integers $\ell$, compare with~\cite[p.~274]{Mori:1985}.

Recall, that for a $\QQ$-Cartier divisor $D$, the Iitaka dimension $\kappa(X,D)$ can be described in terms of growth.  Precisely, fix a positive integer $m$ so that $m D \in \Div(X)$.  Then, there exists constants $\alpha,\beta >0$ so that if $\kappa := \kappa(X,D)$, then
\begin{equation}\label{prelim:normal:var:eqn7}
\alpha \cdot \ell^{\kappa} \leq h^0(X,\Osh_X(\ell m D)) \leq \beta \cdot \ell^{\kappa}
\end{equation}
for all sufficiently large $\ell \in \N(X, mD)$ \cite[Theorem 1]{Iitaka:1971}. 
(When $\N(X,mD) = (0)$, $\kappa(X,D) = -1$.)  

\subsection{Log pairs}  
Let $\K_X$ be a canonical divisor on $X$. 
A \emph{log pair} $(X,\Delta)$ is obtained by fixing an effective divisor $\Delta \in \WDiv_\QQ(X)$.  We say that $(X,\Delta)$ is \emph{$\QQ$-Gorenstein} if $\K_X + \Delta$ is $\QQ$-Cartier.  The \emph{log canonical ring} of a $\QQ$-Gorenstein pair $(X,\Delta)$ is defined to be the section ring of the $\QQ$-Cartier divisor $\K_X + \Delta$:
$$
R(X,\K_X + \Delta) := \bigoplus_{\ell \geq 0} \H^0(X,\Osh_X(\lfloor \ell (\K_X + \Delta) \rfloor )).
$$

\begin{defn}\label{Iitaka-Q-Gorenstein-pair}
The \emph{Iitaka dimension} of a $\QQ$-Gorenstein pair $(X,\Delta)$ is defined to be $\kappa(X,\K_X + \Delta)$.  This is the Iitaka dimension of the $\QQ$-Cartier divisor $\K_X + \Delta$.
\end{defn}

\begin{remark}
Let $(X,\Delta)$ be a $\QQ$-Gorenstein pair and $\ell$ a positive integer so that $\ell(\K_X + \Delta) \in \Div(X)$.  Then:
$$
\kappa(X,\K_X + \Delta) = \kappa(X,\ell(\K_X + \Delta))
$$
as is implied by \eqref{prelim:normal:var:eqn6}.   
\end{remark}

\subsection{Brauer classes, maximal orders and ramification divisors}\label{Brauer:ramification}  
Let $\Sigma$ be a central simple $\KK$-algebra with degree prime to $\cchar(\kk)$.  By an \emph{$\mathcal{O}_X$-order} in $\Sigma$, we mean a sheaf of algebras $\Lambda$ on $X$, which has the three properties that:
\begin{enumerate}
\item[(a)]{
$\Lambda$ is a coherent $\Osh_X$-module on $X$;
}
\item[(b)]{
$\Lambda$ is torsion free; 
and
}
\item[(c)]{$\KK \otimes_{\Osh_X} \Lambda  \simeq \Sigma$,} by a fixed isomorphism.
\end{enumerate}

Such an $\Osh_X$-order $\Lambda$ is \emph{maximal} if it is maximal with respect to inclusion amongst orders in $\Sigma$.  
If $\Lambda$ is a maximal $\Osh_X$-order in $\Sigma$ and if $\mathfrak{p}$ is a codimension $1$ point of $X$, then 
$\Lambda_{\mathfrak{p}} :=  \Osh_{X,\mathfrak{p}} \otimes\Lambda $ 
is a maximal $\Osh_{X,\mathfrak{p}}$-order in $\Sigma$ and  
$\hat{\Lambda}_{\mathfrak{p}} :=  \hat{\mathcal{O}}_{X,\mathfrak{p}} \otimes \Lambda_{\mathfrak{p}} $ 
is a maximal $\hat{\Osh}_{X,\mathfrak{p}}$-order in the $\KK_{\mathfrak{p}}$-central simple algebra 
$\Sigma_{\mathfrak{p}} := \KK_{\mathfrak{p}} \otimes_{\KK} \Sigma$.
 Also,  by~\cite[Theorem 1.5, p.~ 3]{Auslander:Goldman:1960}, an $\Osh_X$-order $\Lambda$ in $\Sigma$ is maximal if and only if it satisfies the two conditions that:
\begin{itemize}
\item[(a)]{
it is reflexive as an $\Osh_X$-module; and
 }
 \item[(b)]{
$\Lambda_{\mathfrak{p}}$ is a maximal $\Osh_{X,\mathfrak{p}}$-order for each codimension $1$ point $\mathfrak{p}$ of $X$.
 }
\end{itemize}
Let $\Lambda$ be a maximal $\Osh_X$-order in $\Sigma$, fix $\mathfrak{p}$ a codimension $1$ point of $X$, denote by 
$
\kappa_{\Lambda}(\mathfrak{p})
$
the centre of the ring 
$\hat{\Lambda}_{\mathfrak{p}} / \operatorname{rad} \hat{\Lambda}_{\mathfrak{p}}$ 
and put: 
$$
e_{\Lambda}(\mathfrak{p}) := [\kappa_{\Lambda}(\mathfrak{p}):\kappa (\mathfrak{p})]. 
$$
Then since $\hat{\Lambda}_\mathfrak{p}$ is maximal it follows, by Theorem~\ref{ramification:data:theorem1}, 
that:
$$
\hat{\Lambda}_{\mathfrak{p}}\hat{\mathfrak{m}}_{\mathfrak{p}}  =  \hat{\mathfrak{m}}_{\mathfrak{p}} \hat{\Lambda}_{\mathfrak{p}} = (\operatorname{rad} \hat{\Lambda}_{\mathfrak{p}})^{e_{\Lambda}(\mathfrak{p})}.
$$

\begin{defn}  We say that $e_{\Lambda}(\mathfrak{p})$ is the \emph{ramification index} of $\Lambda$ at $\mathfrak{p}$.  If $e_{\Lambda}(\mathfrak{p}) > 1$, then $\Lambda$ is \emph{ramified} at $\mathfrak{p}$.
 \end{defn}

     Before stating  Proposition~\ref{azumaya:discrim:locus:prop} below, we recall the concept of discriminant divisor of a maximal order $\Lambda$ inside of $\Sigma$ (compare with \cite[\S 25]{Reiner:2003}).  Let $U_{\mathrm{loc.free}} \subseteq X$ be the Zariski open subset over which $\Lambda$ is locally free as an $\Osh_X$-module and let $U = \spec R \subseteq U_{\mathrm{loc. free}}$ be an affine open subset over which $\Lambda$ is free.
     
     Fix a basis $e_1,\dots,e_{n^2}$, $n^2 = [\Sigma:\KK]$, for $\Lambda(U) \subseteq \Sigma$ over $\Osh_X(U) = R \subseteq \KK$.  The \emph{discriminant} of $\Lambda$ over $U$, with respect to the basis $e_1,\dots, e_{n^2}$ and the reduced trace pairing induced by $\Trd_{\Sigma/\KK}$, is given by:
     \begin{equation}\label{discrim:eqn}
     \mathrm{Disc}_{\Lambda / R}(e_1,\dots,e_{n^2}) := \det(\Trd_{\Sigma / \KK}(e_i e_j)) \in R.
     \end{equation}
     The principal ideal $\mathfrak{d}_{\Trd_{\Sigma / K} , R}$ generated by \eqref{discrim:eqn} does not depend on the choice of basis and such principal ideals $\mathfrak{d}_{\Trd_{\Sigma / K}, R}$, defined affine locally, patch together to give a principal ideal sheaf $\mathfrak{d}_{\Trd_{\Sigma / K}, \Osh_X}$ which determines an effective Cartier divisor on $U_{\mathrm{loc. free}}$.  Denote its closure in $X$ by $\operatorname{Disc}_X(\Lambda)$.  This is a Weil divisor in $X$.

  Proposition \ref{azumaya:discrim:locus:prop} below is well-known and its proof is omitted.  On the other hand, one approach to its proof is based on the concept  of \emph{discriminant} for maximal orders,  as described above,  together with Proposition \ref{dual:trace:rad:lemma}.
  
  \begin{proposition}\label{azumaya:discrim:locus:prop}
    If $\Lambda$ is a maximal $\Osh_X$-order in a central simple $\KK$-algebra $\Sigma$ then there exists a Weyl divisor $\operatorname{Disc}_X(\Lambda)$ in $X$, the {\it discriminant,}
    such that
    for all codimension one points $\mathfrak{p} \in X$, it holds true that $e_{\Lambda}(\mathfrak{p}) > 1$
    if and only if $\mathfrak{p} \in \operatorname{Disc}_X(\Lambda)$. 
    Furthermore, $\Disc_X(\Lambda)$ does not depend on the choice of maximal order $\Lambda$ in $\Sigma.$
    \end{proposition}

Now, let $\alpha \in \Br (\KK)$ be the class of our fixed central simple $\KK$-algebra $\Sigma$ and fix a codimension $1$ prime $\mathfrak{p}$ of $X$. Assume that the degree of $\alpha$ is prime to $\cchar(\kk)$.  Let $\hat{\alpha}_{\mathfrak{p}} \in \Br(\KK_{\mathfrak{p}})$ denote the image of $\alpha$ under the natural map 
$$\Br(\KK) \rightarrow \Br(\KK_{\mathfrak{p}}). $$
Then $\hat{\alpha}_{\mathfrak{p}}$ is the class of the central simple $\KK_{\mathfrak{p}}$-algebra $\Sigma_{\mathfrak{p}}$ and we can write
$$\Sigma_{\mathfrak{p}} \sim_{\mathrm{Morita}} \mathfrak{D}_{\mathfrak{p}} $$
for some unique, up to isomorphism, $\KK_{\mathfrak{p}}$-central division ring $\mathfrak{D}_{\mathfrak{p}}$. Put
$$ e_{\mathfrak{p}}(\alpha) := e_{\mathfrak{D}_{\mathfrak{p}}} = e(\mathfrak{D}_{\mathfrak{p}} / \KK_{\mathfrak{p}}),$$
where the number $e(\mathfrak{D}_{\mathfrak{p}}/\KK_{\mathfrak{p}})$ is defined as in \S \ref{maximal:orders:division:algebras:complete:dvrs}.

\begin{defn} In the above setting, 
we say that $e_{\mathfrak{p}}(\alpha)$ is the \emph{ramification index} of the Brauer class $\alpha$ at $\mathfrak{p}$.
\end{defn}

We observe that the numbers $e_\Lambda(\mathfrak{p})$ and $e_{\mathfrak{p}}(\alpha)$ coincide.

\begin{proposition}\label{Brauer:ram:prop}  Fix a codimension $1$ prime $\mathfrak{p}$ of $X$, let $\Lambda$ be a maximal $\Osh_X$-order in $\Sigma$ and  let $\alpha \in \Br(\KK)$ denote the Brauer class of $\Sigma$. In this setting, we then have that
$$ e_{\mathfrak{p}}(\alpha) = e_{\Lambda}(\mathfrak{p}).$$
\end{proposition}

\begin{proof}
This is an immediate consequence of Corollary~\ref{ramification:data:corollary1'} and the definitions.  Note that the hypothesis of Corollary~\ref{ramification:data:corollary1'} are indeed satisfied because the degree of $\alpha$ is assumed to be prime to $\cchar(\kk)$.  
\end{proof}

One consequence of Proposition~\ref{Brauer:ram:prop} and Proposition ~\ref{azumaya:discrim:locus:prop} is that, for a fixed Brauer class $\alpha \in \Br (\KK)$ with degree prime to $\cchar(\kk)$, and by considering the ramification of a maximal order $\Lambda$ in $\Sigma$, a central simple algebra representing $\alpha$, there are at most finitely many prime divisors $\mathfrak{p} \in X$ with  $e_{\mathfrak{p}}(\alpha) > 1$.  In light of this, we associate to every fixed Brauer class $\alpha \in \Br (\KK)$, a 
Weil divisor
$$\Delta_{X,\alpha} = \Delta_\alpha \in \WDiv_\QQ(X)$$ 
which is defined by
\begin{equation}\label{brauer:boundary:defn}
\Delta_\alpha := \sum\limits_{\mathfrak{p} \in X, \, \operatorname{ht}(\mathfrak{p}) = 1} \left(1 - \frac{1}{e_{\mathfrak{p}}(\alpha)} \right)\mathfrak{p}.
\end{equation}

Motivated by the above discussion, we make the following definitions.

\begin{defns}\label{Iitaka:Brauer:defn}
Let $\alpha \in \Br (\KK)$ be a Brauer class, with degree prime to $\cchar(\kk)$. 
We say that the divisor \eqref{brauer:boundary:defn} is its  \emph{ramification divisor}.  The \emph{canonical divisor} of $\alpha$ is the Weil divisor
$
\K_{X,\alpha} = \K_\alpha := \K_X + \Delta_\alpha \in \WDiv_\QQ(X)\text{.}
$
 A \emph{Brauer pair} is a pair $(X,\alpha)$ of a normal variety $X$, with function field $\KK = \kk(X)$, together with a Brauer class $\alpha \in \Br (\KK)$.  A Brauer pair $(X,\alpha)$ is \emph{$\QQ$-Gorenstein} if its canonical divisor $\K_\alpha$ has the property that $\K_\alpha \in \Div_\QQ(X)$.  
When $X$ is additionally assumed to be proper, we define the \emph{Iitaka dimension} of a $\QQ$-Gorenstein Brauer pair $(X,\alpha)$ to be the Iitaka dimension of the $\QQ$-Gorenstein pair $(X,\Delta_\alpha)$ as given by Definition~\ref{Iitaka-Q-Gorenstein-pair}.   We denote the Iitaka dimension of $(X,\alpha)$ by $\kappa(X,\alpha)$.
\end{defns}

\subsection{Models and b-divisors}\label{b-div:models}
We now fix our conventions about \emph{birational divisors}.
Let $\kk$ be an algebraically closed field and let $\KK$ be a finitely generated field over $\kk$.
Let $\mathcal{M}_{/ \KK}$ be the category of normal proper models of $\KK$.  Then objects of $\mathcal{M}_{/ \KK}$ are normal integral proper varieties $Y / \kk$ together with a fixed morphism $\spec \KK \rightarrow Y$ over $\kk$ to the generic point of $Y$, and morphisms between objects of $\mathcal{M}_{/ \KK}$ are given by those birational morphisms $X \to Y$ over $\spec \kk$ such that the diagram
$$
\begin{tikzcd}
\Spec \KK \arrow{r}{} \arrow{dr}[swap]{} & X \arrow{d}{}\\
& Y
\end{tikzcd} 
 $$
commutes.  Every morphism of models induces pushforward morphisms, at the level of Weil divisors,
and these morphisms behave well with respect to composition.  The collection of such morphisms form an inverse system with respect to the directed set of objects of $\mathcal{M}_{/ \KK}$.    We denote the resulting inverse limit in the category of abelian groups as:
$$\DDiv_\QQ(\KK) := \idlim\limits_{Y \in \mathcal{M}_{/ \KK}} \WDiv_\QQ(Y). $$

\begin{defn}
A (fractional) \emph{b-divisor} on $\KK$ is an element $\mathbb{D} \in \DDiv_\QQ(\KK)$.
\end{defn}

We may represent a b-divisor as an infinite collection of divisors
$$
\mathbb{D} = \{D_Y\}_{Y \in \mathcal{M}_{/ \KK}}.
$$
Here the divisors $D_Y \in \WDiv_\QQ(Y)$ are divisors on normal models $Y \in \mathcal{M}_{/ \KK}$  which have the property that whenever we are given a morphism $g : Y' \rightarrow Y$ of normal models, then
$$D_Y = g_*(D_{Y'}).$$
  In this setting, 
we define 
$$\mathbb{D}_Y := D_Y$$ 
and we refer to $\mathbb{D}_Y$ as the \emph{trace of the b-divisor $\mathbb{D}$ on $Y$}.
 
\subsection{b-log $\QQ$-Gorenstein pairs}\label{b-log:div}  We now make several definitions related to the concept of a \emph{b-log $\QQ$-Gorenstein pair}.  These concepts are from~\cite{Chan:plus:10}.

\begin{defn}\label{b-Gorenstein:defn}   Let $X$ be a normal proper variety over an algebraically closed field $\kk$ with function field $\KK = \kk(X)$.  By a \emph{b-log pair}, or simply a \emph{pair}, we mean a pair $(X,\mathbb{D})$ where $\mathbb{D} \in \DDiv_\QQ(\KK)$ is a b-divisor, with the property that $\mathbb{D}_Y$ has coefficients in $[0,1) \cap \QQ$ for each $Y \in \mathcal{M}_{/\KK}$.
By a \emph{b-log $\QQ$-Gorenstein pair}, or simply a \emph{$\QQ$-Gorenstein pair}, we mean a  pair $(X,\mathbb{D})$ where $\mathbb{D} \in \DDiv_\QQ(\KK)$ is a b-divisor on $\KK$ with the property that $\K_X+\mathbb{D}_X \in \Div_\QQ(X)$.
\end{defn} 
 
Fix a $\QQ$-Gorenstein pair $(X,\mathbb{D})$ and let $f : Y \rightarrow X$ be a birational morphism.  We  assume that $Y$ is normal and $\K_{Y} + \mathbb{D}_{Y} \in \Div_\QQ(Y)$.  Let $E_1,\dots,E_k$ be the exceptional prime divisors of $f$; we can  write
\begin{equation}\label{discrep:eqn1}
\K_{Y} + \mathbb{D}_{Y}  \equiv  f^*(\K_X + \mathbb{D}_X) + \sum_i b_i' E_i,
\end{equation}
for rational numbers $b_i' \in \QQ$.  Finally, write the trace of $\mathbb{D}$ on $Y$ in the form:
\begin{equation}\label{discrep:eqn1'}
\mathbb{D}_Y = \sum (1-1/r_D)D
\end{equation}
for rational numbers $r_D$.

\begin{defn}[\cite{Chan:Ingalls:2005}, \cite{Chan:plus:10}]
The rational numbers $b(E_i;X,\mathbb{D}):=b_i'r_{E_i}$, determined by the expressions \eqref{discrep:eqn1} and \eqref{discrep:eqn1'}, are called the \emph{b-discrepancy} of the prime divisor $E_i \in \WDiv(Y)$. 
\end{defn}
Having defined the concept of b-discrepancy for a $\QQ$-Gorenstein pair $(X,\mathbb{D})$ for prime exceptional divisors  $E_i$ in a fixed normal model $f : Y \rightarrow X$, we define the b-discrepancy for a $\QQ$-Gorenstein pair $(X,\mathbb{D})$ in the following manner.

\begin{defn}[\cite{Chan:Ingalls:2005}, \cite{Chan:plus:10}]\label{QQ:Gorenstein:canonical}
We define the \emph{b-discrepancy} of a $\QQ$-Gorenstein pair $(X,\mathbb{D})$ to be:
$$
\operatorname{b-discrep}(X,\mathbb{D}) := \inf  
 \left\{ b(E;X,\mathbb{D}) : \substack{\text{$E$ is a prime  exceptional divisor  } \\ \text{ on any normal model $f:Y\rightarrow X$} \\ \text{with $\K_{Y} + \mathbb{D}_{Y} \in \Div_\QQ(Y)$ } } \right\}. 
 $$
\end{defn}
 
Finally, we introduce the concepts of b-canonical and b-terminal $\QQ$-Gorenstein pairs.

\begin{defn}[\cite{Chan:Ingalls:2005}, \cite{Chan:plus:10}]\label{b-canonial:defn}  Let $(X,\mathbb{D})$ be a $\QQ$-Gorenstein pair.
We say that $(X,\mathbb{D})$ is \emph{b-canonical}, or simply that $(X,\mathbb{D})$ has
\emph{b-canonical singularities}, if: 
$$\operatorname{b-discrep}(X,\mathbb{D}) \geq 0$$ and that $(X,\mathbb{D})$ is \emph{b-terminal} if: 
$$\operatorname{b-discrep}(X,\mathbb{D}) > 0.$$
\end{defn}

\subsection{Iitaka dimensions of b-log-$\QQ$-Gorenstein pairs}\label{b-Kodaria:defn}  We now define a concept of Iitaka dimension for a given b-log $\QQ$-Gorenstein pair.

\begin{defn}\label{b-iitaka-defn}
The \emph{Iitaka dimension} of a b-log $\QQ$-Gorenstein pair $(X,\mathbb{D})$ is denoted by $\kappa(X,\K+\mathbb{D})$.  It is defined to be that of the pair $(X,\mathbb{D}_X)$:
$$ \kappa(X,\K+\mathbb{D}) := \kappa(X, \K_X + \mathbb{D}_X). $$
\end{defn}

Theorem~\ref{b-divisor-theorem} below pertains to birational invariance of Kodaira dimension for $\QQ$-Gorenstein pairs, see Corollary~\ref{b-divisor-theorem-cor} and \S \ref{canonical:kodaira}.  That b-terminal, and hence b-canonical, $\QQ$-Gorenstein resolutions for a given fractional b-log pair exist, is established in~\cite{Chan:plus:10}. Further, in \S \ref{Galois:embeddings:division:algebras}, we formulate and establish an equivariant version of that result, see Theorem~\ref{existence:equivariant:terminal:resolutions}. 

\begin{theorem}\label{b-divisor-theorem}
Let the characteristic of $\kk$ be zero and let $\KK$ be a finitely generated field over $\kk$.  Let $\mathbb{D}$ be a b-divisor on $\KK$, with coefficients in $[0,1)\cap \QQ$, and let $X$ and $Y$ be normal proper models of $\KK$ with the pairs $(X,\DD)$ and $(Y,\DD)$ having b-canonical singularities.  Then, if $\ell(\K_Y + \DD_Y)$ and $\ell(\K_X+\DD_X)$ are Cartier, for some positive integer $\ell$, we have that:
$$R(Y, \ell(\K_Y + \DD_Y)) = R(X,\ell(\K_X + \DD_X))$$
and so 
$$\kappa(Y, \K+\mathbb{D}) = \kappa(X,\K+\mathbb{D}).$$
\end{theorem}

Theorem~\ref{b-divisor-theorem} is a consequence of Proposition~\ref{b-divisor-section-ring-proposition} and Corollary~\ref{b-divisor-proposition} below.  Before stating these results, recall that a \emph{$\QQ$-Gorenstein pair} $(X,D)$ consists of an effective divisor 
$D \in \WDiv_\QQ(X)$ 
with the property that
$\K_X + D \in \Div_\QQ(X).$  

\begin{proposition}\label{b-divisor-section-ring-proposition}  Let $X$ and $Y$ be normal proper varieties over $\kk$ and suppose given $\QQ$-Gorenstein pairs $(Y,D_Y)$, $(X,D_X)$ together with a proper birational morphism $f : Y \rightarrow X$.  Assume that: 
$$\K_Y + D_Y = f^*(\K_X + D_X) + \sum b_i E_i$$
for $E_i$ $f$-exceptional prime  divisors and non-negative rational numbers $b_i \geq 0$.  
If $\ell(\K_Y + D_Y)$ and $\ell(\K_X+D_X)$ are Cartier, for some positive integer $\ell$, then
we have equality of section rings:
$$R(Y, \ell(\K_Y + D_Y)) = R(X,\ell(\K_X+D_X)).$$
\end{proposition}
\begin{proof}
Choose $\ell > 0$ so that the divisors
$\ell(\K_Y + D_Y)$,  
$\ell(\K_X + D_X),$ 
are Cartier and so consequently
$E := \sum \ell b_i E_i$  
is also Cartier. Next, put
$L := \Osh_X(\ell(\K_X + D_X)) $
and
$M := \Osh_Y(\ell(\K_Y + D_Y)); $
then 
$M \simeq f^* L \otimes \Osh_Y(E). $
Further, since $X$ and $Y$ are normal and $E$ is exceptional and effective we have natural isomorphisms:
\begin{equation}\label{b-divisor-proposition:eqn4'}
\H^0(X,L^{\otimes m})  = \H^0(Y,M^{\otimes m}),
\end{equation}
for all integers $m>0$, since $f_* M^{\otimes m} = L^{\otimes m}$,~\cite[Lemma B.2.5]{deF:E:M}, alternatively this follows from~\cite[Example 2.1.16, p.~ 126]{Laz}.

Finally, we note that \eqref{b-divisor-proposition:eqn4'} implies equality of section rings 
$$R(Y, \ell(\K_Y + D_Y)) = R(X,\ell(\K_X+D_X))$$ as desired. 
\end{proof}

As an immediate consequence of Proposition~\ref{b-divisor-section-ring-proposition}, we record the following remark for later use:

\begin{corollary}\label{b-divisor-proposition}  With the same hypothesis as Proposition~\ref{b-divisor-section-ring-proposition}, we have equality of Iitaka dimensions:
$$\kappa(Y,\K_Y+D_Y) = \kappa(X,\K_X + D_X). $$
\end{corollary}

\begin{proof} 
Use the relation \eqref{prelim:normal:var:eqn6} combined with Proposition~\ref{b-divisor-section-ring-proposition}.
\end{proof}

We now prove Theorem~\ref{b-divisor-theorem}.

\begin{proof}[Proof of Theorem~\ref{b-divisor-theorem}]
Applying Proposition~\ref{b-divisor-section-ring-proposition} and Corollary 
\ref{b-divisor-proposition} to a common resolution $Z$ of the $\QQ$-Gorenstein pairs $(X, \DD_X)$ and $(Y, \DD_Y)$ it follows that
$$R(Y, \ell(\K_Y + \DD_Y)) = R(Z,\ell(\K_Z + \DD_Z)) = R(X,\ell(\K_X + \DD_X))$$
and so consequently
$$\kappa(X,\K_X+ \DD_X) = \kappa(Y,\K_Y + \DD_Y). $$
In light of this, the conclusion of Theorem~\ref{b-divisor-theorem} is satisfied too because of Definition~\ref{b-iitaka-defn}.
\end{proof}

We also note that Theorem~\ref{brauer:iitakacanonical:singulatiries:main:result} is implied by 
Theorem~\ref{b-divisor-theorem}. 

\begin{proof}[Proof of Theorem~\ref{brauer:iitakacanonical:singulatiries:main:result}]
Immediate consequence of Theorem~\ref{b-divisor-theorem}.
\end{proof}

\subsection{b-divisors and Brauer classes}\label{b-div:brauer}  We continue with the setting of \S \ref{b-Kodaria:defn} and explain how every class $\alpha \in \Br(\KK)$, with degree prime to $\cchar (\kk)$, determines a b-divisor $\mathbb{D}(\alpha) \in \DDiv_\QQ(\KK)$.  To this end, we have, for each $Y \in \mathcal{M}_{/\KK}$, a ramification divisor 
$$\Delta_{Y,\alpha} \in \WDiv_\QQ(Y)$$
determined by a fixed Brauer class $\alpha \in \Br(\KK)$ and the prime divisors on $Y$.  These divisors have the property that:
\begin{equation}\label{brauer:pairs:b:div}
g_* \Delta_{Y,\alpha} = \Delta_{Y',\alpha}
\end{equation}
for each morphism $g \colon Y \rightarrow Y'$ in $\mathcal{M}_{/\KK}$.  The reason that \eqref{brauer:pairs:b:div} holds true follows from the fact that the (divisorial) valuation of $\KK$ induced by the prime divisors over two models $Y$ and $Y'$ are equivalent.

In this setting, Corollary~\ref{b-divisor-proposition} takes the following form: 

\begin{corollary}\label{b-divisor-theorem-cor} Let $X$ be a normal proper variety with function field $\KK = \kk(X)$ and fix a Brauer class $\alpha \in \Br(\KK)$ with degree prime to $\cchar(\kk)$.  Suppose that $(X,\mathbb{D}(\alpha))$ is a $\QQ$-Gorenstein b-canonical pair.  Then for each proper normal model $f \colon Y \rightarrow X$ with 
$$\K_Y + \mathbb{D}(\alpha)_Y \in \Div_\QQ(Y),$$ 
we have that:
$$\kappa(X,\K+\mathbb{D}(\alpha)) = \kappa(Y,\K+\mathbb{D}(\alpha)). $$
\end{corollary}
\begin{proof}
Immediate consequence of Corollary~\ref{b-divisor-proposition}.
\end{proof}

\subsection{b-canonical resolutions, b-canonical models and Kodaira dimensions for b-log pairs}\label{canonical:kodaira}   
We now assume that $\cchar(\kk) = 0$ and formulate resolution of singularities for a given b-log $\QQ$-Gorenstein pair.  

\begin{defn}\label{b:can:res:defn} 
Let $X$ be a normal proper variety and $(X,\DD)$ a b-log pair.  By a \emph{$\QQ$-Gorenstein resolution} (resp.~\emph{b-canonical resolution}, resp.~\emph{b-terminal resolution})  of the pair $(X,\DD)$, we mean a proper birational morphism $f\colon Y \rightarrow X$ with the property that the pair $(Y,\DD)$ is $\QQ$-Gorenstein (resp.~\emph{b-canonical}, resp.~\emph{b-terminal}).
  Given a fixed b-divisor $\DD$, with coefficients in $[0,1)\cap\QQ$, we say the pair $(Y,\DD)$ is a \emph{$\QQ$-Gorenstein} (resp.~\emph{b-canonical}, resp.~\emph{b-terminal}) \emph{model} if the pair $(Y,\DD)$ is $\QQ$-Gorenstein, (resp.~\emph{b-canonical}, resp.~\emph{b-terminal}).
\end{defn}

Finally, we define concepts which relate to the Kodaira dimension of a b-divisor $\mathbb{D} \in \DDiv_\QQ(\KK)$.  Before doing so, note that Theorem~\ref{b-divisor-theorem} implies that the quantity \eqref{Kodaira:pair:defn}, below, is well defined.  In particular, it is a birational invariant of the pair $(X,\mathbb{D})$.  Further, as is a consequence of results obtained in~\cite{Chan:plus:10}, b-canonical models, in the sense of Definition~\ref{b:can:res:defn} exist.  Here we use that result to study  resolutions and Kodaira dimensions of Brauer classes.  Specifically, given a Brauer class $\alpha \in \Br(\KK)$, we consider the pair $(X,\mathbb{D}(\alpha))$ together with a b-canonical resolution $(Y,\mathbb{D}(\alpha))$.

\begin{defns}\label{b-Kodaira:defn}
Assume that $\mathbb{D} \in \DDiv_\QQ(\KK)$ has coefficients in $[0,1)\cap\QQ$ and define the \emph{Kodaira dimension} of the b-divisor $\mathbb{D}$ to be that of the pair $(Y,\mathbb{D})$:
\begin{equation}\label{Kodaira:pair:defn}
\kappa(\mathbb{D}) := \kappa(Y,\K+\mathbb{D}) = \kappa(Y,\K_Y + \mathbb{D}_Y),
\end{equation}
for  $(Y,\mathbb{D})$ some b-canonical model. 
For a fixed Brauer class, $\alpha \in \Br(\KK)$ let $(Y,\mathbb{D}(\alpha))$ be a b-canonical model and define:
\begin{equation}\label{Kodaira:Brauer:pair:defn} \kappa(\alpha) := \kappa(\mathbb{D}(\alpha)) = \kappa(Y,\K_Y + \mathbb{D}(\alpha)_Y);
\end{equation}
as in \eqref{Kodaira:pair:defn}, the definition given by \eqref{Kodaira:Brauer:pair:defn} is well defined.
We say that $\kappa(\alpha)$ is the \emph{Kodaira dimension} of $\alpha \in \operatorname{Br}(\KK)$. 
Finally, let $\KK$ be a finitely generated field over $\kk$ and let  $\Sigma$ be a  central simple algebra finite dimensional over its centre $\KK$ with Brauer class $\alpha \in \Br(\KK)$.  We define the \emph{Kodaira dimension} of  $\Sigma$, which we denote by  $\kappa(\Sigma)$, to be $\kappa(\Sigma) := \kappa(\alpha)$.  This is the Kodaira dimension of any b-log pair $(Y,\mathbb{D}(\alpha))$ which is $\QQ$-Gorenstein and has b-canonical singularities. 
\end{defns}
 
\section{Embeddings of central simple algebras and ramification formulas}\label{embeddings:div:alg:ramification:formulas}  In this section, we study the question of ramification of maximal orders in a variety of algebraic and geometric situations.

\subsection{Embeddings of central simple algebras and existence of maximal orders}\label{ram:calc}  First, we study the behaviour of  maximal orders under embeddings of central simple algebras.  

\subsubsection{Local embeddings and maximal orders}  Let $R$ be a noetherian integrally closed domain, with quotient field $\KK$, and let $\Sigma$ be a finite dimensional central simple $\KK$-algebra.  A \emph{full $R$-lattice in $\Sigma$} is a finitely generated $R$-submodule $M \subseteq \Sigma$ such that if
$$
\KK\cdot M := \left \{\sum_{\text{finite}} \alpha_i s_i : \alpha_i \in \KK, s_i \in M \right \},
$$
then
$
\KK \cdot M = \Sigma.
$
An \emph{$R$-order} in $\Sigma$, is a subring $\Lambda \subseteq \Sigma$, with $1_{\Lambda} = 1_{\Sigma}$, and such that $\Lambda$ is a full $R$-lattice in $\Sigma$.  This definition is the affine local version of that which was given in \S \ref{Brauer:ramification}.  The \emph{right order} of a full $R$-lattice $M \subseteq \Sigma$ is:
$$
\mathrm{O}_r(M) := \{x \in \Sigma : Mx \subseteq M \}.
$$
It is an $R$-order in $\Sigma$.

We recall the following fact for later use.

\begin{lemma}[{\cite[Exercise 1, p.~112]{Reiner:2003}}]\label{Reiner:Excercise:p.112}
Let $\Gamma$ be a subring of $\Sigma$ which contains $R$ and suppose further that $\Gamma$ is finitely generated as an $R$-module.  Then, for every full $R$-lattice $M$ in $\Sigma$, the set
$$
M \cdot \Gamma := \left  \{ \sum_{\mathrm{finite}} s_i \gamma_i :  s_i \in M, \gamma_i \in \Gamma \right \}
$$
is a full $R$-lattice in $\Sigma$ and $\Gamma$ is contained in the right order $\mathrm{O}_r(M\cdot \Gamma)$. 
\end{lemma}
\begin{proof}
Clearly $\Gamma$ is contained in $\mathrm{O}_r(M\cdot\Gamma)$ and also $\KK \cdot (M \cdot \Gamma) = \Sigma$.  \end{proof}

Lemma \ref{Reiner:Excercise:p.112} implies that embeddings of central simple algebras induce embeddings of maximal orders.  
\begin{proposition}\label{embedding:prop2} Suppose given two finite dimensional $\KK$-central simple algebras $\Sigma_1$ and $\Sigma_2$ together with a $\KK$-algebra embedding $\Sigma_1 \hookrightarrow \Sigma_2$.  Further, fix an order $\Lambda_1$ in $\Sigma_1$.
Then, under this hypothesis,
 there exists an order $\Lambda_2$ in $\Sigma_2$ together with an $R$-algebra embedding $\Lambda_1 \hookrightarrow \Lambda_2$ which fits into a commutative diagram:
\begin{equation}\label{embedding:diagram}
\begin{tikzcd}
\Lambda_1 \arrow[hook]{r} \arrow[hook]{d}
               & \Lambda_2 \arrow[hook]{d}\\
\Sigma_1 \arrow[hook]{r} 
& \Sigma_2 
\end{tikzcd}
\end{equation}
of $R$-algebras.
\end{proposition}
\begin{proof} 
We identify $\Lambda_1$ and $\Sigma_1$ with their respective image in $\Sigma_2$.  Then $\Lambda_1$ is a subring of $\Sigma_2$ which contains $R$ and which is a finitely generated $R$-module.  Let $M$ be a full $R$-lattice in $\Sigma_2$.  Then:
$$M \cdot \Lambda_1 = \left  \{ \sum_{\text{finite}} \alpha_i s_i : \alpha_i \in M, s_i \in \Lambda_1\right \} $$
is a full $R$-lattice in $\Sigma_2$ and, by Lemma~\ref{Reiner:Excercise:p.112},  
$\Lambda_1 \subseteq \mathrm{O}_r(M\cdot \Lambda_1).$
Further,  the diagram \eqref{embedding:diagram} applied to the order 
$\Lambda_2 = \mathrm{O}_r(M\cdot\Lambda_1) $
commutes.
\end{proof}

One consequence of Proposition~\ref{embedding:prop2} reads:
\begin{corollary}\label{embedding:cor2} 
With the same assumptions as Proposition~\ref{embedding:prop2}, if $\Lambda_1$ is a maximal order in  $\Sigma_1$, then there exists a maximal order $\Lambda_2$ in $\Sigma_2$ so that the diagram \eqref{embedding:diagram} commutes.
\end{corollary}
\begin{proof}
Combine Proposition~\ref{embedding:prop2} with the fact that every order is contained in a maximal order.
\end{proof}

We make use of the following remark in Example \ref{complete:local:embedding:remark} and Corollary \ref{complete:local:embeddings:division:algebras}.

\begin{proposition}\label{prelims:full:lattice:orders:non-unit}
  Given a maximal order $\Lambda_1$ and a diagram as in \eqref{embedding:diagram}
 where $\Sigma_1=\mathcal{D}_1 \text{ and } \Sigma_2=\mathcal{D}_2$ are division algebras, if $a \in \Lambda_1$ is a non-unit, then its image in $\Lambda_2$ is a non-unit.
\end{proposition}
\begin{proof}
Suppose that the assertion is false.  Then there exists a non-unit $a \in \Lambda_1$ which is a unit in $\Lambda_2$ and such an $a$ must be a unit in $\mathcal{D}_1$ too. Let
$$\Gamma := \Lambda_2 \cap \mathcal{D}_1 \subseteq \Lambda_2.$$
Then 
 $\Gamma$ is a finitely generated $R$-module since $\Lambda_2$ is finitely generated as an $R$-module and because $R$ is Noetherian.  Further,
$\Lambda_1 \subseteq \Gamma$ so $\Gamma$
is a full $R$-lattice in $\mathcal{D}_1$.  So $\Gamma$ is an $R$-order in $\mathcal{D}_1$ which contains $a^{-1}$ and which contains $\Lambda_1$ too.  Since $\Lambda_1$ is a maximal order, we must have $\Lambda_1 = \Gamma$ 
and so $a^{-1} \in \Lambda_1$ which contradicts the fact that $a \in \Lambda_1$ is assumed to be a non-unit.
\end{proof}

\subsection{Behaviour of ramification under embeddings}\label{ramification:embeddings}  For the applications that we have in mind, it is of interest to  determine the nature of ramification under embeddings of  central simple algebras.  This is the content of Proposition \ref{ramification:division}.  In what follows, we fix a discrete valuation ring $(R,\mathfrak{m})$ with fraction field $\KK$, residue field $\kappa$ and uniformizer $t \in R$.  Let $(\hat{R}, \hat{\mathfrak{m}})$ be the completion of $(R,\mathfrak{m})$ with fraction field $\hat{\KK}$.

\subsubsection{Local embeddings of central simple algebras} Here, we explore the concept of local embeddings of central simple algebras together with the question of divisibility of ramification indices.  First, we make precise what we mean by local embeddings of central simple algebras.

\begin{defn}\label{local:embedding:defn}  
By a \emph{local embedding} of $\KK$-central simple algebras $\Sigma_1$ and $\Sigma_2$, we mean an embedding $\Sigma_1 \hookrightarrow \Sigma_2$ which fits into a commutative diagram of $R$-algebras of shape:
$$
\begin{tikzcd}
\operatorname{rad} \Lambda_1 \arrow[hook]{r} \arrow[hook]{d}
               & \operatorname{rad} \Lambda_2 \arrow[hook]{d}\\
\Lambda_1 \arrow[hook]{r} \arrow[hook]{d}
               & \Lambda_2 \arrow[hook]{d}\\
\Sigma_1 \arrow[hook]{r} 
& \Sigma_2 
\end{tikzcd}
$$
for some maximal orders $\Lambda_i$ contained in the central simple algebras $\Sigma_i$, for $i = 1,2$.  
\end{defn}

\begin{example}\label{complete:local:embedding:remark}
When $(R,\mathfrak{m})$ is complete, then every embedding of division algebras is local in this sense as is a consequence of Proposition \ref{prelims:full:lattice:orders:non-unit}.    Indeed, if $\mathfrak{D}_1 \hookrightarrow \mathfrak{D}_2$ is such and embedding, then the image in $\mathfrak{D}_2$ of a prime element $\pi_{\mathfrak{D}_1} \in \mathfrak{D}_1$ cannot be a unit.  This shows that the image of the radical of the unique maximal order $\mathfrak{R}_1 \subseteq \mathfrak{D}_1$ is contained in the radical of the unique maximal order $\mathfrak{R}_2 \subseteq \mathfrak{D}_2$.   Moreover, note that $\mathfrak{R}_2 \cdot \operatorname{rad} \mathfrak{R}_1 = \operatorname{rad} \mathfrak{R}_1 \cdot \mathfrak{R}_2$ since every one-sided ideal of $\mathfrak{R}_2$ is two-sided (\cite[Theorem 13.2]{Reiner:2003}). 
\end{example}

That the concept of local embedding for central simple algebras relates to Proposition \ref{ramification:division} is the content of Theorem \ref{embedding:divide:theorem} below.

\begin{theorem}\label{embedding:divide:theorem}  Let $(R,\mathfrak{m})$ be a discrete valuation ring with fraction field $\KK$.   Let $\Sigma_1 \hookrightarrow \Sigma_2$ be 
an 
embedding 
of central simple algebras 
with respect to maximal orders $\Lambda_i \subseteq \Sigma_i$, for $i = 1,2$.  The following assertions are true.
\begin{enumerate}
\item{ If the embedding $\Sigma_1 \hookrightarrow \Sigma_2$ is a local embedding, then the  ramification index of $\Lambda_2$ is greater to or equal to the ramification index of $\Lambda_1$.
}
\item{ Assume that $\Lambda_2 \cdot \operatorname{rad} \Lambda_1 = \operatorname{rad} \Lambda_1 \cdot \Lambda_2$.  Then  the embedding is local and the ramification index of $\Lambda_2$ is divisible by the ramification index of $\Lambda_1$. 
}
\end{enumerate}
\end{theorem}

\begin{proof}
 By the assumption of local embedding,  the two sided ideal of ${\Lambda}_2$ which is generated by the image of  the ideal $\operatorname{rad} {\Lambda}_1$ is not the unit ideal.  On the other hand, by \cite[Theorem 18.3, p.~176]{Reiner:2003}, we may write this two sided ideal
$
{\Lambda}_2 \cdot  \operatorname{rad} {\Lambda}_1   \cdot {\Lambda}_2
$
as a power of $\operatorname{rad} {\Lambda}_2$:
\begin{equation}\label{rad:theorem:eqn1}{\Lambda}_2 \cdot \operatorname{rad} {\Lambda}_1   \cdot {\Lambda}_2 = (\rad {\Lambda}_2)^k, 
\end{equation}
for some nonnegative integer $k$.  Since  ${\Lambda}_2 \cdot \operatorname{rad} {\Lambda}_1 \cdot {\Lambda}_2$ is not the unit ideal, we must have that $k > 0$.   

Then, by \eqref{rad:theorem:eqn1}, we have that
\begin{equation}\label{rad:theorem:eqn3} {\Lambda}_2 \cdot (\rad \Lambda_1)^{e_1} \cdot {\Lambda}_2 \subseteq (\Lambda_2 \cdot   \operatorname{rad} {\Lambda}_1  \cdot {\Lambda}_2)^{e_1}  = (\rad {\Lambda}_2)^{k e_1}.
\end{equation}
Further
\begin{equation}\label{rad:theorem:eqn2}
(\rad {\Lambda}_i)^{e_i}  = \Lambda_i \cdot  \mathfrak{m}  \text{.}
\end{equation}
Combing these we get
\begin{equation}\label{rad:theorem:eqn5}
  (\rad {\Lambda}_2)^{e_2} = {\Lambda}_2 \cdot  \mathfrak{m}
  =  {\Lambda}_2 \cdot  (\rad \Lambda_1)^{e_1} \cdot {\Lambda}_2 \subseteq (\rad  \Lambda_2)^{k e_1};
 \end{equation}
so we conclude that
$ke_1 \leq e_2$;
whence  $e_1 \leq e_2$ as desired.   This establishes assertion (a).  

The first statement of (b),   is a special case of \cite[Exercise  6, p.~68]{Farb:Dennis:1993}.  Indeed, note that $ \Lambda_2 $ is finitely generated as a left $\Lambda_1$-module.  Further, by assumption, $\Lambda_2 \cdot \operatorname{rad} \Lambda_1 = \operatorname{rad} \Lambda_1 \cdot \Lambda_2$.  We then check, using Nakayama's lemma, that $\operatorname{rad} \Lambda_1$ acts trivially on each simple $\Lambda_2$-module.   Indeed, write 
$$\Lambda_2 = \Lambda_1\cdot x_1 + \dots + \Lambda_1 \cdot x_n\text{,}$$ where $x_i \in \Lambda_2$, for $i = 1,\dots,n$.  Let $M$ be a simple left $\Lambda_2$-module.  Write 
$$M = \Lambda_2 \cdot a\text{,}$$ 
for some $a \in M$.  Then
\begin{align*}
M & = \left( \Lambda_1 \cdot x_1 + \dots + \Lambda_1 \cdot x_n \right) \cdot a \\
& = \Lambda_1 x_1 a + \dots + \Lambda_1 x_n a \text{.}
\end{align*}
Thus, $M$ is finitely generated as a left $\Lambda_1$-module.  Now, consider $\operatorname{rad} \Lambda_1 \cdot M$.  Since $\Lambda_2 \cdot \operatorname{rad} \Lambda_1 = \operatorname{rad} \Lambda_1 \cdot \Lambda_2 $, this is a $\Lambda_2$-submodule of $M$.  Since $M \not = 0$, Nakayama's lemma implies that $\operatorname{rad} \Lambda_1 \cdot M \subsetneq M$.  But, since $M$ is a simple left $\Lambda_2$-module, it follows that $\operatorname{rad} \Lambda_1 \cdot M = 0$.
  Thus $\operatorname{rad} \Lambda_1 \subseteq \operatorname{rad} \Lambda_2$.  Moreover, since $\Lambda_2 \cdot \operatorname{rad} \Lambda_1 = \operatorname{rad} \Lambda_1 \cdot \Lambda_2$, equality holds in \eqref{rad:theorem:eqn3} and \eqref{rad:theorem:eqn5} whence $e_1 \mid e_2$.
\end{proof}

Theorem \ref{embedding:divide:theorem} has the following consequence.  

\begin{corollary}\label{complete:local:embeddings:division:algebras}
Let $(R,\mathfrak{m})$ be a complete discrete valuation ring with fraction field $\KK$ and suppose that $\mathcal{D}_1 \hookrightarrow \mathcal{D}_2$ is an embedding of $\KK$-central division algebras.  Then the ramification index of $\mathcal{D}_2$ is divisible by the ramification index of $\mathcal{D}_1$.
\end{corollary}
\begin{proof}
As noted in Example \ref{complete:local:embedding:remark},  the hypothesis of Theorem \ref{embedding:divide:theorem} (b) is satisfied.  Thus, Theorem \ref{embedding:divide:theorem} (b) implies the conclusion of Corollary \ref{complete:local:embeddings:division:algebras}.
\end{proof}

\subsubsection{Complete local numerical invariants} With a view towards Proposition \ref{ramification:division}, we first fix some notations for complete local invariants of central simple algebras.  More details about these invariants may be found in \cite[\S 18]{Reiner:2003}.
Let $\Sigma_2$ be a $\KK$-central  simple algebra and 
$\Sigma_1 \subseteq \Sigma_2$ 
a subalgebra with centre equal to $\KK$.  Then $\Sigma_1$ is a  central simple algebra and, by the centralizer theorem \cite[Corollary 7.14]{Reiner:2003}, there exists a central simple subalgebra 
 $\Sigma_3 \subseteq \Sigma_2$
which has the property that
$$
\Sigma_2 \simeq \Sigma_1 \otimes \Sigma_3.
$$
Henceforth, we assume that either the residue field $\kappa$ is perfect or that the index of central simple $\Sigma_2$ is prime to the characteristic of $\kappa$.  Let $m_i$ be the degree of $\Sigma_i$.  Then
\begin{equation}\label{eqn3}
m_2 = m_1 m_3.
\end{equation}
Fix maximal orders $\Lambda_i \subseteq  \Sigma_i$, for $i = 1,2,3$, and write  
$
\hat{\Sigma}_i \simeq \mathfrak{D}_i^{s_i \times s_i}
$
for $\hat{\KK}$-central division algebras $\mathfrak{D}_i$ with unique maximal orders $\mathfrak{R}_i$.  Then 
$
\hat{\Lambda}_i \simeq \mathfrak{R}_i^{s_i \times s_i}
\text{.}
$
Denote the respective centres of the division algebras
$
\Delta_i := \mathfrak{R}_i / \operatorname{rad} \mathfrak{R}_i  \text{, }
$
for $i = 1,2,3$, by $\kappa_i$.  Let $d_i$ be the index of $\mathfrak{D}_i$.  Then 
$
d_i^2 = [\mathfrak{D}_i : \hat{\KK}]
$
and
\begin{equation}\label{eqn7}
m_i = s_id_i.
\end{equation}
Recall, that 
the ramification indices $e_i$ have the property that
$
e_i = [\kappa_i : \kappa] \text{, }
$
by Theorem \ref{ramification:data:theorem1}.
Let $g_i$ be the index of $\Delta_i$.  Since
$$
g_i^2 = [\Delta_i : \kappa_i] \text{, }
$$
it follows that
\begin{equation}\label{eqn10}
f_i^2 := \dim_{\kappa_i} \hat{\Lambda}_i / (\operatorname{rad} \hat{\Lambda}_i) = s_i^2 g_i^2.
\end{equation}

\subsubsection{A characterization for divisibility of ramification for embeddings of division algebras}
Here, we establish Proposition \ref{ramification:division} below.  It pertains to the behaviour of ramification under embeddings of central simple algebras.  We are extremely grateful to an anonymous referee for providing comments about completions of division algebras.  
These more subtle points are partly explained in Proposition \ref{ramification:division} below.  

\begin{proposition}\label{ramification:division}  
Let $(R,\mathfrak{m})$ be a discrete valuation ring with fraction field $\KK$.  Let 
$\Sigma_2 \simeq \Sigma_1 \otimes \Sigma_3$ 
be a tensor product of $\KK$-central simple algebras $\Sigma_i$, for $i = 1,2,3$.  Assume that either the residue field $\kappa$ is perfect or that the index of $\Sigma_2$ is prime to the characteristic  of $\kappa$.
Then 
$
e_1 \mid e_2 \text{ if and only if } g_2 s_2 \mid g_1 s_1 d_3 s_3 \text{.}
$
Equivalently, it holds true that
$
e_1 \mid e_2 \text{ if and only if } f_2 \mid f_1 m_3.
$
\end{proposition}

\begin{proof}
By computing $\kappa$-dimensions and combining equations \eqref{eqn7} and \eqref{eqn10}, we obtain that
\begin{equation}\label{eqn11}
m_i = e_i f_i = e_i s_i g_i = s_i d_i.
\end{equation}
This means that we may rewrite equation \eqref{eqn3} in the form
\begin{equation}\label{eqn12}
m_2 = s_2 d_2 = e_2 s_2 g_2 = e_2 f_2 = e_1 f_1 e_3 f_3;
\end{equation}
expanding equation \eqref{eqn12}, we obtain: 
\begin{equation}\label{eqn13}
m_2 = s_2 d_2 = e_2 s_2 g_2 = e_1 s_1 g_1 e_3 s_3 g_2 = s_1 d_1 s_3 d_3.
\end{equation}
One other consequence of \eqref{eqn11} and \eqref{eqn7} is the fact that
$
d_i = e_i g_i.
$

We may thus rewrite equation \eqref{eqn13} in the form
\begin{equation}\label{eqn15}
e_2 s_2 g_2 = d_2 s_2 = s_1 s_3 d_1 d_3 = e_1 e_3 s_1 g_1 s_3 g_3.
\end{equation}
Recall that 
$
\hat{\Sigma}_i = \mathfrak{D}_i^{s_i \times s_i}  
$, 
with each $\mathfrak{D}_i$ a division algebra, 
and
$
\hat{\Sigma}_2 = \hat{\Sigma}_1 \otimes \hat{\Sigma}_3.
$
Thus:
\begin{equation}\label{eqn18}
\mathfrak{D}_2^{s_2 \times s_2} = \mathfrak{D}_1^{s_1 \times s_1} \otimes \mathfrak{D}_3^{s_3 \times s_3} \text{.}
\end{equation}
But also
\begin{equation}\label{eqn19}
\mathfrak{D}_1^{s_1 \times s_1} \otimes \mathfrak{D}_3^{s_3 \times s_3} \simeq (\mathfrak{D}_1 \otimes \mathfrak{D}_3)^{s_1 s_3 \times s_1 s_3}.
\end{equation}
equation \eqref{eqn19} then implies that
$
d_2 s_2 = d_1 s_1 d_3 s_3.
$
In other words,
$$
d_2 = d_1 d_3 \left( \frac{s_1 s_3}{s_2} \right) \text{.}
$$
Finally, solving for $e_2$, yields:
$$
e_2 = e_1 \left( \frac{g_1 d_3 s_1 s_3}{g_2 s_2}  \right)
= 
e_1 \left(\frac{f_1 m_3}{f_2}\right) \text{.}
$$
\end{proof}

In what follows, in the context of Proposition \ref{ramification:division}, when $e_1 \mid e_2$, we put:
\begin{equation}\label{embedding:divide:theorem:notation}
 e_{\Sigma_2 / \Sigma_1} := \frac{e_2}{e_1}.
\end{equation} 

We conclude our discussion of divisibility of ramification indices by mentioning the following result which was suggested to us by an anonymous referee.  

\begin{proposition}\label{period:index}
  Let $(R,\mathfrak{m})$ be a discrete valuation ring with field of fractions $\KK$ and residue field $\kappa$.
Let $\Sigma_1 \hookrightarrow \Sigma_2$ be an embedding of central simple algebras with centres $\KK$.
    Assume that either the residue fields $\kappa_\mathfrak{p}$ are perfect for all divisorial places $\mathfrak{p}$ of $\KK,$ or that the index of $\Sigma_2$ is prime to the characteristic of $\kappa$.  Suppose further  that the period of $\Sigma_2$ is equal to its index
we have that
     $$e_\mathfrak{p}(\Sigma_1) \mid e_\mathfrak{p}(\Sigma_2).$$
\end{proposition}
\begin{proof}[Proof of Propositions \ref{perinduni} and \ref{period:index}]  Proposition \ref{perinduni}, is clearly implied by  Proposition \ref{period:index}, so it suffices to establish this last result. 
To this end, 
the centralizer theorem implies that 
$$\Sigma_2 \simeq \Sigma_1 \otimes_\KK \Sigma_3$$
for a $\KK$-central  central simple algebra  $\Sigma_3$. 
Let $p_i,m_i$ for $i = 1,2,3$, denote the period and index of  $\Sigma_i$.
By assumption, $p_2 = m_2$ and $p_i$ is the order of the class of $\Sigma_i$ in $\Br \left( \KK \right)$.  Moreover, Theorem \ref{explicit:ramification:map:theorem} implies that $e_i$ is the order of the image of this class under the ramification map 
$$
a \colon \operatorname{Br}(\KK) \rightarrow \H^1(\kappa ,\QQ/\ZZ).
$$ 
In particular, for $i=1,2,3$, $e_i \mid p_i$, and~\cite[Proposition~4.5.13]{Gille:Szamuley:2006} implies that $p_i \mid m_i$. 
Thus, the following divisibility relations hold true:
$$
p_2 \mid \lcm(p_1,p_3) \mid p_1 p_3 \mid m_1 m_3 = m_2;
$$
on the other hand, by assumption,  $p_2 = m_2$, whence equality holds everywhere:
$$
p_2 = \lcm(p_1,p_3) = p_1 p_3 = m_1 m_3 = m_2.
$$
Thus:
$$
\gcd(p_1,p_3) = 1
$$  
and:
 $$
 e_i \mid p_i\text{,}
 \gcd(e_1,e_3) = 1\text{ and }  
 e_2 = e_1 e_3 \text{.}
 $$
\end{proof}

\begin{example}\label{embedding counter:example}
Here, we illustrate the fact that not every embedding of division algebras 
of the form considered in Proposition \ref{ramification:division}, has the property that $e_1 \mid  e_2$.
For a field $\mathbf{E}$ and $\alpha, \beta \in \mathbf{E}$, write 
$(\alpha,\beta)$ for the quaternion algebra with basis $1,s,t,st$ over $\mathbf{E}$ and relations
$$
s^2 = \alpha, t^2 = \beta \text{ and } st = - t s \text{.}
$$ 
Now, let $\KK = \kk(z_1,z_2,z_3,z_4)$
be a purely transcendental extension of an algebraically closed base field $\kk$ of characteristic $\neq 2$.  
As in \cite[Example 1.5.7]{Gille:Szamuley:2006}, consider the biquaternion division algebra
$$
\mathcal{D} = (z_1,z_2) \otimes_{\KK} (z_3,z_4) \text{.}
$$
This division algebra has index $4$ and period equal to $2$.  By setting 
$$
\text{
$z_1 = vx$, $z_2 = u$, $z_3 = vy$ and $z_4 = u+v$, 
}
$$
we may identify $\KK \simeq \kk(u,v,x,y)$.
Via this isomorphism, we deduce that
$$
\mathcal{D}_2 := (vx,u) \otimes_\KK (vy, u+v)
$$
is a division algebra with index $4$ and period $2$.
Note now that, when working over the discrete valuation
 ring $R_{\mathfrak{m}(v)}$,
we have the ramification map in equation (\ref{ramification:group:homomorphism:defn}), below,
where $a(vx,u)$ and $a(vy,u+v)$ 
both correspond to the double cover
$\kappa(\sqrt{u})$ of the residue field $\kappa$ at ${\mathfrak{m}(v)}$. 
So $\mathcal{D}_1 = (vx, u)$ 
has ramification index $e_1 = 2$,  but $\mathcal{D}_2$ is not ramified at the place $\mathfrak{m}(v)$ so we
do not have the property that $e_1 \mid e_2$. 
\end{example}

\subsection{Galois theoretic interpretation of divisibility of ramification}\label{Galois:divisible:ramifcation}  We now give 
  an alternative interpretation of
  effective embeddings of division algebras in terms of Galois cohomology. 
Let $(R,\mathfrak{m})$ be a discrete valuation ring
 with fraction field $\KK$ and residue field $\kappa$.
Recall, that there exists the well known \emph{ramification group homomorphism}:
\begin{equation}\label{ramification:group:homomorphism:defn}
a \colon \Br(\KK) \rightarrow  \H^1(\kappa,\QQ/\ZZ)
\end{equation}
as in \cite[Chapter 10]{Saltman:1999}; compare with \eqref{ramification:group:homomorphism:defn'}.   Recall, also the following bijective correspondence:
$$\H^1(\kappa,\QQ/\ZZ) = \Hom_{\mbox{cont}}(\Gal(\overline{\kappa}/\kappa),\QQ/\ZZ) \leftrightarrow \{ (L,\pi) : \Gal(L : \kappa) = \la \pi \ra\}$$
where $\overline{\kappa}$ is the separable closure of $\kappa.$
This correspondence maps 
$\phi \in \H^1(\kappa,\QQ/\ZZ) $ 
to the extension 
$L = \overline{\kappa}^{\ker \phi}$ 
with 
$\phi(\pi) = 1/[L:\kappa] \in \QQ/\ZZ$.

Fix 
$\alpha_1,\alpha_3 \in \Br(\KK)$
and put 
$\alpha_2 := \alpha_1 \otimes \alpha_3\text{.}$  Let $e_i$ be the ramification index of $\alpha_i$.  We want to understand the condition that $e_1 \nmid e_2$.  The following proposition gives an equivalent  formulation of this condition. 
It is expressed using the tools of Galois cohomology.  To this end, we assume that either the residue field $\kappa$ is perfect or that the degree of $\alpha_2$ is prime to the characteristic of $\kappa$.

\begin{proposition}\label{e1:not:divide:e3}
Let $\kappa_i$ be the cyclic extension of $\kappa$ determined by the class $\alpha_i \in \Br(\KK)$.  Let $\sigma_i$ be the generators for the Galois group $\Gal(\kappa_i / \kappa)$ determined by $\alpha_i$, for $i = 1,3$.  Then $e_1 \nmid e_2$ if and only if there exists a prime number $p$ which has the following two properties:
\begin{enumerate}
\item{
  $e_1 = p^{\ell} u$, $e_3 = p^{\ell} v$, for some $\ell > 1$ so that
  $p \nmid uv$
and
}
\item{
the $p$ part of the order of 
$\sigma_1 \cdot \sigma_3 |_{\kappa_1 \bigcap \kappa_3}$ 
in $\Gal(\kappa_1 \bigcap \kappa_3 / \kappa)$ is strictly smaller than $[\kappa_1 \cap \kappa_3 : \kappa]$.
}
\end{enumerate}
\end{proposition}
\begin{proof}
Without loss of generality, we reduce to the case that the cyclic extension $\kappa_i$, for $i = 1,2,3$ are powers of $p$.  Let
$
e_i := p^{\ell_i} := [\kappa_i : \kappa] \text{, }
$
for $i = 1,2,3$. If $e_1 \nmid e_2$ then we are in case (d) of Lemma~\ref{p:e1:not:divide:e3} below.
Using the notation there, we see that the order 
of $\sigma_1 \cdot \sigma_3 |_{\kappa_1 \bigcap \kappa_3}$ 
in $\Gal(\kappa_1 \bigcap \kappa_3 / \kappa)$ is given by $p^{k-a}$. So 
if $e_1 \nmid e_2$, then $a \geq 1$. 
\end{proof}

\begin{lemma} \label{p:e1:not:divide:e3} Let $\kappa$ be a field with absolute Galois group $G$.  Let $\phi_1,\phi_3 \in  \Hom_{\operatorname{cont}}(G,\ZZ[p^{-1}]/\ZZ).$
  Let  $\phi_2 = \phi_1+\phi_3$ and $e_i =[G:\ker \phi_i]$.
\begin{enumerate}
\item{
If $e_1 \not = e_3$, then $e_2 = \lcm(e_1,e_3)$;
}
\item{
If $e_1 = e_3$ and $\kappa_1 \cap \kappa_3 = \kappa$, then $e_2 = e_1 = e_3$;
}
\item{
  If $e_1 = e_3$, $\kappa \neq   \kappa_1 \cap  \kappa_3$ and $\pi_1|_{\kappa_1 \cap \kappa_3} = \pi_3^c|_{\kappa_1 \cap \kappa_3}$, with $c \not \equiv -1\pmod{p},$ then
  $e_2 = e_1 = e_3$; and
}
\item\label{caseiv}{
  If $e_1 = e_3 =p^\ell$, $\kappa  \not = \kappa_1 \cap  \kappa_3, [\kappa_1 \cap  \kappa_3:\kappa] = p^k$ and $c \equiv -1\pmod{p^a},$ with $k \geq a>0$ smallest possible,
  then 
  $e_2 = p^{\ell-a} < e_1 = e_3$.
}
\end{enumerate}
\end{lemma}
\begin{proof}
Let 
$\kappa_i$ be the subextension of $\overline{\kappa}$ corresponding to $\ker \phi_i$.  Write $\kappa_1\kappa_3$ for the composite extension.  Note that $\kappa_1\kappa_3$
corresponds to $\ker \phi_1 \cap \ker \phi_3$, which is contained in
$\ker( \phi_1+\phi_3)$.  So to compute the order of $G/\ker( \phi_1+\phi_3)$ it is sufficient to work in the Galois group
$G' = \Gal(\kappa_1\kappa_3/\kappa) = G/(\ker\phi_1 \cap \ker \phi_3).$  Since we have abelian
extensions, we obtain a diagram of Galois groups
 $$\begin{CD} G' & @>>> & \Gal(\kappa_1) \\
      @VVV & & @V{\psi_1}VV \\
      \Gal(\kappa_3) & @>{\psi_3}>> &  \Gal(\kappa_1 \cap \kappa_3) \end{CD}$$
Since $\la \pi_i \ra = \Gal(\kappa_i)$, we can conclude that
$\la \pi_1,\pi_3 \ra = G'$.  Let $G'' =  \Gal(\kappa_1 \cap \kappa_3)$
which is also cyclic and generated by the image of $\pi_1$ or $\pi_3.$
Let $p^{\ell_i} = |\Gal(\kappa_i)|$, and let $p^{k} = |G''|.$
Let $c \in (\ZZ/p^k)^\times$ be such that $\psi_1(\pi_1) = \psi_3(\pi_3)^c$.
So we see that $G'$ is the pullback
$$G' = \{ \pi_1^i\pi^j_3 : (i,j) \in \ZZ/p^{\ell_1}\oplus\ZZ/p^{\ell_3} \mbox{ and } i \equiv cj \pmod{p^k}\}.$$  We will abbreviate elements $\pi_1^i\pi^j_3$
of $G'$ as $(i,j)$.
Note that
$$\phi_1(i,j) = i/p^{\ell_1} \text{ and } \phi_3(i,j) = j/p^{\ell_3}$$
in $\QQ/\ZZ.$
So $$(\phi_1+\phi_3)(i,j) = i/p^{\ell_1}+j/p^{\ell_3}.$$
First note that if $\ell_1 \neq \ell_3$, say $\ell_3> \ell_1$, then the order of
$(\phi_1+\phi_3)(c,1)$ is $p^{\ell_3}$ which is the largest possible, and so $e_2 = |G'/\ker(\phi_1+\phi_3)| = e_3$.
So in this case $e_2 = \lcm(e_1,e_3)$.

Now suppose that $\ell = \ell_1 = \ell_3.$ In this case
 $$(\phi_1+\phi_2)(i,j) = \frac{i+j}{p^{\ell}}.$$
If $\kappa_1 \cap \kappa_3 = \kappa$ then $k=0$ and so $(0,1) \in G'$ then as above, we have $e_3 = p^\ell = e_2 = e_1$, so we must consider the case where $(0,1) \not\in G'$ and $k >0$.
In this case the elements of  $G'$ are of the form  $(cj,j)\pmod{p}$ so in order to have $\ker(\phi_1 +\phi_3)$ correspond to a degree $p^\ell$ extension,
we need that $cj+j \equiv 0 \pmod{p}$ and so $c\equiv-1 \pmod{p}$.  Lastly, if
$c \equiv -1 \pmod{p^a}$ with $a >0$ minimal, then $c \equiv -1+xp^{a} \pmod{p^{a+1}}$ and so $(c,1) = (-1+xp^{a},1)$ is an element of $G'$ with $(\phi_1+\phi_3)(c,1)$ of order $p^{\ell - a}$. 
\end{proof}

\begin{remark}
As is a consequence of Proposition \ref{e1:not:divide:e3}, one way to  achieve a cancellation of ramification is when
we have the same cyclic extension $L$ of $\kappa$ with
inverse generators.  So $\phi_1,\phi_2$ correspond to
$(L,\sigma)$ and $(L,\sigma^{-1}).$
\end{remark}

 \subsection{Global embeddings and ramification}\label{global:ramification:embedding:formula}    Let $\kk$ be an algebraically closed characteristic zero field and $X$ a normal proper variety with function field $\KK = \kk(X)$.  Fix two $\KK$-central  simple algebras $\Sigma_1$ and $\Sigma_2$ and an $X$-effective embedding $\Sigma_1 \hookrightarrow \Sigma_2$ over $\KK$.  Recall, that the concept of  $X$-effective embedding means that the  ramification of $\Sigma_2$ is greater or equal to the ramification of $\Sigma_1$ at all codimension $1$ points $\mathfrak{p}$ of $X$. 
Further, denote by $\alpha_1$ and $\alpha_2$, respectively, their classes in the Brauer group $\Br(\KK)$.   Also, given a prime divisor $\mathfrak{p}$ of $X$, we denote by
$$
e_{\alpha_2 / \alpha_1}(\mathfrak{p}) := e_{\Sigma_2 / \Sigma_1}(\mathfrak{p})
$$
the number given by \eqref{embedding:divide:theorem:notation} at $\mathfrak{p}$.  Finally, let $\Delta_{\alpha_1}$ and $\Delta_{\alpha_2}$ denote the ramification divisors determined by $\alpha_1$ and $\alpha_2$, respectively, and we also denote by  $\K_{\alpha_1}$ and $\K_{\alpha_2}$  the corresponding log canonical divisors.  Recall that $e_{\mathfrak{p}}(\alpha_i)$ denotes the ramification of $\alpha_i$ at a prime divisor $\mathfrak{p}$.

\begin{theorem}\label{embedding:iitaka:theorem} Let $\KK = \kk(X)$ be the function field of an integral normal proper variety $X$ over an algebraically closed field $\kk$.   Fix two $\KK$-central  simple algebras $\Sigma_1$ and $\Sigma_2$  with respective Brauer classes $\alpha_1, \alpha_2 \in \Br(\KK)$, and an embedding $\Sigma_1 \hookrightarrow \Sigma_2$ over $\KK$.  Assume that the index of $\Sigma_2$ is prime to the characteristic of $\kk$ and 
  $\Sigma_1 \leq_X \Sigma_2$.
 Let $\Delta_{\alpha_i}$ and $\K_{\alpha_i}$, denote, respectively, the ramification and log canonical divisors on $X$ determined  by $\alpha_i$, for $i = 1,2$.
It  then holds true that:
$$ \Delta_{\alpha_1} \leq \Delta_{\alpha_2},$$
$$ \K_{\alpha_1} \leq \K_{\alpha_2}$$
Further, if the ramification divisors $\Delta_{\alpha_i}$, for $i=1,2$, determine $\QQ$-Gorenstein pairs $(X,\Delta_{\alpha_i})$, then the inequality of Iitaka dimensions
$$ \kappa(X,\alpha_1) \leq \kappa(X,\alpha_2)$$
holds true.
If we have $\Sigma_1 \mid_X \Sigma_2$ then
$$
\K_{\alpha_2} = \K_{\alpha_1} + \sum\limits_{\mathfrak{p}}\frac{1}{e_{\mathfrak{p}}(\alpha_1)}\left(1 - \frac{1}{e_{\alpha_2/\alpha_1}(\mathfrak{p})} \right)\mathfrak{p}.
$$
\end{theorem}
We use the following immediate lemma in the proof of Theorem~\ref{embedding:iitaka:theorem}.
\begin{lemma}\label{easyLemma'}  Let $X$ be an integral normal proper variety over $\kk$.  Let $D$ be a $\QQ$-Cartier divisor on $X$ and let $E$ be an effective $\QQ$-Cartier divisor.  Then
$\kappa(X,D) \leq \kappa(X,D+E)$.
\end{lemma}
\begin{proof}
By \eqref{prelim:normal:var:eqn6} we may assume that $D$ and $D+E$ are integral.  Then the assertion follows from the inclusion $\Osh_X(D) \hookrightarrow \Osh_X(D + E)$.
\end{proof}

\begin{proof}[Proof of Theorem~\ref{embedding:iitaka:theorem} and Proposition \ref{embedding:main:result1}]
Note that our hypothesis implies that the 
ramification of  $\Sigma_2$ is divisible by the ramification of $\Sigma_1$ at all prime divisors $\mathfrak{p}$ of $X$.  The fact that $\Delta_{\alpha_1} \leq \Delta_{\alpha_2}$ then follows from the definition given in \eqref{brauer:boundary:defn} since, by  divisibility of ramification indicies, we have that: 
\begin{equation}\label{embedding:iitaka:theorem:key:eqn} 1-\frac{1}{e_{\mathfrak{p}}(\alpha_2)} = 1 - \frac{1}{e_{\alpha_2 / \alpha_1}(\mathfrak{p}) e_{\mathfrak{p}}(\alpha_1) } \geq 1 - \frac{1}{e_{\mathfrak{p}}(\alpha_1)}
\end{equation}
at each prime divisor $\mathfrak{p}$ of $X$.
We note
that, as is a consequence of \eqref{embedding:iitaka:theorem:key:eqn}, the precise relation amongst $\K_{\alpha_1}$ and $\K_{\alpha_2}$ is given by:
$$
\K_{\alpha_2} = \K_{\alpha_1} + \sum\limits_{\mathfrak{p}}\frac{1}{e_{\mathfrak{p}}(\alpha_1)}\left(1 - \frac{1}{e_{\alpha_2/\alpha_1}(\mathfrak{p})} \right)\mathfrak{p}$$
and so $\K_{\alpha_2} - \K_{\alpha_1}$ is effective.

So by Lemma~\ref{easyLemma'} we obtain
$$\kappa(X,\K_{\alpha_1}) \leq \kappa(X,\K_{\alpha_2})$$ 
as desired.
\end{proof}

\subsection{A Riemann-Hurwitz type theorem for ramification divisors determined by Brauer classes}\label{R-H-Brauer}  Here, we study the behaviour of ramification divisors under field extensions.  
 
\subsubsection{Ramification of Brauer classes and extensions of discrete valuation rings}\label{maximal:orders:dvr:extensions}  Fix a primitive $m$th root of unity $\zeta_m$ and let $\mu_m$ be the multiplicative group of $m$th roots of unity.  Let $(R,\mathfrak{m})$ be a characteristic zero discrete valuation ring, with residue field $\kappa = \kappa(\mathfrak{m})$,   and assume that  $(m, \operatorname{char}(\kappa)) = 1$ if $\operatorname{char}(\kappa) > 0$.  Again, this assumption $(m, \operatorname{char}(\kappa)) = 1$ has to do with existence of unramified splitting fields.   Let $\KK$ be the fraction field of $R$, $\hat{R}$ the completion and $\hat{\KK}$ its field of fractions.  Let $\FF / \KK$ be a finite dimensional field extension, $R'$ the integral closure of $R$ in $\FF$ and fix a maximal ideal $\mathfrak{m}'$ of $R'$.   The natural extension of discrete valuation rings
$
(R,\mathfrak{m}) \hookrightarrow (R'_{\mathfrak{m}'},\mathfrak{m}')
$
induces an extension of their completions
$
(\hat{R},\hat{\mathfrak{m}}) \hookrightarrow (\hat{R}'_{\mathfrak{m}'}, \hat{\mathfrak{m}}') \text{.}
$
Let $\hat{\FF}$ be the completion of $\FF$ with respect to $\mathfrak{m}'$ and $\kappa(\mathfrak{m}')$ the residue field.  Similar to \cite[Chapter XII, Exercise 2]{Serre:local:fields}, see also 
\cite[Theorem 10.4]{Saltman:1999},
consider the commutative  diagram
\begin{equation}\label{restriction:ramification:commutative:diagram}
\begin{tikzcd}
\Br(\KK)[m] \arrow[r] \arrow[d, "\operatorname{Res}(\cdot)"] & \H^2(\hat{\KK},\mu_m) \arrow[r] \arrow[d, , "\operatorname{Res}(\cdot)"]  &\H^1(\kappa(\mathfrak{m}),\ZZ/m\ZZ) \arrow[d,"e_{\mathfrak{m}' / \mathfrak{m}} \cdot \operatorname{Res}(\cdot) "] \\
\Br(\FF)[m] \arrow[r] & \H^2(\hat{\FF},\mu_m) \arrow[r] & \H^1(\kappa(\mathfrak{m'}),\ZZ/m\ZZ)
\end{tikzcd}
\end{equation}
Here, the vertical maps are induced by the natural pullback (restriction) maps.  The integer 
$$e_{\mathfrak{m}'/\mathfrak{m}} := [\Gamma_{\mathfrak{m}'} : \Gamma_{\mathfrak{m}}]$$ 
denotes the \emph{ramification index} of the field extension $\FF / \KK$ with respect to $\mathfrak{m}' \mid \mathfrak{m}$.  Recall that, if 
$$f_{\mathfrak{m}'/\mathfrak{m}} := [\kappa(\mathfrak{m}') : \kappa(\mathfrak{m})]$$ 
denotes the \emph{residue degree}, then
$$
e_{\mathfrak{m}'/\mathfrak{m}} \cdot f_{\mathfrak{m}'/\mathfrak{m}} = [\hat{\FF}:\hat{\KK}] \text{.}
$$
In the diagram \eqref{restriction:ramification:commutative:diagram},  composition in each of the horizontal rows yields the respective  ramification group homomorphisms $a_{\mathfrak{m}}(\cdot)$ and $a_{\mathfrak{m}'}(\cdot)$ as in \eqref{ramification:group:homomorphism:defn'} or \eqref{ramification:group:homomorphism:defn}.  Further, fixing an $m$-torsion Brauer class
$
\alpha \in \Br(\KK)[m]
$
and letting 
$$
\alpha' := \operatorname{Res}(\alpha) = \FF \otimes_{\KK} \alpha \in \Br(\FF)[m]
$$
be its pullback, then the above diagram \eqref{restriction:ramification:commutative:diagram} implies that
\begin{equation}\label{restriction:ramfication:relation}
a_{\mathfrak{m}'}(\alpha') = e_{\mathfrak{m}'/\mathfrak{m}} \operatorname{Res}(a_{\mathfrak{m}}(\alpha)) \text{.}
\end{equation}

\subsubsection{Riemann-Hurwitz type theorems}\label{R-H-Brauer-thms}  Let $X$ be a normal projective variety over an algebraically closed characteristic zero field $\kk$, with function field $\KK = \kk(X)$, let $\FF / \KK$ be a finite dimensional  field extension and $f \colon X' \rightarrow X$ the normalization of $X$ in $\FF$.  Recall, that the morphism $f$ is flat in codimension $1$.  
First, we make some comments in regards to \S \ref{maximal:orders:dvr:extensions} within our present global geometric context.  Fix a prime divisor $\mathfrak{p} \in X$ and let $\mathfrak{p}' \in X'$ have the property that $\mathfrak{p}' \mid \mathfrak{p}$.  Then
$$
e_{\mathfrak{p}' / \mathfrak{p}} \cdot f_{\mathfrak{p}' / \mathfrak{p}} = [\FF_{\mathfrak{p}'} : \KK_{\mathfrak{p}}] ;
$$
the extension $\FF / \KK$ is \emph{unramified} at $\mathfrak{p}' \mid \mathfrak{p}$ if 
$e_{\mathfrak{p}' / \mathfrak{p}} = 1$ 
and \emph{totally ramified} if 
$f_{\mathfrak{p}'/\mathfrak{p}} = 1$.  Fix $\alpha \in \operatorname{Br}(\KK)$ and let $\Delta_\alpha$ be the ramification divisor that it determines.  

In our formulation of Proposition \ref{brauer-r-h} below, we make use of the following concept which pertains to the nature of ramification under restriction.

\begin{defn}\label{weak:ramification}
With the notations and hypothesis of this section, we say that $\alpha \in \Br(\KK)$ is \emph{weakly ramified at all primes $\mathfrak{p} \in \operatorname{Supp}(\Delta_{\alpha})$ with respect to $f \colon X' \rightarrow X$}, the normalization of $X$ in $\FF$, if for all primes 
$\mathfrak{p} \in \operatorname{Supp}(\Delta_{\alpha})$ and all primes $\mathfrak{p}' \in X'$, with $\mathfrak{p}' \mid \mathfrak{p}$
either 
$e_{\mathfrak{p}' / \mathfrak{p}}(\FF / \KK) > 1$ 
or
$e_{\mathfrak{p}'}(\alpha') = e_{\mathfrak{p}}(\alpha)$.  When these conditions are satisfied, we also say that the Brauer class $\alpha$ has $f$  {\emph weak
  ramification} or, equivalently, that $f $ is \emph{weakly $\alpha$-ramified}.
\end{defn}

The following result shows that the conditions which are required by Definition \ref{weak:ramification} are satisfied in a variety of situations.

\begin{proposition}\label{weak:ramification:prop}
Let $X$ be a normal projective variety, over an algebraically closed characteristic zero field $\kk$, with function field $\KK = \kk(X)$, let $\FF / \KK$ be a finite dimensional Galois extension and let $f \colon X' \rightarrow X$ be the normalization of $X$ in $\FF$.  Fix $\alpha \in \Br(\KK)$ and let $\alpha' \in \Br(\FF)$ be its pullback.  Consider the following three conditions:
\begin{enumerate}
\item{the degree of $\FF / \KK$
and the index of $\alpha$ are relatively prime;
}
\item{
for all primes $\mathfrak{p}' \in X'$, if $\mathfrak{p} = f(\mathfrak{p}') \in \operatorname{Supp}(\Delta_{\alpha})$, then either 
$e_{\mathfrak{p}'/\mathfrak{p}} > 1$ or 
$(f_{\mathfrak{p}'/\mathfrak{p}}, e_{\mathfrak{p}}(\alpha))=1\text{;}$ 
}
\item{ the morphism
$f \colon X' \rightarrow X$ is weakly $\alpha$-ramified.
}
\end{enumerate}
Then condition (a) implies condition (b) and condition (b) implies condition (c).
\end{proposition}

\begin{proof}
In what follows, fix a prime $\mathfrak{p}' \in X'$ with $\mathfrak{p} = f(\mathfrak{p}') \in \operatorname{Supp}(\Delta_{\alpha})$.  Assume condition (a) and suppose that $e_{\mathfrak{p}'/\mathfrak{p}} = 1$.  
Since $\FF / \KK$ is Galois, 
$$
\#f^{-1}(\mathfrak{p}) \cdot e_{\FF / \KK} \cdot f_{\FF / \KK}  = [\FF : \KK ] = \deg f \text{,}
$$
for $e_{\FF / \KK} = e_{\mathfrak{p}'/\mathfrak{p}}$ and $f_{\FF / \KK} = f_{\mathfrak{p}'/\mathfrak{p}}$ (independent of $\mathfrak{p}'$).  
Thus
$f_{\mathfrak{p}'/\mathfrak{p}} \mid \operatorname{deg} f
$.   Now let $\hat{\alpha}$ denote the image of $\alpha$ in $\Br(\KK_{\mathfrak{p}})$.  Then the period of $\hat{\alpha}$ divides the period of $\alpha$ and,
 by assumption, the degree of $f$ is relatively prime to the index of $\alpha$.  Together, these assertions imply that 
$$
1  = (\deg f, \operatorname{index}(\hat{\alpha}) )
 = (f_{\mathfrak{p}'/\mathfrak{p}},\operatorname{index}(\hat{\alpha})) \text{.}
 $$
Finally, since
$
[\FF_{\mathfrak{p}'} : \KK_{\mathfrak{p}}] = e_{\mathfrak{p}'/\mathfrak{p}} \cdot f_{\mathfrak{p}'/\mathfrak{p}} 
$
and $e_{\mathfrak{p}'/\mathfrak{p}} = 1$, the above discussion implies that
\begin{equation}\label{weak:ram:eqn5}
(\operatorname{period} (\hat{\alpha}), [ \FF_{\mathfrak{p}'}:\KK_{\mathfrak{p}}]) = 1
\end{equation}
and
\begin{equation}\label{weak:ram:eqn6}
(e_{\mathfrak{p}}(\alpha), f_{\mathfrak{p}'/\mathfrak{p}}) = 1 \text{.}
\end{equation}
This last relation \eqref{weak:ram:eqn6} is a consequence of \eqref{weak:ram:eqn5} since the  ramification index
$
e_{\mathfrak{p}}(\alpha) 
$
divides the index of $\hat{\alpha}$ 
and because the index and period of $\hat{\alpha}$ have the same prime factors.  Thus (a) implies (b) as desired.

Now assume condition (b) and suppose that $e_{\mathfrak{p}' / \mathfrak{p}}=1$.  Let $m$ be the period of $\alpha$.  To establish (c), recall that the order of $e_{\mathfrak{p}'/\mathfrak{p}} \cdot \operatorname{Res}(a_{\mathfrak{p}}(\alpha))$ in $\H^1(\kappa(\mathfrak{p}'),\mu_m)$ equals $e_{\mathfrak{p}'}(\alpha')$.  But it also equals $e_{\mathfrak{p}}(\alpha)$, the order of $a_{\mathfrak{p}}(\alpha)$ in $\H^1(\kappa(\mathfrak{p}),\mu_m)$, since $e_{\mathfrak{p}'/\mathfrak{p}} = 1$ and since the period of $\hat{\alpha}$ is relatively prime to $ [\FF_{\mathfrak{p}'}:\KK_{\mathfrak{p}}]$.  That (b) implies (c) is then evident in light of Definition \ref{weak:ramification}.
\end{proof}

Using Definition \ref{weak:ramification}, we state Proposition \ref{brauer-r-h}  in the following way.   

\begin{proposition}\label{brauer-r-h} Let $X$ be a normal projective variety, over an algebraically closed characteristic zero field $\kk$, with function field $\KK = \kk(X)$.  Let $\FF / \KK$ be a finite 
field extension
and let $f \colon X' \rightarrow X$ be the normalization of $X$ in $\FF$.  Finally, fix a Brauer class $\alpha \in \Br(\KK)$
and put $\alpha' = f^*\alpha \in \Br(\FF)$.  
Then, in this situation, the divisor $\K_{f^* \alpha} - f^* \K_\alpha$ is effective
if and only if $\alpha$ is weakly ramified, in the sense of Definition \ref{weak:ramification}, at all primes $\mathfrak{p} \in \operatorname{Supp}(\Delta_{\alpha})$.
\end{proposition}
 
\begin{proof}
We want to study the divisor
$$
\K_{f^*\alpha} - f^* \K_\alpha = \Delta_{f^* \alpha} - f^* \Delta_{\alpha} + \operatorname{Ram}(f) \text{.}
$$
Here, the ramification divisor $\operatorname{Ram}(f)$ has the shape
$$
\operatorname{Ram}(f) = \sum_{\mathfrak{p}'} (e_{\mathfrak{p}'/\mathfrak{p}} - 1) \mathfrak{p}' \text{.}
$$
The sum is taken over all prime divisors $\mathfrak{p}'$ of $X'$.  In what follows, the prime $\mathfrak{p}$ is the prime divisor of $X$ that lies below the prime $\mathfrak{p}' \in X'$.  

Consider now the divisor $\Delta_{\alpha}$ which we may write in the form
$$
\Delta_{\alpha} = \sum_{\mathfrak{p}} \left( 1 - \frac{1}{e_{\mathfrak{p}}(\alpha)} \right) \mathfrak{p} \text{.}
$$
Then
$$
f^* \Delta_{\alpha} = \sum_{\mathfrak{p}} \left( 1 - \frac{1}{e_{\mathfrak{p}}(\alpha)} \right) f^*\mathfrak{p} \text{.}
$$
Furthermore, for each fixed $\mathfrak{p} \in \operatorname{Supp}(\Delta_{\alpha})$, by flat pullback of cycles, it holds true that
$$
f^* \mathfrak{p} = \sum_{\mathfrak{p}' | \mathfrak{p}} \mathfrak{p}' \text{.}
$$
Combining, it follows that
$$
f^*\Delta_{\alpha} = \sum_{\mathfrak{p} \in \operatorname{Supp}(\Delta_{\alpha})} \sum_{\mathfrak{p}' | \mathfrak{p}} \left(1 - \frac{1}{e_{\mathfrak{p}}(\alpha)}  \right) \mathfrak{p}' \text{.}
$$

On the other hand
$$
\Delta_{f^* \alpha} = \sum_{\text{primes } \mathfrak{p}' \in X'} \left( 1 - \frac{1}{e_{\mathfrak{p}'}(\alpha')} \right) \mathfrak{p}' \text{.}
$$
Thus, using the above discussion, we obtain the relation 
\begin{align*}
\K_{f^* \alpha} - f^* \K_{\alpha} & = \Delta_{f^* \alpha} - f^* \Delta_{\alpha} + \operatorname{Ram}(f)  \\
& = \sum_{\text{primes } \mathfrak{p}' \in X'} \left( 1 - \frac{1}{e_{\mathfrak{p}'}(\alpha')} \right) \mathfrak{p}' \\
& - \left( \sum_{\mathfrak{p} \in \operatorname{Supp}(\Delta_{\alpha})} \sum_{ 
\substack{
\text{ primes } \mathfrak{p}' \in X' \\
\text{ with } \mathfrak{p}' | \mathfrak{p} } } \left( 1 - \frac{1}{e_{\mathfrak{p}}(\alpha) }\right) \mathfrak{p}' \right) \\
& + \sum_{\text{ primes } \mathfrak{p}' \in X'} \left( e_{\mathfrak{p}' / \mathfrak{p} }- 1 \right) \mathfrak{p}' 
\end{align*}
which we may rewrite in the form
\begin{align*}
\K_{f^* \alpha} - f^* \K_{\alpha} & = \sum_{ \substack{ \text{prime divisors} \\ \mathfrak{p}' \in X'} } \left( e_{\mathfrak{p}' / \mathfrak{p}} - \frac{1}{e_{\mathfrak{p}'}(\alpha')} \right) \mathfrak{p}'  \\
& - \left(
\sum_{ \mathfrak{p} \in \operatorname{Supp}(\Delta_{\alpha})} 
\sum_{
\substack{\text{prime divisors} \\
\mathfrak{p}' \in X' \\
\text{with $\mathfrak{p}' | \mathfrak{p}$}
}
}
\left( 1 - \frac{1}{e_{\mathfrak{p}}(\alpha)} \right)
 \mathfrak{p}'
\right) \text{.}
\end{align*}

We want to study the negative contribution in the right hand side above.  Fix a prime divisor $\mathfrak{p}' \in X'$ which has the property that $\mathfrak{p}' | \mathfrak{p}$ for some prime $\mathfrak{p} \in \operatorname{Supp}(\Delta_\alpha)$.  Then the coefficient of $\mathfrak{p}'$ in $\K_{f^* \alpha} - f^* \K_{\alpha}$  is given by
$$
e_{\mathfrak{p}' / \mathfrak{p}} - \frac{1}{e_{\mathfrak{p}'}(\alpha')} + \frac{1}{e_{\mathfrak{p}}(\alpha)} - 1 \text{.}
$$
Thus
$$
\K_{f^* \alpha} - f^* \K_{\alpha} \geq 0
$$
if and only if
\begin{equation}\label{effective:characterization}
e_{\mathfrak{p}' / \mathfrak{p}} - \frac{1}{e_{\mathfrak{p}'}(\alpha')} + \frac{1}{e_{\mathfrak{p}}(\alpha)} \geq 1
\end{equation}
for all prime divisors $\mathfrak{p}' \in X'$ which lie over some prime $\mathfrak{p} \in \operatorname{Supp}(\Delta_{\alpha})$.
\end{proof}

\begin{example}
Let us now consider the above quantity \eqref{effective:characterization} in some further details.  Let $\mathfrak{p} \in \operatorname{Supp}(\Delta_{\alpha})$ and consider the case that 
$
e_{\mathfrak{p}'/\mathfrak{p}} =1 \text{ and } e_{\mathfrak{p}'}(\alpha') = 1 \text{.}
$
In particular, the field extension $\FF / \KK$ is not ramified at $\mathfrak{p}' | \mathfrak{p}$ nor is $\alpha'$, the pullback of $\alpha \in \Br(\KK)$, ramified at $\mathfrak{p}'$.  In this case, 
$\frac{1}{e_{\mathfrak{p}}(\alpha)} < 1$, since $\mathfrak{p} \in \operatorname{Supp}(\Delta_{\alpha})$,
and so the divisor
$\K_{f^* \alpha} - f^* \K_{\alpha}$
cannot be effective because of Proposition \ref{brauer-r-h}.  We thank an anonymous referee for stimulating examples of this flavour. 
\end{example}

Proposition \ref{brauer-r-h} has the following consequence.

\begin{corollary}\label{R-H-orders:Kod:cor1}
In the setting of Proposition \ref{brauer-r-h}, especially if $\alpha \in \Br(\KK)$ is weakly ramified with respect to $f \colon X' \rightarrow X$, the normalization of $X$ in $\FF$, assume
further that the classes $\alpha$ and $\alpha'$ determine $\QQ$-Gorenstein pairs.  Then, in this situation, it holds true that:
$$ \kappa(X',\K_{f^*\alpha}) \geq \kappa(X,\K_{\alpha}).$$
\end{corollary} 

\begin{proof}[Proof of Corollary~\ref{R-H-orders:Kod:cor1}]
To begin with, the hypothesis of Corollary~\ref{R-H-orders:Kod:cor1} implies that the difference
$$\K_{f^*\alpha} - f^* \K_{\alpha}$$
is an effective $\QQ$-Cartier divisor.  Thus, by Lemma~\ref{easyLemma'}, we have that
\begin{equation}\label{R-H-orders:Kod:cor1:eqn}
\kappa(X',\K_{f^* \alpha})  \geq \kappa(X',f^* \K_{\alpha})
\end{equation}
holds true.
Using \eqref{R-H-orders:Kod:cor1:eqn}, it then follows, from~\cite[Proposition 1.5 (ii)]{Mori:1985} for instance, that
$$\kappa(X,\K_{\alpha}) = \kappa(X',f^* \K_{\alpha}) \leq  \kappa(X',\K_{f^* \alpha})$$
as desired.
\end{proof}

We are now interested in two  central simple algebras $\Sigma_1$ and $\Sigma_2$ with centres $\KK$ and $\FF$ respectively.
We assume we have embeddings $\Sigma_1 \hookrightarrow \Sigma_2$ such that there is an induced map of centres $\KK \subseteq \FF$.
We summarize our setting below in the following diagram of extensions of  central simple algebras with extensions of centres: 
\begin{equation}\label{divisionAlgDiagram}
\begin{tikzcd}
\Sigma_1 \arrow[hook]{r}  & \FF \otimes_{\KK} \Sigma_1 \arrow[hook]{r}  & \Sigma_2  \\
\KK \arrow[hook]{r} \arrow[hook]{u} & \FF  \arrow[hook]{ur} 
\end{tikzcd}
\end{equation}
Let $\alpha_1 \in \Br (\KK)$ denote the class of $\Sigma_1$, $f^* \alpha_1 \in \Br(\FF)$ the class of $\FF \otimes_\KK  \Sigma_1$, and { $\alpha_2 \in \Br(\FF)$  the class of $\Sigma_2$.  

\begin{corollary}\label{R-H-embedding:cor} Assume, in addition to the assumptions of Corollary~\ref{R-H-orders:Kod:cor1}, that 
the classes $\alpha_1 \in \Br(\KK)$, $f^*\alpha_1 \in \Br(\FF)$ and $\alpha_2 \in \Br(\FF)$ determine $\QQ$-Gorenstein pairs and that the ramification of $\alpha_2$ is divisible by the ramification of $f^*\alpha_1$ at all prime divisors $\mathfrak{p}'$ of $X'$. Then, under these assumptions, it holds true that:
$$ \kappa(X,\K_{\alpha_1}) \leq \kappa(X',\K_{\alpha_2}).$$
\end{corollary}
\begin{proof}
We are considering extensions of central simple algebras which fits into an extension diagram as given in \eqref{divisionAlgDiagram} above.
Thus, the conclusion follows by 
combining Corollary~\ref{R-H-orders:Kod:cor1} and Theorem~\ref{embedding:iitaka:theorem}.
\end{proof}

 Note that the divisibility hypothesis
of Corollary~\ref{R-H-embedding:cor} is satisfied in the context of Proposition \ref{period:index} whereas the condition about weak ramification is implied under the assumptions given by Proposition \ref{weak:ramification:prop}. 
Also, observe that Corollary~\ref{R-H-orders:Kod:cor1} and Corollary~\ref{R-H-embedding:cor} pertain to Iitaka dimensions and not Kodaira dimensions.  In other words, these numbers may depend on a given choice of model of $\KK$.  Removing this birational dependence is one of the aims of \S \ref{Galois:embeddings:division:algebras}.

\subsubsection{Galois embeddings, central simple algebras and Kodaira dimensions}\label{Galois:embeddings:division:algebras}  We continue to assume that $\cchar (\kk) = 0$ and our goal is to consider extensions of central simple algebras over $\kk$ which induce a Galois extension of their centres.   Let $\KK = \kk(X)$ be the function field of a normal projective variety $X$ and let $\FF / \KK$ be a finite Galois extension with Galois group $\G = \operatorname{Gal}(\FF/\KK)$.  Next, we let $\Sigma$ be a $\KK$-central simple algebra with Brauer class $\alpha \in \Br(\KK)$ and we suppose that the pair $(X,\mathbb{D}(\alpha))$ has b-canonical singularities.  Let $f \colon X' \rightarrow X$ be the normalization of $X$ in $\FF$.  Recall that $\G$ acts on $X'$ and that $f \colon X' \rightarrow X$ is the quotient map (this is a consequence of ~\cite[Theorem 9.3]{Mat}).  In what follows, we study a concept of $\G$-equivariant resolution for the pair $(X', \mathbb{D}(f^* \alpha))$.  We make this precise as follows.

\begin{defn}\label{G:b:log:pair}
Let $\G$ be a finite group.  A $\G$-b-log pair is a b-log pair $(Y,\mathbb{E})$ consisting of an integral $\G$-variety $Y$ over $\kk$ and a b-divisor $\mathbb{E} \in \DDiv_\QQ(\kk(Y))$ with the property that if $Z \in \mathcal{M}_{/\kk(Y)}$ is a $\G$-variety compatible with the action of $\G$ on $\kk(Y)$, then $\mathbb{E}_Z$, the trace of $\mathbb{E}$ on $Z$, is $\G$-stable. 
\end{defn}
 
Having defined the concept of a $\G$-b-log pair, we now make precise what we mean by $\G$-equivariant, b-canonical and b-terminal resolutions of such a pair.

\begin{defn}\label{b-equiv-b-terminal}
Let $\G$ be a finite group and $(Y,\mathbb{E})$ a $\G$-b-log pair.  By a \emph{$\G$-equivariant b-canonical resolution} of $(Y,\mathbb{E})$, we mean a $\G$-b-log pair $(\widetilde{Y},\mathbb{E})$ having b-canonical singularities together with a proper $\G$-equivariant birational morphism $p \colon \widetilde{Y} \rightarrow Y$.  We define a \emph{$\G$-equivariant b-terminal resolution} similarly.
\end{defn}

We establish existence of a stronger version of such a resolution.  We first make the following definition.

\begin{defn}
  Let $\KK$ be a finitely generated field over $\kk$ and let $\FF$ be a Galois extension of $\KK$ with Galois group $\G$.  We define the \emph{branch b-divisor} $\mathbb{B}$, of the Galois extension $\FF/\KK$, in the following way.  Let $X \in \mathcal{M}_{/\KK}$ have normalization $\pi_X \colon X' \to X$ in $\FF$.  Then $\mathbb{B} \in \DDiv_{\QQ}(\KK)$ is defined by the condition that its trace $\mathbb{B}_X$ is defined in terms of the equation:
  \begin{equation}\label{G-b-log-eqn5}
\K_{X'} = \pi^*_X(\K_X + \mathbb{B}_X).
  \end{equation}
  Such a divisor $\mathbb{B}_X$ exists  by the Riemann-Hurwitz formula since $\FF$ is Galois over $\KK.$ This b-divisor is also described in~\cite[Example 2.10]{Chan:plus:10}. 
  \end{defn}

The manner in which branch divisors arise within the context of a given $\G$-b-log pair $(Y,\mathbb{E})$, where the $b$-divisor $\mathbb{E}$ has coefficients in $[0,1)\cap \QQ$, is made precise in the following way.

\begin{proposition}\label{fractional:branch:prop}
  Let $\G$ be a finite group and $(Y,\mathbb{E})$ a $\G$-b-log pair with the property that the coefficients of $\mathbb{E}_{Z'}$, for each $Z' \in \mathcal{M}_{/\kk(Y)}$, lie in $[0,1) \cap \QQ.$ Let $\mathbb{B}$ be the branch b-divisor of the extension $\kk(Y)/\kk(Y)^{\G}$.  Let $\pi_{Z} \colon Z' \rightarrow Z = Z'/{\G}$ be the quotient map. Then the $b$-divisor
\begin{equation}    \label{eqn:B:plus:E}
  \mathbb{B}_Z+\frac{1}{|\G|} \pi_{Z*} \mathbb{E}_{Z'}
  \end{equation}
    has coefficients in  $[0,1) \cap \QQ.$
\end{proposition}
\begin{proof}   Let $\FF = \kk(Y)$, $\KK = \kk(Y)^{\G}$.
Fix  
$Z \in \mathcal{M}_{/\KK}$ and fix a prime divisor $\mathfrak{p}$ on $Z$.  Let $\pi_Z \colon Z' \rightarrow Z$ be the normalization of $Z$ in $\FF$.  Put $\mathbb{F}_Z = \mathbb{B}_Z+\frac{1}{|\G|} \pi_{Z*} \mathbb{E}_{Z'}$. Then the coefficient of $\mathfrak{p}$ in $\mathbb{B}_Z$
 is $1-1/e$ where $e$ is the ramification index of $\FF$ over $\KK$ at $\mathfrak{p}$.

Let $d \in [0,1) \cap \QQ$ be the coefficient of $\mathbb{E}_{Z'}$ at $\mathfrak{p}'$.  Then, since the coefficient of $\pi_{Z*} \mathbb{E}_{Z'}$ at $\mathfrak{p}$ is $d|\G|/e$, it follows that the coefficient of $\mathbb{F}_Z$ at $\mathfrak{p}$ is:
$$\left(1 - \frac{1}{e} \right) + \frac{d}{e} \in [0,1) \cap \QQ .$$
\end{proof}

Fix a finite Galois extension $\FF/\KK$, with Galois group $\G$, and consider the canonical map 
$\pi \colon \Spec \FF \to \Spec \KK$.  In  what follows, we write the b-divisor which is defined by equation~\eqref{eqn:B:plus:E} as
\begin{equation}\label{log:branch:b-divisor}
\mathbb{B}+\frac{1}{|\G|} \pi_{*} \mathbb{E} \text{.}
\end{equation}

Having fixed some preliminaries, our refined form of Definition \ref{b-equiv-b-terminal} is expressed as follows.

  \begin{defn}
    Let $\G$ be a finite group and $(Y,\mathbb{E})$ a $\G$-b-log pair.
    Let $\mathbb{B}$ be the branch b-divisor of the extension $\kk(Y)/\kk(Y)^{\G}$.
    By a \emph{strong $\G$-equivariant b-canonical resolution} of $(Y,\mathbb{E})$, we mean a $\G$-b-log pair $(\widetilde{Y},\mathbb{E})$ having b-canonical singularities together with a proper $\G$-equivariant birational morphism $p \colon \widetilde{Y} \rightarrow Y$ such that the b-log pair 
$$\left(\widetilde{Y}/\G,\mathbb{B}+ \frac{1}{|\G|} \pi_*\mathbb{E} \right)$$ 
is b-canonical.  We define a \emph{strong $\G$-equivariant b-terminal resolution} similarly.
\end{defn}

The following result uses existence of b-terminal resolutions of b-log pairs \cite[Theorem 2.30]{Chan:plus:10}.  It establishes existence of strong $\G$-equivariant b-terminal resolutions.  For later use, we state explicitly the manner in which such resolutions are obtained.

\begin{theorem}\label{existence:equivariant:terminal:resolutions} 
  Let $\G$ be a finite group and $(Y,\mathbb{E})$ a $\G$-b-log pair with the property that the coefficients of $\mathbb{E}_Z$, for each $Z \in \mathcal{M}_{/\kk(Y)}$, lie in $[0,1) \cap \QQ$.  Then the pair $(Y,\mathbb{E})$ admits a strong $\G$-equivariant b-terminal resolution.  In more explicit terms, let $X = Y / \G$, $\FF = \kk(Y)$, $\KK = \kk(X)$ and denote by
 $$\mathbb{F} = \mathbb{B} + \frac{1}{|\G|} \pi_* \mathbb{E}$$
 the b-divisor defined by equation \eqref{log:branch:b-divisor}.  
 Fix a b-terminal resolution $(\widetilde{X},\mathbb{F})$ of $(X,\mathbb{F})$ and let $\tilde{\pi} \colon \widetilde{Y} \rightarrow \widetilde{X}$ be the normalization of $\widetilde{X}$ in $\FF$.  Then $(\widetilde{Y},\mathbb{E})$ is a strong $\G$-equivariant b-terminal resolution of $(Y,\mathbb{E})$.
\end{theorem}
\begin{proof} 
 By~\cite[Theorem 2.30]{Chan:plus:10}, which applies since $\mathbb{F}_W \in [0,1) \cap \QQ$ for each $W \in \mathcal{M}_{/\KK}$, as is a consequence Proposition \ref{fractional:branch:prop}, there exists a b-terminal  resolution $(\widetilde{X},\mathbb{F})$ of the b-log pair $(X,\mathbb{F})$ and we let $\tilde{\pi} \colon \widetilde{Y} \rightarrow \widetilde{X}$ be the normalization of $\widetilde{X}$ in $\FF$.   Then $\G$ acts on $\widetilde{Y}$. Further, since $Y$ is the normalization of $X$ in $\FF$, it then follows, by the universal property of normalizations, that $\widetilde{Y}$ admits a natural birational $\G$-equivariant map $p\colon \widetilde{Y} \rightarrow Y$. The pair $(\widetilde{Y},\mathbb{E})$ is $\QQ$-Gorenstein since $(\widetilde{X},\mathbb{F})$ is $\QQ$-Gorenstein and since  
 $$\K_Y+\mathbb{E}_Y = \widetilde{\pi}^*(\K_X+\mathbb{F}_X)\text{.}$$ 
Our goal now is to show that the pair $(\widetilde{Y},\mathbb{E})$ is a strong $\G$-equivariant b-terminal  resolution of $(Y,\mathbb{E})$.  Since $(\widetilde{X}, \mathbb{F})$ is b-terminal and since the morphism $p \colon \widetilde{Y} \rightarrow Y$ is $\G$-equivariant, it remains only to prove that $(\widetilde{Y},\mathbb{E})$ is $b$-terminal.

Let $D$ be an irreducible exceptional divisor over $\widetilde{Y}$;  
we want to show that $D$ has positive discrepancy with respect to $\mathbb{E}$.
Without loss of generality we can choose $f \colon \widetilde{Y}' \rightarrow \widetilde{Y}$ to be a $\G$-equivariant proper birational morphism which has $D$ as an irreducible exceptional divisor.   Indeed, we can
take the Galois orbit of $D$ interpreted as a collection of  divisorial valuations on
$\FF$.  Then we can perform successive equivariant blow-ups of the centres of these valuations to obtain the desired model.

Let $\widetilde{X}' = \widetilde{Y}'/\G$ with canonical morphism
$\tilde{\pi}'\colon  \widetilde{Y}' \to \widetilde{X}'$.  Note that
$E=\tilde{\pi}'_*(D)$ is an irreducible exceptional divisor over $\widetilde{X}$.  We then have a commutative diagram:
$$
\begin{tikzcd}
  \widetilde{Y}' \arrow{r}{f} \arrow{d}[swap]{\tilde{\pi}'} & \widetilde{Y} \arrow{d}{\tilde{\pi}} \\
  \widetilde{X}' \arrow{r}{g} & \widetilde{X}
\end{tikzcd} 
$$  
of $\G$-varieties which is induced by the universal property of quotients.  

We next show that $D$ has positive discrepancy.    To that end, consider now
the divisor
$$ B = \K_{\widetilde{X}'} + \mathbb{F}_{\widetilde{X}'} - g^*(\K_{\widetilde{X}} + \mathbb{F}_{\widetilde{X}})$$
on $\widetilde{X}'$.

Since $(\widetilde{X},\mathbb{F})$ is b-terminal, all exceptional divisors have  positive discrepancy over $\widetilde{X}$. 
Pulling back $B$ via $\tilde{\pi}'$, we then obtain:
\begin{equation}\label{G-b-log-eqn5'} 
\tilde{\pi}^{' *} \left( \K_{\widetilde{X}'} + \mathbb{F}_{\widetilde{X}'} \right) = \tilde{\pi}^{' *} g^* \left(\K_{\widetilde{X}} + \mathbb{F}_{\widetilde{X}} \right) + \tilde{\pi}^{' *} B.
\end{equation}
On the other hand, using \eqref{G-b-log-eqn5} and the definition of $\mathbb{F}$, we have:
\begin{equation}\label{G-b-log-eqn5''}
\tilde{\pi}'^*(\K_{\widetilde{X}'} + \mathbb{F}_{\widetilde{X}'} ) = \K_{\widetilde{Y}'} + \mathbb{E}_{\widetilde{Y}'}
\end{equation}
and so, combining \eqref{G-b-log-eqn5'} and \eqref{G-b-log-eqn5''}, it follows that:
$$
\K_{\widetilde{Y}'} + \mathbb{E}_{\widetilde{Y}'} = f^* \tilde{\pi}^{*} \left( \K_{\widetilde{X}} + \mathbb{F}_{\widetilde{X}} \right) + \tilde{\pi}^{' *} B
$$
which we can rewrite
as:
\begin{equation}\label{G-b-log-eqn9}
\K_{\widetilde{Y}'} + \mathbb{E}_{\widetilde{Y}'} = f^* \left( \K_{\widetilde{Y}} + \mathbb{E}_{\widetilde{Y}} \right) + \tilde{\pi}^{' *} B.
\end{equation}
Finally, since $B$ is a positive discrepancy divisor over $\widetilde{X}$, it follows from \eqref{G-b-log-eqn9} that $\tilde{\pi}^{' *} B$ is a positive discrepancy divisor over $\widetilde{Y}$ too.  
In particular, the coefficient of $D$ in $\tilde{\pi}^{' *} B$ is positive.  We have shown that all exceptional divisors over $\widetilde{Y}$ have positive discrepancy with respect to $\mathbb{E}$ and thus $(\widetilde{Y},\mathbb{E})$ is b-terminal.  
\end{proof}      
  
Having shown the existence of strong $\G$-equivariant b-terminal resolutions, we establish Proposition~\ref{Galois:embedding:theorem} below which is a  birationally invariant form of Proposition~\ref{brauer-r-h} and Corollary~\ref{R-H-orders:Kod:cor1}.  To do so, we first formulate a variant of Definition \ref{weak:ramification}.  The idea is to control the ramification for a Galois extension with respect to an equivariant resolution of singularities.  We make this precise 
by way of the following concept:

\begin{defn}\label{G:weak:ramification:2}
  Let $\KK$ be a finitely generated field over $\kk$.  Let $\FF/\KK$ be a finite Galois extension with Galois group $\G = \operatorname{Gal}(\FF / \KK)$.
 Let $\Sigma$ be a $\KK$-central simple algebra with Brauer class $\alpha \in \Br(\KK)$ and let $\alpha' \in \Br(\FF)$ be the pullback of $\alpha$.  We say that $\alpha$ is \emph{weakly $\FF/\KK$ ramified}  if there exists a projective model $Y$ of $\FF$ such that $(Y,\DD(\alpha'))$ is $\G$-equivariant, $b$-canonical and the map $Y \to Y/\G$ is weakly $\alpha$-ramified. 
\end{defn}

The following Lemma is a consequence of Proposition \ref{weak:ramification:prop} combined with Theorem \ref{existence:equivariant:terminal:resolutions}. It gets used in the proof of Theorem \ref{Division:Alg:Galois:Embedding:Thm} and Corollary \ref{rel-prime-period-index}.

\begin{lemma}\label{prime:to:index}
Let $\KK$ be a finitely generated field over $\kk$ and $\FF / \KK$ a finite Galois extension with Galois group $\G = \Gal(\FF/\KK)$.  Let $\alpha \in \Br(\KK)$ with pullback $\alpha' \in \Br(\FF)$.  If the index of $\alpha$ is relatively prime to the degree of $\FF$ over $\KK$, then $\alpha$ is weakly $\FF / \KK$ ramified.
\end{lemma}
\begin{proof}
Let $X$ be a normal projective model of $\KK$ and let $Y$ be the normalization of $X$ in $\FF$.   Then $(Y,\mathbb{D}(\alpha'))$ is a $\G$-b-log pair.  Let 
$$\mathbb{F} = \mathbb{B} + \frac{1}{|\G|} \pi_* \mathbb{D}(\alpha')$$ 
and let $(\widetilde{X},\mathbb{F})$ be a b-terminal resolution of $(X,\mathbb{F})$ as in Theorem \ref{existence:equivariant:terminal:resolutions}.  Then, as noted in Theorem \ref{existence:equivariant:terminal:resolutions}, $(\widetilde{Y},\mathbb{D}(\alpha'))$, the strong $\G$-equivariant b-terminal resolution of $(Y,\mathbb{D}(\alpha'))$, has the property that $\tilde{\pi} \colon \widetilde{Y} \rightarrow \widetilde{X}$ is the normalization of $\widetilde{X}$ in $\FF$.  On the other hand, by assumption, $[\FF : \KK]$ is relatively prime to the index of $\alpha$.  Then Proposition \ref{weak:ramification:prop} implies that $\tilde{\pi}$ is weakly $\alpha$-ramified.  In light of Definition \ref{G:weak:ramification:2}, this means that $\alpha$ is weakly $\FF / \KK$ ramified.
\end{proof}
  
With this concept of $\G$-weak ramification, we establish Proposition \ref{Galois:embedding:theorem} below which is a key technical point for establishing Theorem \ref{Division:Alg:Galois:Embedding:Thm}.

\begin{proposition}\label{Galois:embedding:theorem}  
Let $X$ be a normal projective variety over $\kk$, with function field $\KK = \kk(X)$ and let $\FF / \KK$ be a finite Galois extension with Galois group $\G = \operatorname{Gal}(\FF / \KK)$.  Let  $\Sigma$ be a $\KK$-central simple algebra with Brauer class $\alpha \in \Br(\KK)$ and let $\alpha' \in \Br(\FF)$ be the pullback of $\alpha$.  Then, in this setting, if the pair $(X,\mathbb{D}(\alpha))$ has b-canonical singularities and if $\alpha$ is weakly $\FF / \KK$ ramified in the sense Definition \ref{G:weak:ramification:2}, then
$$ \kappa(\alpha) \leq \kappa(\alpha').$$
\end{proposition}

\begin{proof}
Let $f \colon X' \rightarrow X$ be the normalization of $X$ in $\FF$ and let $(\widetilde{Y},\mathbb{D}(\alpha'))$ be a strong $\G$-equivariant b-terminal resolution of the b-log pair $(X',\mathbb{D}(\alpha'))$.  That such a resolution exists is assured by 
Theorem~\ref{existence:equivariant:terminal:resolutions}.  Recall, that since $\kk(\widetilde{Y}) = \FF$ and $\FF^{\G}=\KK$ that $Y = \widetilde{Y} / \G$ is indeed normal and $\KK = \kk(Y)$.  Further, since $\widetilde{Y}$ is normal, the canonical finite map $g \colon \widetilde{Y} \rightarrow Y$ is the normalization of $Y$ in $\FF$.   Finally, we note that, as is a consequence of~\cite[Proposition 5.20 (1)]{Kollar:Mori:1998} for instance, the pair $(Y, \mathbb{D}(\alpha))$ is $\QQ$-Gorenstein.  By assumption, $\alpha$ has $\FF / \KK$ weak ramification.  
Thus, we may assume that $\alpha$ has weak ramification, in the sense of Definition \ref{weak:ramification}, with respect to the map $\widetilde{Y} \rightarrow Y$.  

We now make the following deductions.  To begin with, since the pair $(X,\mathbb{D}(\alpha))$ has b-canonical singularities, it follows that:
\begin{equation}\label{Galois:embedding:theorem:eqn1}
\kappa(Y,\alpha) = \kappa(\alpha).
\end{equation}
On the other hand, by Corollary~\ref{R-H-orders:Kod:cor1} applied to $g \colon \widetilde{Y} \rightarrow Y$, we obtain:
\begin{equation}\label{Galois:embedding:theorem:eqn2}
\kappa(\widetilde{Y},\alpha') \geq \kappa(Y,\alpha).
\end{equation}
Finally, since $(\widetilde{Y},\alpha')$ is b-terminal, and hence b-canonical, we have:
\begin{equation}\label{Galois:embedding:theorem:eqn3}
\kappa(\alpha') = \kappa(\widetilde{Y},\alpha');
\end{equation}
combining \eqref{Galois:embedding:theorem:eqn1}, \eqref{Galois:embedding:theorem:eqn2}, and \eqref{Galois:embedding:theorem:eqn3}, it then follows that
$$ \kappa(\alpha') \geq \kappa(\alpha)$$
as desired.
\end{proof}

Together, the results which  we have obtained thus far allow for establishing our main result, Theorem \ref{Division:Alg:Galois:Embedding:Thm}, of which Theorem \ref{Division:Alg:Galois:Embedding:Thm:Intro} is a special case.

\begin{theorem}\label{Division:Alg:Galois:Embedding:Thm}
Suppose that $\kk$ is an algebraically closed field of characteristic zero, that  $\Sigma_1 \subseteq \Sigma_2$ are central simple algebras, finitely generated over $\kk$ and finite dimensional over their centres $\KK, \FF$ and having the properties that:
\begin{enumerate}[label=(\alph*), ref=(\alph*)]
\item{ $\KK \subseteq \FF$; }
\item{ the extension $\FF / \KK$ is Galois 
 and the Brauer class of $\Sigma_1$ is weakly ramified for the extension $\FF / \KK$; and 
}
\item\label{embd:eff}{  
the pullback central simple algebra
$\FF \otimes_{\KK} \Sigma_1$ is $X'$-effectively embedded in $\Sigma_2$, for some normal proper model $X'$ of $\FF$ such that $(X',\DD([\Sigma_2]))$ has b-canonical singularities and dominates a  normal proper model $X$ of $\FF$ where
$(X,\DD([\FF \otimes_{\KK} \Sigma_1]))$ 
has b-canonical singularities. 
}
\end{enumerate}
Then, with these assumptions, it holds true that: 
$$ \kappa(\Sigma_1) \leq \kappa(\Sigma_2).$$
\end{theorem}

\begin{proof}[Proof of Theorem~\ref{Division:Alg:Galois:Embedding:Thm} and Theorem~\ref{Division:Alg:Galois:Embedding:Thm:Intro}.]  For the case of Theorem \ref{Division:Alg:Galois:Embedding:Thm:Intro}, that condition (b) holds in Theorem \ref{Division:Alg:Galois:Embedding:Thm} follows because of Lemma \ref{prime:to:index}.  Thus, it remains only to establish Theorem \ref{Division:Alg:Galois:Embedding:Thm}.
Let $\alpha_1$ and $\alpha_2$ denote, respectively, the Brauer classes of $\Sigma_1$ and $\Sigma_2$.  Fix a normal proper model $X$ of $\KK$ such that $(X,\mathbb{D}(\alpha_1))$ is b-canonical.  Let $f \colon X' \rightarrow X$ be the normalization of $X$ in $\FF$ and let $(\widetilde{Y},\mathbb{D}(f^* \alpha_1))$ be a $\G$-equivariant b-canonical resolution of the $\G$-b-log pair $(X',\mathbb{D}(f^* \alpha_1))$.  
Again that such a resolution $(\widetilde{Y}, \mathbb{D}(f^*\alpha_1))$ exists is assured by Theorem~\ref{existence:equivariant:terminal:resolutions}.

We then have that:
\begin{equation}\label{galois:kod:eqn1}
\kappa(f^*\alpha_1) = \kappa(\widetilde{Y},\K+\mathbb{D}(f^*\alpha_1))
\end{equation}
and
\begin{equation}\label{galois:kod:eqn2}
\kappa(\alpha_1) = \kappa(X,\K+\mathbb{D}(\alpha_1)).
\end{equation}
Thus, using \eqref{galois:kod:eqn1} and \eqref{galois:kod:eqn2} combined with Proposition~\ref{Galois:embedding:theorem}, which applies because of our assumption about weak ramification, we obtain that
\begin{equation}\label{galois:kod:eqn3}
\kappa(\alpha_1) \leq \kappa(f^*\alpha_1).
\end{equation} 
Having established \eqref{galois:kod:eqn3}, now let $X$ and $X'$ be the b-canonical models supplied by the assumption \ref{embd:eff}.
So we have
\begin{equation}\label{kod:ineqn}
\kappa(\alpha_2) = \kappa(X',\K+\mathbb{D}(\alpha_2)) \geq \kappa(X',\K+\mathbb{D}(f^*\alpha_1)) = \kappa(X,\K+\mathbb{D}(f^*\alpha_1)) = \kappa(f^* \alpha_1)
\end{equation}
where the first and last equalities follow from the b-canonical assumption,
the second last equality follows from Corollary~\ref{b-divisor-theorem-cor}, and the inequality follows from the fact that $\mathbb{D}(\alpha_2) - \mathbb{D}(f^*\alpha_1)$ is effective on $X'$. (Here is where we have used the concept of $b$-effective embedding.)

Finally, combining \eqref{galois:kod:eqn3} and
\eqref{kod:ineqn} we have that 
$ \kappa(\alpha_2) \geq \kappa(\alpha_1)$
as desired.
\end{proof}

A more conceptual instance of Theorem \ref{Division:Alg:Galois:Embedding:Thm} is achieved via Corollary \ref{rel-prime-period-index} below.

\begin{corollary}\label{rel-prime-period-index}
Suppose that $\kk$ is an algebraically closed field of characteristic zero.  Let  $\Sigma_1 \subseteq \Sigma_2$ be central simple  algebras, finite dimensional over their centres $\KK, \FF$ which are finitely generated over $\kk$.  Suppose further that:
\begin{enumerate}
\item{ $\KK \subseteq \FF$;}
\item{ the extension $\FF / \KK$ is Galois with degree relatively prime to the index of $\Sigma_1$; and}
\item{ the target central simple algebra $\Sigma_2$ has the property that $\operatorname{period}(\Sigma_2) = \operatorname{index}(\Sigma_2)$.
}
\end{enumerate}
Then
$$ \kappa(\Sigma_1) \leq \kappa(\Sigma_2).$$
\end{corollary}

\begin{proof}[Proof of Corollaries \ref{rel-prime-period-index} and \ref{rel-prime-period-index-intro}]  Corollary \ref{rel-prime-period-index} is an instance of Theorem \ref{Division:Alg:Galois:Embedding:Thm}.  Indeed that the condition (b), of weak ramification for $\FF / \KK$ in Theorem \ref{Division:Alg:Galois:Embedding:Thm}, is satisfied because of Lemma \ref{prime:to:index} and the fact that the degree of $\FF / \KK$ is assumed to relatively prime to the index of $\Sigma_1$.    That the condition \ref{embd:eff} of Theorem \ref{Division:Alg:Galois:Embedding:Thm}, which pertains to effectivity for the embedding $\Sigma_1 \otimes_{\KK} \FF \hookrightarrow \Sigma_2$, holds true is a consequence of Proposition \ref{perinduni}.  Indeed, by assumption the period and index of $\Sigma_2$ are assumed to be equal.  In particular, the hypothesis of Theorem \ref{Division:Alg:Galois:Embedding:Thm} are satisfied, whence the conclusion of Corollaries \ref{rel-prime-period-index} and \ref{rel-prime-period-index-intro} as desired.
\end{proof}

The following example shows why it is important to impose the condition for the embedding  $\Sigma_1 \hookrightarrow \Sigma_2$ to induce a map on centres $\KK \subseteq \FF$.  

\begin{example}\label{kod:example}
Choose homogeneous coordinates $u,v,w$ for $\PP^2_\kk$ and let
$$\KK = \kk(\PP^2) = \kk(u/w,v/w). $$
Next, let $\mathcal{D}$ be the quaternion (division) algebra generated by $x,y$ over $\KK$ and satisfying the relations:
$$ x^2 = \frac{u}{w}, y^2 = \frac{v}{w}, yx = - xy.$$

Now fix rational functions $a,b,c \in \KK$ and put:
$$ \gamma = a x + b y + c xy \in \mathcal{D}.$$
Then:
$$ 
\gamma^2 = (ax + by + cxy)^2 = a^2 x^2 + b^2 y^2 - c^2 x^2 y^2 = a^2 \frac{u}{w} + b^2 \frac{v}{w} - c^2 \frac{uv}{w^2} \in \KK.
$$
Let $\sigma = w^2\gamma^2$, so  $\sigma \in \H^0(\PP^2,\Osh_{\PP^2}(2d))$, for some $d > 0$, and denote by $B \subseteq \PP^2$ the divisor determined by $\sigma$:
$$B = \operatorname{div}(\sigma).$$
Then $B$ is a plane curve of degree $2d$ and there exist $a,b,c \in \KK$ so that $B$ is smooth; fix such $a,b,c \in \KK$.

Next, let $f \colon Y \rightarrow \PP^2$ denote the double cover branched on $B$, fix a line $\ell \subseteq \PP^2$ and canonical divisors $\K_Y$ and $\K_{\PP^2}$ for $Y$ and $\PP^2$ respectively.  We then have:
$$ 
\K_Y \sim f^* \left(\K_{\PP^2} + \frac{1}{2} B \right) \sim (d-3) f^* \ell.
$$
Further, since $f^* \ell$ is ample, it follows that $Y$ is of general type when $d > 3$, that $Y$ has Kodaira dimension $0$ when $d = 3$ and that $Y$ has negative Kodaira dimension when $d = 1$ or $d = 2$.

Finally, we note that $\kappa(\mathcal{D}) = \kappa(\PP^2, \Delta) < 0$ since, as can be deduced from~\cite[Proposition 1.4.9]{Artin:Chan:deJong:Lieblich} for example, that $\Delta$, the ramification divisor for $\mathcal{D}$, equals one half of the sum of the coordinate lines of $\PP^2$.

Now put $\mathcal{D}_1 := \kk(Y) = \KK(\gamma) \subseteq \mathcal{D}$ and $\mathcal{D}_2 := \mathcal{D}$.  We then have:
\begin{enumerate}
\item[(a)]{$\mathcal{D}_1 \subseteq \mathcal{D}_2$;}
\item[(b)]{$Z(\mathcal{D}_1) \not\subseteq Z(\mathcal{D}_2)$; and }
\item[(c)]{$\kappa(\mathcal{D}_1) > \kappa(\mathcal{D}_2)$} { .}
\end{enumerate}
\end{example}

In light of Example~\ref{kod:example}, we ask:

\begin{question}  Let $\mathcal{D}$ be the division algebra in the above example.  Let $C$ be a curve such that $\kk(C)\subseteq \mathcal{D}$.  Then is the genus of $C$ necessarily zero?
\end{question}

\section{Canonical rings, maximal orders and function fields of division algebras}  Let $X$ be an integral normal variety, assumed to be proper over an algebraically closed field $\kk$, with function field $\KK = \kk(X)$.

\subsection{The canonical ring of a maximal order}\label{canonical:ring:order}  Our main point here is to establish Corollary~\ref{canonical:ring:order:cor4} which provides some motivation for the study of the ramification divisor defined in \eqref{brauer:boundary:defn}. Furthermore, in Definition~\ref{can:ring:defn:1}, we define the canonical ring of a maximal $\Osh_X$-order $\Lambda$.  Then, motivated by this concept, in Definition~\ref{can:ring:defn:3}, we define the section ring of $\Lambda$ with respect to a Cartier divisor on $X$.  We then show, in Corollary~\ref{can:ring:cor:6}, that this ring is a p.i.~ring.    
Finally, what we discuss here should also help to provide further motivation for the results we establish in \S \ref{division:ring:function:field}.

First, we denote the reflexive hull of a coherent sheaf $F$ on $X$ by
$$
F^{\vee \vee} := \mathcal{H}om_{\Osh_X}(\mathcal{H}om_{\Osh_X}(F,\Osh_X),\Osh_X).
$$
Further, we remark that if $F$ is a $\Lambda$-(left, right, bi)-module, then the same is true for $F^{\vee \vee}$.
Next, consider a Weil divisor
$
D = \sum_{\text{finite}} d_{\mathfrak{p}} \mathfrak{p} \in \WDiv(X) 
$
on $X$.  We then can form, given a maximal $\Osh_X$-order $\Lambda$ on $X$, the reflexive $\Osh_X$-module:
\begin{equation}\label{Lambda:D}
(\Osh_X(D)  \otimes_{\Osh_X} \Lambda)^{\vee \vee}.
\end{equation}
Since the $\Osh_X$-module
$$
\Lambda(D) := \Osh_X(D)  \otimes_{\Osh_X} \Lambda
$$
is also a $\Lambda$-bimodule, the reflexive $\Osh_X$-module \eqref{Lambda:D} is also a $\Lambda$-bimodule too.

\begin{proposition}\label{canonical:ring:order:prop1}
Let $X$ be a normal proper variety over $\kk$, with function field $\KK = \kk(X)$, let $\Lambda$ be a maximal $\Osh_X$-order, in a $\KK$-central simple algebra with degree prime to the characteristic of $\kk$, and let $D = \sum_{\mathrm{finite}} d_{\mathfrak{p}} \mathfrak{p} \in \WDiv(X) 
$ be a Weil divisor on $X$.  Then, in this context, if $\mathfrak{p} \in X$ is a codimension $1$ prime, then the stalk of the sheaf $\Lambda(D)$ at $\mathfrak{p}$ is the $\Lambda_{\mathfrak{p}}$-bimodule:
$$
(\operatorname{rad} \Lambda_{\mathfrak{p}})^{-d_{\mathfrak{p}}  e_{\Lambda}(\mathfrak{p}) }.
$$
\end{proposition}
\begin{proof}
We simply note:
$$
(\operatorname{rad} \Lambda_{\mathfrak{p}})^{e_{\Lambda}(\mathfrak{p})}  \simeq \Lambda_{\mathfrak{p}} \mathfrak{m}_{\mathfrak{p}} .
$$
\end{proof}

Consider now the $\Lambda$-bimodule
$
\Lambda^\vee := \mathcal{H}om_{\Osh_X}(\Lambda,\Osh_X)
$
which has stalk at a height $1$ prime $\mathfrak{p}$ of $X$ the $\Osh_{X,\mathfrak{p}}$-module:
$$
\Lambda^{\vee}_{\mathfrak{p}} := \Osh_{X,\mathfrak{p}} \otimes \mathcal{H}om_{\Osh_X}(\Lambda,\Osh_X) = \mathcal{H}om_{\Osh_{X,\mathfrak{p}}}(\Lambda_{\mathfrak{p}},\Osh_{X,\mathfrak{p}}).
$$
The module $\Lambda^{\vee}_{\mathfrak{p}}$ naturally carries the structure of a $\Lambda_{\mathfrak{p}}$-bimodule.  In particular, as a $\Lambda_{\mathfrak{p}}$-bimodule, we have:
\begin{equation}\label{canonical:ring:order:eqn9}
\Lambda^\vee_{\mathfrak{p}} \simeq (\operatorname{rad} \Lambda_{\mathfrak{p}})^{1-e_{\Lambda}(\mathfrak{p})} ,
\end{equation}
as follows from Proposition~\ref{dual:trace:rad:lemma}.

\begin{defn}\label{defn:cansheafOfOrder}
We now  fix a canonical divisor, $\K_X$, on $X$ and we let $\omega_X := \Osh_X(\K_X)$ denote the canonical sheaf that it determines.  In this setting,  since $\omega_X$ is a reflexive sheaf on $X$, we can consider the $\Lambda$-bimodule $\omega_\Lambda$, the \emph{canonical sheaf} of $\Lambda$, which is the reflexive sheaf defined by the condition that:
$$
\omega_{\Lambda} := \mathcal{H}om_{\Osh_X}(\Lambda,\omega_X).
$$
\end{defn}

The stalk of $\omega_X$ at a height $1$ prime ideal $\mathfrak{p}$ of $X$ is described in:

\begin{proposition}\label{canonical:ring:order:prop3}  Let $X$ be a normal proper variety over $\kk$, with function field $\KK = \kk(X)$, and let $\Lambda$ be a maximal $\Osh_X$-order, in a $\KK$-central simple algebra $\Sigma$ with degree prime to $\cchar \kk$, on $X$. Then, at
 a height $1$ prime $\mathfrak{p}$ of $X$, the stalk of the $\Osh_X$-module $\omega_\Lambda$ is the $\Lambda_{\mathfrak{p}}$-bimodule
\begin{equation}\label{canonical:ring:order:eqn11}
\omega_{\Lambda,\mathfrak{p}} \simeq \omega_{X,\mathfrak{p}} \otimes_{\Osh_{X,\mathfrak{p}}}   \Lambda^\vee_{\mathfrak{p}} \simeq \omega_{X,\mathfrak{p}} \otimes_{\Osh_{X,\mathfrak{p}}}  (\operatorname{rad} \Lambda_{\mathfrak{p}})^{1-e_{\Lambda}(\mathfrak{p})} .
\end{equation}
\end{proposition}
\begin{proof}
Since, for each height one prime $\mathfrak{p} \subseteq X$, $\omega_{X,\mathfrak{p}}$ is a free $\Osh_{X,\mathfrak{p}}$-module,  we have a $\Lambda_{\mathfrak{p}}$-bimodule isomorphism
\begin{equation}\label{canonical:ring:order:eqn12}
 \omega_{\Lambda, \mathfrak{p}} \simeq \omega_{X,\mathfrak{p}} \otimes_{\Osh_{X, \mathfrak{p}}}  \mathcal{H}om_{\Osh_{X,\mathfrak{p}}}(\Lambda_{\mathfrak{p}},\Osh_{X, \mathfrak{p}}). 
\end{equation}
The final isomorphism follows by combining \eqref{canonical:ring:order:eqn9} and \eqref{canonical:ring:order:eqn12}.
\end{proof}

In what follows, we denote by $\omega_{\Lambda}^{\otimes \ell}$, for a nonnegative integer $\ell$,
the $\ell$-fold tensor product of the $\Lambda$-bimodule $\omega_{\Lambda}$.  Also, we put $n^2 = [\Sigma : \KK]$.

Proposition~\ref{canonical:ring:order:prop3} has the following consequence which relates $\omega_\Lambda$ to the log-canonical sheaf determined by $\Lambda$.  As such, Corollary~\ref{canonical:ring:order:cor4} below can be seen as a sort of adjunction formula for $\omega_\Lambda$.
 This corollary is essentially~\cite[\S 3, Proposition 5]{Chan:Kulkarni2003} and was originally noticed by M.~Artin.

\begin{corollary}\label{canonical:ring:order:cor4}  With the same assumptions as Proposition~\ref{canonical:ring:order:prop3}, if $n$ denotes the degree of $\Sigma$, then
there exists a natural $\Lambda$-bimodule isomorphism:
\begin{equation}\label{canonical:ring:order:eqn14}
(\omega_{\Lambda}^{\otimes \ell n})^{\vee \vee} \simeq \left( \Osh_X(\ell n(\K_X + \Delta_{\Lambda} ))  \otimes_{\Osh_X} \Lambda \right)^{\vee \vee},
\end{equation}
for each integer $\ell \geq 0$.  Further, if $\ell n (\K_X + \Delta_{\Lambda})$ is Cartier or $\Lambda$ is locally free over $X$, then the $\Lambda$-bimodule isomorphism \eqref{canonical:ring:order:eqn14} takes the form:
\begin{equation}\label{canonical:ring:order:eqn14'}
(\omega_{\Lambda}^{\otimes \ell n})^{\vee \vee} \simeq  \Osh_X(\ell n(\K_X + \Delta_{\Lambda} ))  \otimes_{\Osh_X} \Lambda .
\end{equation}
\end{corollary}
\begin{proof}
Since both the left and right hand sides of \eqref{canonical:ring:order:eqn14} are reflexive sheaves, it suffices to check the above isomorphism at codimension $1$ primes.  In light of this reduction, that \eqref{canonical:ring:order:eqn14} holds true follows from Proposition~\ref{canonical:ring:order:prop3} combined with Proposition~\ref{canonical:ring:order:prop1}. Finally, \eqref{canonical:ring:order:eqn14'} follows from \eqref{canonical:ring:order:eqn14} since, when $\ell n (\K_X + \Delta_{\Lambda})$ is Cartier or $\Lambda$ is locally free over $X$, $\Osh_X(\ell n(\K_X + \Delta_{\Lambda} ))  \otimes_{\Osh_X} \Lambda$ is reflexive. 
\end{proof}

Motivated by Corollary~\ref{canonical:ring:order:cor4}, we define the canonical ring of $\Lambda$ in:
\begin{defn}\label{can:ring:defn:1}
Let $X$ be a normal proper variety over $\kk$, with function field $\KK = \kk(X)$, let $\Lambda$ be a maximal $\Osh_X$-order in a central simple algebra $\Sigma$, with degree $n$ prime to $\operatorname{char} \kk$, and let $\Delta_\Lambda$ be the ramification divisor of $\Lambda$.  We define the \emph{canonical ring of $\Lambda$} to be the graded $\kk$-algebra:
\begin{equation}\label{can:ring:eqn:1}
R(\Lambda,\omega_\Lambda) := \bigoplus_{\ell \geq 0} \Hom_{\Lambda}(\Lambda, (\omega_\Lambda^{\otimes \ell })^{\vee \vee}).
\end{equation}
Here, in \eqref{can:ring:eqn:1}, for each $\phi \in \Hom_\Lambda(\Lambda,(\omega_{\Lambda}^{\otimes \ell  })^{\vee \vee})$ and each $\psi \in \Hom_{\Lambda}(\Lambda,(\omega_{\Lambda}^{\otimes k  })^{\vee \vee})$ the multiplication is given by the natural maps:
\begin{equation}\label{can:ring:eqn:2}
(\phi \otimes \psi)^{\vee \vee} : \Lambda = (\Lambda \otimes_{\Lambda} \Lambda)^{\vee \vee} \rightarrow \left(\omega_{\Lambda}^{\otimes(\ell + k)} \right)^{\vee \vee},
\end{equation}

\end{defn}

The following proposition, among other things, gives an alternative description of $R(\Lambda,\omega_\Lambda)$.

\begin{proposition}\label{can:ring:prop:2}
Let $X$ be a normal proper variety over $\kk$, with function field $\KK = \kk(X)$ and let $\Lambda$ be a maximal $\Osh_X$-order in a central simple algebra $\Sigma$ with degree $n$ prime to $\operatorname{char} \kk$.  Suppose that $(X,\Delta_\Lambda)$ is $\QQ$-Gorenstein and fix an integer $\ell_0>0$ so that $\ell_0 n \K_\Lambda \in \Div(X)$.  Then the following assertions hold true.
\begin{enumerate}
\item{There exists a natural $\Osh_X$-module isomorphism
\begin{equation}\label{can:ring:eqn:3}
\Osh_X(\ell \ell_0 n (\K_X + \Delta_{\Lambda})) \otimes_{\Osh_X} \Lambda \xrightarrow{\sim} \mathcal{H}om_{\Lambda}(\Lambda,(\omega_{\Lambda}^{\otimes \ell \ell_0 n})^{\vee \vee}),
\end{equation}
for each integer $\ell \geq 0$.  Furthermore, for such $\ell$, the $\ell \ell_0 n$\textsuperscript{th} graded piece of the canonical ring $R(\Lambda,\omega_\Lambda)$ has the form:
\begin{equation}\label{can:ring:eqn:4}
R(\Lambda,\omega_\Lambda)_{\ell \ell_0 n} = \H^0(X,\Osh_X(\ell \ell_0 n(\K_X + \Delta_{\Lambda}))\otimes\Lambda).
\end{equation}
}
\item{
The $\ell_0 n$-th Veronese subalgebra of $R(\Lambda,\omega_\Lambda)$, 
$$
R(\Lambda,\omega_\Lambda)^{(\ell_0n)} := \bigoplus_{\ell \geq 0} \Hom_\Lambda(\Lambda, (\omega_{\Lambda}^{\otimes \ell \ell_0 n})^{\vee \vee})
$$
is naturally isomorphic to the graded $\kk$-algebra
$$
\H^0(X,\mathbf{\operatorname{Sym}^\bullet}(\Osh_X(\ell_0 n \K_{\Lambda})) \otimes \Lambda) = \bigoplus_{\ell \geq 0} \H^0(X,\Osh_X(\ell \ell_0 n \K_{\Lambda}) \otimes \Lambda).
$$
}
\end{enumerate}
\end{proposition}

\begin{proof}
Proposition~\ref{can:ring:prop:2} is essentially just a restatement of Corollary~\ref{canonical:ring:order:cor4}.  
\end{proof}

Next, motivated by Proposition~\ref{can:ring:prop:2}, we define the section ring of $\Lambda$ with respect to a Cartier divisor $D \in \Div(X)$.

\begin{defn}\label{can:ring:defn:3}
Let $X$ be a normal proper variety over $\kk$, with function field $\KK = \kk(X)$ and let $\Lambda$ be a maximal $\Osh_X$-order in a central simple algebra $\Sigma$ with degree prime to $\cchar \kk$.  Fix a Cartier divisor $D \in \Div(X)$.  We define the \emph{section ring} of $\Lambda$ with respect to $D$ to be the graded $\kk$-algebra:
\begin{equation}\label{can:ring:eqn10}
R(\Lambda,D) := \H^0(X,\mathbf{\operatorname{Sym}^\bullet}(\Osh_X(D)) \otimes \Lambda).
\end{equation}
\end{defn}

We make a few remarks concerning Definition~\ref{can:ring:defn:3}.

\begin{remark}\label{can:ring:remark:4}  

\begin{itemize}
\item{
The graded pieces of the algebra \eqref{can:ring:eqn10} have the form:
$$
R(\Lambda,D)_\ell = \H^0(X,\Osh_X(\ell D) \otimes \Lambda) = \Hom_{\Osh_X}(\Osh_X,\Osh_X(\ell D) \otimes \Lambda),
$$
for each integer $\ell \geq 0$.
}
\item{
The algebra structure of $R(\Lambda,D)$ is induced by that of the sheaf of (in general noncommutative) algebras $\mathbf{\operatorname{Sym}^\bullet}(\Osh_X(D)) \otimes_{\Osh_X} \Lambda$ on $X$.
}
\item{There is a natural $\kk$-algebra morphism
\begin{equation}\label{can:ring:eqn12}
R(D) \hookrightarrow R(\Lambda,D)
\end{equation}
which identifies the section ring of $D$ as a subalgebra of the centre of $R(\Lambda,D)$.  This morphism \eqref{can:ring:eqn12} is induced by the natural $\Osh_X$-algebra morphism
$$
\mathbf{\operatorname{Sym}^\bullet}(\Osh_X(D)) \hookrightarrow \mathbf{\operatorname{Sym}^\bullet}(\Osh_X(D)) \otimes_{\Osh_X} \Lambda.
$$
}
\end{itemize}
\end{remark}

We next show that the section ring $R(\Lambda,D)$, for $D \in \Div(X)$ is a subalgebra of $\Sigma[t]$.

\begin{proposition}\label{can:ring:prop:5}
Let $X$ be a normal proper variety over $\kk$, with function field $\KK = \kk(X)$ and let $\Lambda$ be a maximal $\Osh_X$-order in a $\KK$-central simple algebra $\Sigma$ with degree prime to $\cchar \kk$.  If $D \in \Div(X)$ is a Cartier divisor on $X$, then the section ring $R(\Lambda,D)$ is a subalgebra of $\Sigma[t]$.
\end{proposition}
\begin{proof}
Put $L = \Osh_X(D)$, let
$$
\mathcal{A} := \mathbf{\operatorname{Sym}^\bullet}(L) = \bigoplus_{\ell \geq 0} L^{\otimes \ell}
$$
and let $\pi \colon Y = \mathbf{\operatorname{Spec}}_X \mathcal{A} \rightarrow X$ be the natural affine morphism.  In particular,
$$
\H^0(Y,\Osh_Y) = \H^0(X,\mathcal{A}) = \bigoplus_{\ell \geq 0} \H^0(X,L^{\otimes \ell}),
$$
as in~\cite[Ex.~II.~5.17]{Hart} for instance.  Further, by pulling back $\Lambda$ to $Y$ and then considering the base change to $\operatorname{Spec}(\KK)$, we have:
$$
\KK \otimes \pi^* \Lambda = \KK \otimes \mathbf{\operatorname{Sym}^\bullet}(L) \otimes \Lambda = \KK[t] \otimes_\KK \Lambda  \simeq \Sigma[t]
$$
and it then follows, since $\pi$ is affine, that
\begin{equation}\label{can:ring:eqn17}
\H^0(Y,\pi^* \Lambda) = \H^0(X,\pi_* \pi^* \Lambda) = \bigoplus_{\ell \geq 0} \H^0(X,L^{\otimes \ell} \otimes \Lambda) \subseteq \Sigma[t].
\end{equation}
Finally, since
\begin{equation}\label{can:ring:eqn18}
R(\Lambda,D) = \bigoplus_{\ell \geq 0} \H^0(X,L^{\otimes \ell }\otimes \Lambda),
\end{equation}
it follows from \eqref{can:ring:eqn17} and \eqref{can:ring:eqn18} that
$$
R(\Lambda,D) \subseteq \Sigma[t]
$$
as desired.
\end{proof}

In general, the canonical ring of a maximal order need not be  prime.

\begin{example}\label{notPrimeExample}
Let $x,y$ be homogeneous coordinates on $\PP^1$, 
let $E := \Osh_{\PP^1}(-1) \oplus \Osh_{\PP^1}$ and let 
$$\Lambda := \mathcal{E}nd_{\PP^1}(E) \subseteq \kk(x/y)^{2 \times 2}.
$$ 
Then
$$ 
\Lambda(-2\ell) = \left(
\begin{matrix}
\Osh_{\PP^1}(-2\ell) & \Osh_{\PP^1}(-2\ell + 1) \\
\Osh_{\PP^1}(-2\ell -1) & \Osh_{\PP^1}(-2\ell) 
\end{matrix} \right)
$$
and
$$ R(\Lambda,\omega_\Lambda) = \bigoplus_{\ell \geq 0} \H^0(\PP^1,\Lambda(-2\ell)) =  \left(
\begin{matrix}
\kk & \kk x + \kk y \\
0 & \kk
\end{matrix} \right) \subseteq \kk(x,y)^{2 \times 2}$$
which is not a prime ring.
\end{example}

Before stating one consequence of Proposition~\ref{can:ring:prop:5}, we recall that by a \emph{polynomial identity ring}, or simply a \emph{p.i.~ring}, we mean a ring which satisfies some nonzero universal polynomial relation in the free algebra $\kk \langle \mathbf{x} \rangle = \kk \langle x_1,\dots,x_N \rangle$, for some integer $N \gg 0$.  Having recalled this concept, we can now state:
\begin{corollary}\label{can:ring:cor:6}
With the same assumptions and notations of Proposition~\ref{can:ring:prop:5}, the section ring $R(\Lambda,D)$ is a p.i.~ring.
\end{corollary}
\begin{proof}
By Proposition~\ref{can:ring:prop:5}, the section ring $R(\Lambda,D)$ is a subring of the p.i.~ring $\Sigma[t]$. 
\end{proof}

\subsection{The function field of a maximal order with respect to a divisor}\label{division:ring:function:field}  Here we consider maximal orders in division rings, with degree prime to $\cchar \kk$, and our main goal is to show how, given such a maximal $\Osh_X$-order $\Lambda$, together with a Cartier divisor $D$ on $X$, we can define a division algebra $\kk(\Lambda,D)$, which we can think of as the function field of $\Lambda$ with respect to $D$.

\subsubsection{Preliminaries about Graded Rings of Fractions}  To begin with, we recall, following ~\cite{Graded:Ring:Theory}, that if $A$ is a $\ZZ$-graded ring and $S$ a multiplicatively closed subset consisting of nonzero homogeneous elements,  then $A$ satisfies the \emph{graded left Ore conditions} with respect to $S$ if the two conditions hold:
\begin{itemize}
\item{if $r \in A$ is a homogeneous element, $s \in S$ and $rs = 0$, then there exists an $s' \in S$ with the property that $s'r=0$;}
\item{for each homogeneous $r \in A$ and each $s \in S$, there exists homogeneous elements $r' \in A$ and $s' \in S$ with the property that $s' r = r's$.}
\end{itemize}
If $A$ satisfies the graded left Ore conditions with respect to $S$, then the left ring of fractions $S^{-1} A$ can be described as:
\begin{equation}\label{left:ring:fractions:defn}
S^{-1} A = \left\{ s^{-1} a : \text{ $a \in A$ and $s \in S$}    \right\}.
\end{equation}
Here the ring operations are given as:
$$ s^{-1}x + t^{-1}y = u^{-1}(ax + by),$$
for $a,b \in A$ such that $u = as = bt \in S$, and 
$$ s^{-1}x \cdot t^{-1}y = (t_1 \cdot y)^{-1} (x_1 \cdot y),$$
for $t_1 \in S$, $x_1 \in A$ with the property that 
$ t_1 x = x_1 t.$
Furthermore, the ring $S^{-1} A$ is a $\ZZ$-graded ring with gradation given by:
$$
(S^{-1} A)_\ell = \left\{ s^{-1} a : \text{ $s \in S$, $a \in A$ such that $\ell = \operatorname{deg}(a) - \operatorname{deg}(s)$} \right\},
$$
for $\ell \in \ZZ$.  

For later use, we note:

\begin{proposition}\label{graded:ore:pi:prop}
Let $X$ be a normal proper variety over $\kk$ with function field $\KK = \kk(X)$ and let $\Lambda$ be a maximal $\Osh_X$-order in a $\KK$-central division algebra $\mathcal{D}$ with degree prime to $\cchar \kk$.  Let $D \in \Div(X)$ be a Cartier divisor on $X$ and $A$ a graded subring of the section ring $R(\Lambda,D)$.  Then $A$ satisfies the graded left Ore conditions with respect to the multiplicative set of nonzero homogeneous elements of $A$.
\end{proposition}
\begin{proof}
We argue as in~\cite[Corollary 10.26]{Lam:modules:rings} which considers the case of non-graded p.i.~rings.  To begin with, since $R(\Lambda,D) \subseteq \mathcal{D}[t]$, by Proposition~\ref{can:ring:prop:5}, the section ring $R(\Lambda,D)$ is a domain and hence also $A$ is a domain and so the first condition clearly holds.

Next, suppose that the second condition is false.  Then, in this case, there exists homogeneous elements $a,b \in A$ which are left linearly independent over $A$.  Then, as in~\cite[Proposition 10.25 and Lemma 9.2]{Lam:modules:rings}, if $C$ denotes the graded centre of $A$, then $A$ contains the (free) algebra over $C$ generated by $a,b$.  This algebra is not a polynomial identity ring and so we have obtained a contradiction to Corollary~\ref{can:ring:cor:6}, since that corollary implies that $A$ is a p.i.~ring.
\end{proof}

\subsubsection{Division rings determined by Cartier divisors}   We now fix a $\KK$-central division algebra $\mathcal{D}$, with degree prime to $\kk$, we let $\Lambda$ be a maximal $\Osh_X$-order in $\mathcal{D}$ and we fix a Cartier divisor $D$ on $X$.  As one consequence to what we describe here, is an alternative point of view for the Iitaka dimension of the Brauer pair $(X,\alpha)$, for $\alpha \in \Br(\KK)$, the Brauer class of $\mathcal{D}$.  The main idea is that, given our fixed Cartier divisor $D$, we can associate to $\Lambda$ a division ring whose centre sees the \emph{growth} of $\Lambda$ with respect to $D$.  To relate the discussion that follows back to the Iitaka dimension of the pair $(X,\alpha)$, the idea is to replace the Cartier divisor $D$ with an integral multiple of $\K_\alpha$, the canonical divisor of $\alpha$; here we assume that $(X,\alpha)$ is $\QQ$-Gorenstein so that some integral multiple of $\K_\alpha$ is indeed Cartier.  

Let 
$ \N(\Lambda,D) := \{ \ell \geq 0 : \H^0(X,\Lambda(\ell D)) \not = 0\}$ denote the \emph{semigroup} of $R(\Lambda,D)$, let $\kk(\Lambda,D)$ be the degree zero division ring of fractions of the graded quotient ring of $R(\Lambda,D)$.  In what follows, we assume that $\N(\Lambda,D) \not = (0)$.  Then, for each $\ell \in \N(\Lambda,D)$, let $\Lambda^{\langle \ell \rangle} \subseteq R(\Lambda,D)$ denote the graded $\kk$-subalgebra generated by the degree $\ell$ part of $R(\Lambda,D)$:
$$ R(\Lambda,D)^{\langle \ell \rangle} = \kk \langle \H^0(X,\Lambda(\ell D) \rangle \subseteq R(\Lambda,D).$$ 

By Proposition~\ref{graded:ore:pi:prop}, the set of nonzero homogeneous elements of $\Lambda^{\langle \ell \rangle}$ satisfy the left Ore conditions
and we let 
$\kk(\Lambda,D)^{\langle \ell \rangle}$ denote the division ring 
obtained by localizing $R(\Lambda,D)^{\langle \ell \rangle}$ at the multiplicative Ore set of nonzero homogeneous elements and then taking degree zero elements.  In particular:
\begin{equation}\label{Division:ell:defn}
 \kk(\Lambda,D)^{\langle \ell \rangle} := \left\{  s^{-1}a : s,a \in R(\Lambda,D)^{\langle \ell \rangle} \text{ are homogeneous, $s \not = 0$, and $\operatorname{deg}(a) = \operatorname{deg}(s)$}\right\}.
\end{equation}
Furthermore, set:
\begin{equation}\label{infinity:div:alg} \kk(\Lambda,D)^{\langle \infty \rangle} := \bigcup_{\ell \in \N(\Lambda,D)} \kk(\Lambda,D)^{\langle \ell \rangle}. \end{equation}

Our goal next is to study the division algebra $\kk(\Lambda,D)$. To that end, we say that a central simple algebra over a field $k$ (and not necessarily with centre $k$)  is \emph{finitely generated} if it is the ring of quotients of a finitely generated $k$-algebra, which will necessarily be p.i.~ and prime.  We use the following lemma.

\begin{lemma}\label{csa:fg}  Let $\Sigma$ be a central simple algebra over a base field $k,$ whose characteristic is prime to the degree of $\Sigma$.  Then $\Sigma$ is finitely generated over $k$ if and only if its centre $Z(\Sigma)$ is a finitely generated field extension of $k$.
\end{lemma}
\begin{proof}  Suppose first that $Z(\Sigma) = k(x_1,\ldots,x_n)$ and let $e_1,\ldots,e_{m^2}$ be a basis of $\Sigma$ over $Z(\Sigma)$.  Let $R = k[x_1,\ldots,x_n]$ and let $S = k \langle x_1,\ldots,x_n,e_1,\ldots,e_{m^2}\rangle.$  Now the fraction field of $R$ is $Z(\Sigma)$ and so the quotient ring of $S$ contains $\sum Z(\Sigma)e_i  =\Sigma$.  So $S$ is a finitely generated $k$-algebra with ring of quotients $\Sigma$.  Now suppose conversely that $\Sigma$ is the ring of quotients of $S$ which is finitely generated over $k$.  Let $T\subseteq Z(\Sigma)$ be the ring of traces of elements of $S$, and we adjoin traces to $S$ to obtain the trace ring $TS$.  As in~\cite[Proposition 13.9.11(ii)]{McConnell-Robson}, which assumes $\cchar k = 0$
but which can also be adapted to our present situation, we see that $Z(TS)$ is a finitely generated $k$-algebra.  Furthermore, if we let $Q(TS)$ and $Q(Z(TS))$ denote, respectively, the ring of quotients of the ring $TS$ and its centre $Z(TS)$, then 
$$Q(Z(TS)) = Z(Q(TS)) = Z(\Sigma)$$ 
by~\cite[Corollary 13.9.7 (i)]{McConnell-Robson} combined with the fact that $TS$ is a finitely generated $T$-module~\cite[Proposition 13.9.11 (ii)]{McConnell-Robson}. 
\end{proof}

We use Lemma \ref{csa:fg} to make the following remark:

\begin{proposition}\label{division:alg:finite:generation}
Let $k$ be a field and let $\Sigma$ be a central simple algebra that is finitely generated over $k$.  Assume that the characteristic of $k$ is prime to the degree of $\Sigma$. 
Then the collection of simple $k$-subalgebras of $\Sigma$ satisfy the ascending chain condition.  In particular, every simple subalgebra of $\Sigma$ is a finitely generated central simple algebra over $k$.
\end{proposition}

\begin{proof}
Let 
\begin{equation}\label{division:alg:finite:generation:eqn1}
\Sigma_1 \subseteq \Sigma_2 \subseteq \hdots \subseteq \Sigma
\end{equation}
be an ascending chain of simple subalgebras of $\Sigma$.  Then, since $\Sigma$ is p.i., each of the $\Sigma_i$ are p.i.~and simple and so central simple algebras by Kaplansky's Theorem.  Inductively, for each $i$, let $\FF_{i+1}$ be a maximal subfield of $\Sigma_{i+1}$ containing $\FF_i$.   In this way, the chain of central simple algebras \eqref{division:alg:finite:generation:eqn1} induces a chain of maximal subfields:
\begin{equation}\label{division:alg:finite:generation:eqn2}
\FF_1 \subseteq \FF_2 \subseteq \hdots \subseteq \FF \subseteq \Sigma,
\end{equation}
for $\FF$ a maximal subfield of $\Sigma$.  By Lemma~\ref{csa:fg} we see that $\FF$ is a finitely generated field extension of $k$. This chain of subfields \eqref{division:alg:finite:generation:eqn2} stabilizes since $\FF$ is finitely generated~\cite[Theorem 24.9]{Isaacs:Algebra}.  Thus the chain \eqref{division:alg:finite:generation:eqn1} reduces to a chain of finite dimensional vector spaces over this maximal subfield and hence stabilizes as well.
\end{proof}

Motivated by Proposition~\ref{division:alg:finite:generation}, we ask:
\begin{question}
Can Proposition~\ref{division:alg:finite:generation} be generalized in some way, for instance, to remove the p.i.~hypothesis?
\end{question}

We now fix $\ell \in \N(\Lambda,D)$,  we note that, for each $k > 0$, we have:
$$
\kk(\Lambda, D)^{\langle \ell \rangle} \subseteq \kk(\Lambda, D)^{\langle k \ell \rangle}
$$
and we consider the behaviour of the ascending chain of division algebras:
$$
\kk(\Lambda,D)^{\langle \ell \rangle} \subseteq \kk(\Lambda,D)^{\langle k_1 \ell \rangle} \subseteq \hdots \subseteq \kk(\Lambda,D)^{\langle k_1 \dots k_p \ell \rangle} \subseteq \kk(\Lambda, D)^{\langle k_1 \hdots k_{p+1} \ell \rangle} \subseteq \hdots \subseteq \kk(\Lambda, D)
$$
for a given collection of positive integers $\{k_i \}$.

\begin{proposition}\label{D:infty:finiteness}  Let $X$ be a normal proper variety over $\kk$, with function field $\KK = \kk(X)$, let $\Lambda$ be a maximal $\Osh_X$-order, in a $\KK$-central division algebra $\mathcal{D}$ having degree prime to $\cchar \kk$, and let $D$ be a Cartier divisor on $X$. 
There exists $0 < \ell_0 \in \N(\Lambda,D)$ so that 
$$\kk(\Lambda, D)^{\langle \infty \rangle} = \kk(\Lambda, D)^{\langle \ell \ell_0 \rangle} = \kk(\Lambda, D)^{\langle \ell_0 \rangle} $$
for all $\ell \geq 1$.
\end{proposition}
\begin{proof}
First note that, by Corollary~\ref{can:ring:cor:6}, $\kk(\Lambda,D)^{\langle  \infty  \rangle}$ is a subdivision algebra of the p.i.~algebra
$\mathcal{D}[t] \subseteq \mathcal{D}(t)$ 
and hence finite dimensional over its centre and finitely generated over $\kk$ 
by Proposition~\ref{division:alg:finite:generation}.  
Next, since  
$$\kk(\Lambda,D)^{\langle \infty \rangle} = \bigcup_{ \ell \in \N(\Lambda,D)} \kk(\Lambda,D)^{\langle \ell \rangle}$$ 
and 
$\kk(\Lambda,D)^{\langle \ell' \rangle} \subseteq \kk(\Lambda,D)^{\langle \ell \rangle}$ 
when $\ell'$ divides $\ell$, it follows 
from Proposition~\ref{division:alg:finite:generation} 
that 
$$\kk(\Lambda, D)^{\langle \infty \rangle} = \kk(\Lambda, D)^{\langle \ell \rangle}$$
for some sufficiently divisible $\ell \in \N(\Lambda,D)$.
\end{proof}

Proposition~\ref{D:infty:finiteness} implies:

\begin{corollary}\label{D:infty:finiteness:cor}
In the setting of Proposition~\ref{D:infty:finiteness}, the degree zero graded fractions of $R(\Lambda,D)$ is
$$\kk(\Lambda,D) = \kk(\Lambda,D)^{\langle \infty \rangle}.$$ 
\end{corollary}
\begin{proof}
Since
$$ \kk(\Lambda,D) = \left\{ s^{-1} a : s,a \in R(\Lambda,D)  \text{ are homogeneous, $s \not = 0$, and $\operatorname{deg}(a) = \operatorname{deg}(s)$}\right\},$$
the equality $\kk(\Lambda,D) = \kk(\Lambda,D)^{\langle \infty \rangle}$ follows from Proposition~\ref{D:infty:finiteness}.
\end{proof}

\subsection{The division ring determined by divisors}\label{division:ring:function:field:centre}  In this section, we fix a maximal $\Osh_X$-order $\Lambda$ in a $\KK$-central division algebra $\mathcal{D}$ having degree prime to $\cchar \kk$.

We now want to compare the division ring 
$\kk(\Lambda,D)$ with the degree zero field of fractions  $\kk(X,D)$ of the
section ring $R(X,D)$ of the Cartier divisor $D$.

\begin{lemma}\label{section:ring:centre:lemma} With the same assumptions as Proposition~\ref{D:infty:finiteness}, the field $\kk(X,D)$ is a central subfield of $\kk(\Lambda,D)$.
\end{lemma}
\begin{proof}  
By Corollary~\ref{D:infty:finiteness:cor}, we have
$$\kk(\Lambda,D)^{\langle \infty \rangle} = \kk(\Lambda,D)$$
and so Lemma~\ref{section:ring:centre:lemma} is an immediate consequence of Remark~\ref{can:ring:remark:4}.
\end{proof}

We now study the nature of the extension $\kk(\Lambda,D) / \kk(X,D)$.

\begin{proposition}\label{alg:extension:prop} Let $X$ be a normal proper variety over $\kk$, with function field $\KK = \kk(X)$, let $\Lambda$ be a maximal $\Osh_X$-order, in a $\KK$-central division ring $\mathcal{D}$ having degree prime to $\cchar \kk$, and let $D$ be a $\QQ$-Cartier divisor on $X$ with $\N(X,D) \not = (0)$.   Then the division algebra $\kk(\Lambda,D)$ is finite dimensional over its centre, which is a finite extension of $\kk(X,D)$.
\end{proposition}

\begin{proof}
By linear equivalence we may assume that $D$ is effective.  Choose $\ell > 0$ so that $\ell D$ is an integral Cartier divisor and let $t \in \H^0(X,\Osh_X(\ell D))$ be a section corresponding to $\ell D$.   Now let $s \in \H^0(X,\Lambda(\ell D))$ and note that $t^{-1}s \in \kk(\Lambda,D)$.  Consider the field $\FF = \KK (t^{-1}s)\subseteq \mathcal{D}$ obtained by adjoining $t^{-1}s$ to $\KK$.  Since $\mathcal{D}$ is finite dimensional over $\KK$, $\FF$ is finite dimensional over $\KK$ too.  Let $Y$ be the normalization of $X$ in $\FF$ with canonical map $p\colon Y \to X$.
Fix an affine open cover $\{U_i\}$ of $X$ which trivializes $\Osh_X(\ell D)$.  Then
  $$t_{|_{U_i}}\Osh_X(U_i) = \Osh_X(\ell D)(U_i) = \{ f \in \KK : \mathrm{div}(f)|_{U_i} +\ell D|_{U_i} \geq 0\}.$$
  We also have that
  $$s_{|_{U_i}} \in \Lambda(\ell D)(U_i) = \Lambda(U_i) t_{|_{U_i}}.$$
Note that $$(t^{-1}s)_{|_{U_i}} = t_{|_{U_i}}^{-1} s_{|_{U_i}}  \in \Lambda(U_i) \cap \FF.$$ So
  $(t^{-1}s)_{|_{U_i}}$ is an element of $\FF$ that is integral over $\Osh_X(U_i)$ so
  $(t^{-1}s)_{|_{U_i}} \in \Osh_Y(U_i)$.  Now we get
  $$s_{|_{U_i}} \in   t_{|_{U_i}} \Osh_Y(U_i) =\Osh_Y(\ell D)(U_i).$$
In particular, 
$$s \in \H^0(X,p_*\Osh_Y\otimes\Osh_X(\ell D)))= \H^0(X,p_*(\Osh_Y(p^*(\ell D)))= \H^0(Y,\Osh_Y(p^*(\ell D))).$$  
The proof of~\cite[Proposition 1.5 (iii)]{Mori:1985} now shows that $R(Y,p^*D)$ is integral over $R(X,D)$.  Since $s$ was arbitrary every element of $R(\Lambda,D)$ is integral over $R(X,D)$.    Now the argument of~\cite[Lemma 1.2 (iii)]{Mori:1985} shows that every element of $\kk(\Lambda,D)$ is algebraic over $\kk(X,D)$.

It remains to show that the extension $\kk(\Lambda,D) / \kk(X,D)$ is finite.  We know that $\kk(\Lambda,D)$ is a subdivision ring  of $\mathcal{D}$.  So by Proposition~\ref{division:alg:finite:generation}, we have that $\kk(\Lambda,D)$ is finite dimensional over its centre $Z(\kk(\Lambda,D))$ which is finitely generated over $\kk$. So $Z(\kk(\Lambda,D))$ is  finitely generated over $\kk(X,D)$.  Finally, we note that $Z(\kk(\Lambda,D))$ is algebraic over $\kk(X,D)$ and finitely generated over $\kk(X,D)$.  Thus $\kk(\Lambda,D)$ is finite over $\kk(X,D)$.
\end{proof}

Note that if a division algebra $\mathcal{D}$ is finite dimensional over its centre $Z(\mathcal{D})$, then we define 
$$\trdeg \mathcal{D} := \trdeg Z(\mathcal{D}).$$
Proposition~\ref{alg:extension:prop} immediately implies:

\begin{corollary}\label{transcendence:degree:equality} With the same assumptions as Proposition~\ref{alg:extension:prop}, the division algebras $\kk(\Lambda,D)$ and $\kk(X,D)$ have the same transcendence degree.
\end{corollary}

We now consider the case that $D = \K_\alpha = \K_X + \Delta_\alpha$, for $\Delta_\alpha$ the ramification divisor of a maximal order $\Lambda \subseteq \mathcal{D}$.
Here we assume that $\K_\alpha \in \Div_\QQ(X)$; equivalently we assume that the pair $(X,\Delta_\alpha)$ is $\QQ$-Gorenstein.
We then have, as a consequence of Corollary~\ref{transcendence:degree:equality}, the following result.

\begin{theorem}\label{div:alg:transcendence:thm}  Let $X$ be a normal proper variety over $\kk$, with function field $\KK = \kk(X)$, and fix a Brauer class $\alpha \in \operatorname{Br}(\KK)$ having degree prime to $\cchar \kk$.  Let $\Lambda$ be a maximal order in a $\KK$-central division algebra $\mathcal{D}$ with Brauer class equal to $\alpha$.  If the pair $(X,\alpha)$ is $\QQ$-Gorenstein and if $\kappa(X,\K_\alpha) \geq 0$, then the transcendence degree of $\kk(\Lambda,\K_\alpha)$ over $\kk$ equals the Iitaka dimension of the pair $(X,\alpha)$.  Explicitly:
$$ \trdeg \kk(\Lambda,\K_\alpha) = \kappa(X,\alpha) = \kappa(X, \K_\alpha).$$
\end{theorem}
\begin{proof}
It is well known that $\trdeg \kk(X,\K_\alpha) = \kappa(X,\K_\alpha),$
see for example~\cite[Corollary 2.1.37]{Laz} and~\cite[Definition 2.1.3]{Laz}.
So Theorem \ref{div:alg:transcendence:thm} follows from Corollary~\ref{transcendence:degree:equality} because, in the notation of that result, we have: 
$$\trdeg \kk(X,\K_\alpha) = \kappa(X,\K_\alpha) = \kappa(X,\alpha).$$
\end{proof}

As an additional remark, fix a $\KK$-central division algebra $\mathcal{D}$, with finite dimension over $\KK$ and having degree prime to $\cchar \kk$, fix a maximal order $\Lambda$ in $\mathcal{D}$ and fix a Cartier divisor $D$ on $X$.  We then note that we can use Corollary~\ref{transcendence:degree:equality}, in conjunction with results from~\cite{Zhang:1998}, 
to describe the behaviour of the lower transcendence degree of the algebra $\kk(\Lambda,D)$ compared to that of $R(X,D)$ the section ring of $D$.  This is the content of Theorem~\ref{Ld:maximal:order:section:ring:thm}.  Before stating this result, if $A$ is an algebra over $\kk$, then we denote by $\GKdim A$ its GK dimension and by $\operatorname{Ld} A$ its lower transcendence degree.  For completeness, we briefly recall, following~\cite{Zhang:1998}, that the GK dimension of $A$ is defined to be:
$$
\GKdim A := \sup_V \limsup_{n\to\infty} \log_n \dim V^n,
$$ 
where $V$ ranges over all finite dimensional subspaces of $A$.  Also, if $A$ is assumed to be a prime p.i.~algebra, then the lower transcendence degree of $A$ can be described as:
$$
\operatorname{Ld} A := \sup_V \inf_b \limsup_{n\to \infty} \log_n \dim(\kk + b V)^n
$$
where $V$ ranges over the finite dimensional subspaces of $A$ and where $b$ ranges over the regular elements of $A$.  Here, in the above definitions, $V^n$ and $(\kk + b V)^n$ denote the $\kk$-subspaces of $A$ generated by their respective $n$-fold products.

\begin{lemma}\label{easyLemma}
Let $M_{A}$ be a right $A$ module such that $M_{A} \supseteq A_{A}$.  Then $\GKdim M_{A} = \GKdim A$.
\end{lemma}
\begin{proof}
Proposition 5.1 (d) of \cite{Krause:Lenagan} states that  
$$\GKdim M_{A} \leq \GKdim A.$$  
Since $M_{A} \supseteq A_{A}$ we also have 
$$\GKdim M_{A} \geq \GKdim A_{A} = \GKdim A$$ 
by \cite[Proposition 5.1 (b)]{Krause:Lenagan}.
\end{proof}

Using Lemma \ref{easyLemma}, we also observe:

\begin{proposition}\label{gk:veronese:prop}
Let $A$ be an $\mathbb{N}$-graded $k$-algebra such that $A_0=k$.  Then $\GKdim A = \GKdim A^{(n)}$.
\end{proposition}
\begin{proof}
Let $V\subset A$ be a finite dimensional graded vector space in $A$ such that $V \cap A_0 = 0.$  Let $k\langle V \rangle \subseteq A$ be the subalgebra of $A$ generated by $V$.   Write $L(V) = k\langle V \rangle A^{(n)}$ for the right $A^{(n)}$
submodule of $A_{A^{(n)}}$ generated by $k\langle V \rangle$.  We note that $L(V)$ is a $k\langle V \rangle - A^{(n)}$ bimodule.  Since $1 \in k\langle V\rangle \cap A^{(n)}$, we see that: 
$$A^{(n)} \subseteq L(V)$$ 
and 
$$k\langle V \rangle \subseteq L(V).$$  
Furthermore, $L(V)$ is a finitely generated right $A^{(n)}$-module, since all products of elements of $V$ of length $n$ are in $A^{(n)}$.
So now we have that:
$$ \GKdim k\langle V \rangle  =   \GKdim \  _{k\langle V \rangle}L(V) \leq  \GKdim L(V)_{A^{(n)}}  =  \GKdim A^{(n)}.
$$
Note that the equalities above follow from Lemma~\ref{easyLemma} and the inequality above is \cite[Lemma 5.3 (b)]{Krause:Lenagan}.
So now we have that $\GKdim k\langle V \rangle \leq \GKdim A^{(n)}$ for all $V$ and so we get that $\GKdim A = \GKdim A^{(n)}$.
\end{proof}

 We now establish the following result.

\begin{theorem}\label{Ld:maximal:order:section:ring:thm} Let $\KK$ be a finitely generated field over $\kk$.  Let $\mathcal{D}$ be a $\KK$-central division algebra  with finite dimension over $\KK$ and having degree prime to $\cchar \kk$, choose a maximal $\Osh_X$-order $\Lambda$ in $\mathcal{D}$ over a proper normal 
 variety $X$, and fix a Cartier divisor $D$ on $X$ with $\N(X,D) \not = (0)$.  Let $\kk(\Lambda,D)$ be the division ring of degree zero fractions of $R(\Lambda,D)$ and let $\kk(X,D)$ be the degree zero field of fractions of the section ring $R(X,D)$.
Then, with these notations and hypothesis, it holds true that:
$$ \GKdim R(D) - 1 = \operatorname{Ld} R(D) - 1 = \trdeg \kk(X,D) = \trdeg \kk(\Lambda,D) = \operatorname{Ld} \kk(\Lambda,D). $$
\end{theorem}
\begin{proof}[Proof of Theorem~\ref{Ld:maximal:order:section:ring:thm} and Theorem~\ref{Ld:maximal:order:section:ring:thm:intro}]
By~\cite[Proposition 0.2]{Zhang:1998} 
and~\cite[Proposition 2.1]{Zhang:1998}, 
it follows that:
$$\GKdim R(D) - 1 = \operatorname{Ld} R(D) - 1 = \trdeg \kk(X,D). $$
On the other hand, combining Corollary~\ref{transcendence:degree:equality},~\cite[Corollary 3.3]{Zhang:1998} and statement (1) of~\cite[Theorem 0.3]{Zhang:1998}, we obtain:
$$\trdeg \kk(X,D) = \trdeg \kk(\Lambda, D) = \operatorname{Ld} \kk(\Lambda, D)  $$
and so Theorem~\ref{Ld:maximal:order:section:ring:thm}  follows.   Finally, the final assertion about birational invariance in Theorem~\ref{Ld:maximal:order:section:ring:thm:intro} follows from Theorem~\ref{brauer:iitakacanonical:singulatiries:main:result}.
\end{proof}

\subsection{Perturbed growth conditions}\label{perturbed:growth}  In this section, we establish Proposition~\ref{perturbed:growth:prop1} which we then use to prove Theorem~\ref{log:abundance:growth}.  To begin with, if $D$ is a $\QQ$-Cartier divisor on a normal projective variety $X$ over $\kk$, then we let $\kappa_\sigma(D)$ denote its \emph{perturbed growth}, or \emph{numerical dimension}:
\begin{equation}\label{perturbed:growth:eqn:defn}
\kappa_\sigma(D) := \max \{\sigma(D;A):  \text{$A$ is an ample divisor on $X$} \}.
\end{equation}
Here, in \eqref{perturbed:growth:eqn:defn}, if $F$ is a reflexive sheaf on $X$, then, as in~\cite[Definition 2.1]{Fujita81}, we define: 
\begin{equation}\label{perturbed:growth:eqn:defn2}
\sigma(D;F) := \max \left\{ k \in \ZZ_{\geq 0} : \limsup_{m \to \infty} \frac{h^0(X,F \otimes \Osh_X(\lfloor m D \rfloor)) }{m^k} > 0 \right\}
\end{equation}
and we put $\sigma(D;F) = - 1$ in case that $h^0(X,F\otimes\Osh_X(\lfloor m D \rfloor)) = 0$ for all sufficiently divisible $m \gg 0$.

\begin{proposition}\label{perturbed:growth:prop1}
Suppose that $X$ is a normal projective variety over $\kk$.
If $F$ is a reflexive sheaf on $X$ and if $D$ is a $\QQ$-Cartier divisor on $X$, then 
$$ \sigma(D;F) \leq \kappa_\sigma(D).$$
\end{proposition}
\begin{proof}
Fix an ample divisor $A$ on $X$ and choose $\ell \gg 0$ so that the coherent sheaf $F^\vee \otimes \Osh_X(\ell A)$ is globally generated.  Then, by taking the dual of the surjective morphism of $\Osh_X$-modules:
\begin{equation}\label{perturbed:growth:eqn1}
\bigoplus_{\mathrm{finite}} \Osh_X(-\ell A) \rightarrow F^\vee \rightarrow 0
\end{equation}
we obtain an injective morphism of $\Osh_X$-modules:
\begin{equation}\label{perturbed:growth:eqn2}
0 \rightarrow F = F^{\vee \vee} \hookrightarrow \bigoplus_{\mathrm{finite}} \Osh_X(\ell A).
\end{equation}
Furthermore, we can twist \eqref{perturbed:growth:eqn1} and \eqref{perturbed:growth:eqn2} by $mD$ for each sufficiently divisible $m \gg 0$.  We then obtain, combining \eqref{perturbed:growth:eqn1} and \eqref{perturbed:growth:eqn2}, 
\begin{equation}\label{perturbed:growth:eqn3}
h^0(X,F\otimes \Osh_X( mD )) \leq h^0(X,F^\vee \otimes \Osh_X(\ell A)) h^0(X,\Osh_X( mD  + \ell A)),
\end{equation}
for each sufficiently divisible $m \gg 0$.   
Clearly \eqref{perturbed:growth:eqn3} implies that 
$\sigma(D;F) \leq \sigma(D;\ell A) \leq \kappa_\sigma(D)$ 
as desired.
\end{proof}

Using Proposition~\ref{perturbed:growth:prop1} together with Conjecture~\ref{perturbed:growth:abundance:conj}, we  now prove Theorem~\ref{log:abundance:growth}.

\begin{proof}[Proof of Theorem~\ref{log:abundance:growth}]
Let $n^2 =\rank \Lambda$.  By Proposition~\ref{gk:veronese:prop} and Corollary~\ref{canonical:ring:order:cor4} we see that
\begin{align}\label{log:abundance:growth:eqn1}
\begin{split}
      \GKdim R(\Lambda,\omega_\Lambda) & 
      = \GKdim R(\Lambda,\omega_\Lambda)^{(n)} \\
       & 
      =   \GKdim R(\Lambda,(\omega_\Lambda^{\otimes n})^{\vee\vee}) 
      \\
& 
= \GKdim \bigoplus_{\ell \geq 0} \H^0(X,\Lambda(\ell n (\K_X+\Delta_\Lambda))) 
\\
& 
= \sigma(\K_\Lambda;\Lambda) - 1. 
\end{split}
\end{align}
On the other hand, since $\Osh_X \hookrightarrow \Lambda$, we deduce:
\begin{equation}\label{log:abundance:growth:eqn1'}
\kappa(X,\K_\Lambda) \leq \sigma(\K_\Lambda;\Lambda) \leq \kappa_\sigma(\K_\Lambda) = \kappa(X,\K_\Lambda). 
\end{equation}
Here the rightmost inequality in \eqref{log:abundance:growth:eqn1'} follows from Proposition~\ref{perturbed:growth:prop1} while the rightmost equality is implied by \eqref{num:abundance:growth:eqn}.  The conclusion desired by Theorem~\ref{log:abundance:growth} then follows by combining \eqref{log:abundance:growth:eqn1} and \eqref{log:abundance:growth:eqn1'}.
\end{proof}

\providecommand{\bysame}{\leavevmode\hbox to3em{\hrulefill}\thinspace}
\providecommand{\MR}{\relax\ifhmode\unskip\space\fi MR }
\providecommand{\MRhref}[2]{%
  \href{http://www.ams.org/mathscinet-getitem?mr=#1}{#2}
}
\providecommand{\href}[2]{#2}

\end{document}